\documentclass{amsart}

\usepackage{cliff-arxiv-style}

\begin{document}

\title[Universal D-modules and stacks of \'etale germs]{Universal $\CD$-modules and stacks of \'etale germs of $n$-dimensional varieties}
\author{Emily Cliff}
\address{Department of Mathematics, University of Illinois at Urbana--Champaign\\
Altgeld Hall, 1409 W Green Street\\
Urbana, Illinois, USA 61801}
\email{ecliff@illinois.edu}

\begin{abstract}
We introduce stacks classifying \'etale germs of pointed varieties of dimension $n$. We show that quasi-coherent sheaves on these stacks are universal $\mathcal{D}$- and $\mathcal{O}$-modules. We state and prove a relative version of Artin's approximation theorem, and as a consequence identify our stacks with classifying stacks of automorphism groups of the $n$-dimensional formal disc. We introduce the notion of convergent universal modules, and study them in terms of these stacks and the representation theory of the automorphism groups.


\end{abstract}

%\subjclass[2010]{Primary 14D23, 14F10. Secondary 14B12.}

\maketitle

\tableofcontents

\section{Introduction}
\label{sec: introduction}

Vertex algebras and vertex operator algebras have been studied and applied fruitfully in a number of areas, ranging from physical disciplines such as conformal field theory and string theory to finite group theory and the geometric Langlands correspondence. Beilinson and Drinfeld \cite{BD1} reformulated the axioms of a vertex algebra in geometric language in terms of chiral algebras, and showed that these are equivalent to factorisation algebras---both geometric objects which take the form of $\CD$-modules over a complex curve, equipped with additional structure. The sophisticated machinery of factorisation and chiral algebras elegantly captures the data of vertex algebras in an often more intuitive way---for example, Francis and Gaitsgory \cite{FG} showed that Beilinson--Drinfeld's definitions can be extended to higher dimensions in a natural way, whereas the vertex algebra picture is much less clear (see for example Borcherds \cite{B2}). 

In the one-dimensional setting, Frenkel and Ben-Zvi \cite{FBZ2} explain the relationship between vertex algebras and chiral algebras over curves. To make this relationship precise, we need adjectives on both sides. First, we require our vertex algebras to be \emph{quasi-conformal}, or equipped with a one-dimensional infinitesimal translation. On the other hand, the chiral algebras we obtain are \emph{universal}: they are defined over all smooth families of curves, and are compatible with pullback by \'etale morphisms of these families. Roughly, the infinitesimal translation allows us to spread the vector space underlying the vertex algebra canonically along any complex curve $C$. In this way, we obtain a $\CD$-module on $C$ which will have the structure of a chiral algebra. The fact that this procedure works for any smooth curve $C$ means that we obtain a universal chiral algebra.

Motivated by this fact, in this paper we study categories of universal $\CD$-modules and universal $\CO$-modules of dimension $n$. Inspired by a claim of Beilinson and Drinfeld \cite{BD1}, we relate these categories to categories of representations of groups of automorphisms of the $n$-dimensional formal disc. A universal $\CD$-module is a rule assigning to each smooth $n$-dimensional variety a $\CD$-module in a way compatible with pullback by \'etale morphisms between the varieties. A key observation is that all of this data is equivalent to the data of a single sheaf on a stack parametrising \'etale germs of $n$-dimensional varieties. 

A second critical observation is that, using a generalisation of Artin's ap\-prox\-i\-ma\-tion theorem to the relative setting, we can relate this stack to the clas\-si\-fy\-ing stack of the group $G$ of automorphisms of the $n$-dimensional for\-mal disc. More precisely, in the case of universal $\CD$-modules, our stack is equiv\-a\-lent to the classifying stack of the group $G^\et$ of automorphisms \emph{of \'etale type}, which is a dense subgroup of $G$. It follows that a universal $\CD$-module is equivalent to a representation of $G^\et$; furthermore, any rep\-re\-sen\-ta\-tion of $G$ restricts to give a representation of $G^\et$ and hence a universal $\CD$-module. 

The difference between $\CD$-modules and $\CO$-modules amounts to an action by in\-fi\-ni\-tes\-i\-mal translations, present only in the case of $\CD$-modules; in the case of the stacks in this paper, this difference is manifested in the automorphism groups as follows: the group corresponding to $\CO$-modules contains only those automorphisms of the formal disc preserving the origin, while in the case of $\CD$-modules infinitesimal translations of the origin are permitted.  

A natural question to ask is whether we can characterise those universal $\CD$-modules which come from representations of $G$ rather than representations of just the subgroup $G^\et$. We give two characterisations of these universal $\CD$-modules, which we call \emph{convergent universal $\CD$-modules}. However, at the time of writing, we do not know whether all universal $\CD$-modules are actually convergent. This question is equivalent to the question of whether all representations of $G^\et$ extend uniquely to representations of $G$. We are able to show that if an extension exists, it is unique. Furthermore, any finite-dimensional representation of $G^\et$ extends to a representation of $G$, and in fact this is true of any representation of $G^\et$ satisfying a weaker finiteness condition which we call being \emph{$K^\et$-locally-finite}. All representations of $G$ satisfy this condition, and so the question is reduced to the existence of representations of $G^\et$ which are not $K^\et$-locally-finite. If such a representation exists, it will give rise to a universal $\CD$-module which is not convergent; on the other hand, it seems that such an object would be unlikely to arise in ordinary applications of the theory and would be unpleasant to work with. In other words, we are only interested in working with $\CD$-modules satisfying the properties implied by convergence, and we suggest that the category of convergent universal $\CD$-modules is the correct category in which to work.  

An intended application of this theory is the following. As in Francis--Gaitsgory \cite{FG}, we know that chiral algebras over a variety $X$ are certain Lie algebra objects in the category of $\CD$-modules on the Ran space of $X$, which is equipped with a monoidal structure called the \emph{chiral} monoidal structure. In the universal setting, we can introduce Ran versions of the stack of \'etale germs and the automorphism groups $G$ and $G^\et$, and define chiral monoidal structures on the associated categories of quasi-coherent sheaves. We should then interpret the monoidal structure in terms of the classifying stack in order to obtain higher-dimensional analogues of vertex algebras as Lie algebra objects in the representation category with this chiral monoidal structure. This will be explored in future work. 

\setcounter{subsection}{0}
\subsection{The key players}
\label{subsec: the key players}
Let us now introduce the main players of this paper, which fall into three classes: we have two flavours of stacks---those corresponding to classifying stacks, and those corresponding to stacks of germs of varieties---and in addition, we have the categories of universal modules. 

Let $G$ denote the group formal scheme of automorphisms of the formal disc. It has a pro-structure, since it can be viewed as the limit of its quotients $G^{(c)}$, which are (in an imprecise sense) the automorphism groups of the $c$th infinitesimal neighbourhood of a point in an $n$-dimensional variety. The classifying stacks of interest will be those corresponding to these groups; quasi-coherent sheaves on these classifying stacks correspond to representations of the associated group. There is a subgroup $G^\et$ of $G$, of automorphisms of \emph{\'etale type}; it is closely related to the stacks of \'etale germs that we will define later. We will see that this subgroup is dense in $G$, so that the representation theory of the two groups is very similar. More specifically, placing a finiteness condition on their representations yields equivalent categories of representations.

The second flavour of stacks are those parametrising \'etale germs of $n$-di\-men\-sion\-al varieties---that is, we are interested in pointed $n$-dimensional varieties with morphisms given by roofs of \'etale morphisms, or \emph{common \'etale neighbourhoods}:
\begin{center}
\begin{tikzpicture}[>=angle 90]
\matrix(c)[matrix of math nodes, row sep=2em, column sep=2em, text height=1.5ex, text depth=0.25ex]
{    &  (V,v)  & \\
 (X_1,x_1) && (X_2,x_2). \\};
 
\path[->, font=\scriptsize]
(c-1-2) edge (c-2-1)
        edge (c-2-3);
\end{tikzpicture}
\end{center}
Imposing different equivalence relations on these classes of morphisms allows us to define the different versions of the stack that we will need, corresponding to the various classifying stacks mentioned above. The equivalence relations are defined by identifying common \'etale neighbourhoods which give rise to the same isomorphisms of the formal completions $\widehat{X_1} \EquivTo \widehat{X_2}$ or of the $c$th infinitesimal neighbourhoods for $c \in \BN$. We denote these stacks by $\varieties$, $\varietiesc{c}$, and $\varietiesc{\infty}$.

Finally, we consider the category $\univcatD$ of universal $\CD$-modules, as introduced by Beilinson and Drinfeld \cite{BD1}. These are families of $\CD$-modules on $n$-dimensional varieties, compatible with pullback along \'etale morphisms. We impose an additional requirement, that these compatibilities be themselves compatible with identifications of \'etale morphisms giving rise to the same morphisms of infinitesimal neighbourhoods. This allows us to define subcategories $\univcatDc{c}$, and finally the category $\univcatDconv$ of \emph{convergent universal $\CD$-modules}, which are the kind of modules arising from vertex algebras. 

Extending the finiteness condition on representations alluded to above allows us to define analogous conditions on the categories of quasi-coherent sheaves, and finally to give a characterisation of convergent universal $\CD$-modules as those universal $\CD$-modules which are of \emph{ind-finite type}.

The relationship between these objects is the main focus of this paper; the results are summarised in Figure 1.\footnote{The reader may wish to pull out the additional copy of the diagram included in Appendix \ref{appendix: main diagram} so that he can refer back to it easily.}
 
\begin{figure}
\caption{The main diagram: this paper in one page.}
\centering
\includegraphics[angle=90,origin=c, height=4.8in]{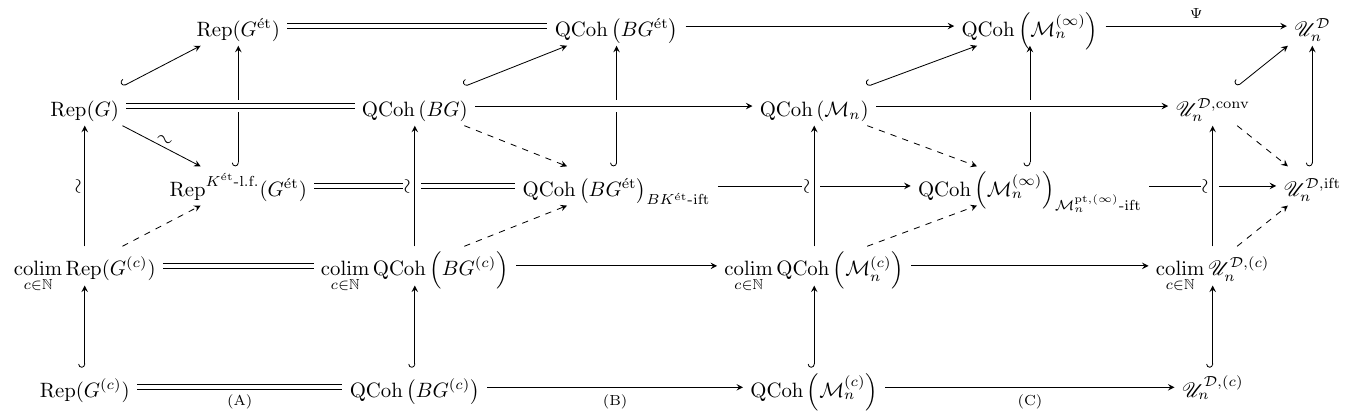}
\end{figure}

This diagram corresponds to the setting of universal $\CD$-modules; we have an entirely analogous diagram to describe the setting of universal $\CO$-modules. The key differences are the following:
\begin{enumerate}
\item We replace the group $G$ of automorphisms of the formal disc by its reduced subgroup $K= G_\red$: this is the group of automorphisms which preserve the origin of the disc. By contrast, $G$ includes automorphisms which may involve an infinitesimal translation of the origin. These infinitesimal translations correspond to the action of the sheaves of differential operators on the corresponding universal $\CD$-modules, not present in the case of universal $\CO$-modules.
\item We replace the stacks of \'etale germs of $n$-dimensional varieties by stacks with the same objects but with fewer isomorphisms: in this case we only allow isomorphisms which fix the distinguished points of our pointed varieties, whereas the original stacks permitted isomorphisms which could shift these points infinitesimally. We will denote these stacks by adding a superscript $\bullet^\pt$ to the symbol denoting the corresponding stack for the $\CD$-module setting (e.g. $\pointedvarieties$).
\item Rather than considering universal families of $\CD$-modules, we consider universal families of $\CO$-modules. We denote these categories by $\univcat$, etc.
\item The finiteness conditions in the bottom part of the back row are simpler. In fact, these finiteness conditions are most naturally defined in the $\CO$-module setting; the corresponding conditions in the $\CD$-module setting are then defined by requiring the objects to be suitably finite when regarded as objects in the $\CO$-module setting after applying a forgetful functor.
\end{enumerate}

We will give the full definitions of the stacks and categories for both the $\CD$- and $\CO$- module settings, but for the proofs of the equivalences we will mainly focus on the story of universal $\CD$-modules. This is the setting needed for working with universal chiral algebras. Moreover, generally the proofs in the $\CO$-module setting are just simpler versions of the $\CD$-module proofs. The exception is in the study of the categories of representations; there we will see that it is first necessary to study the groups $K$ and $K^\et$, and then to extend our results to $G$ and $G^\et$. 

\subsection{The structure of the paper}
\label{subsec: structure of the paper}
We will begin in Section \ref{sec: stacks of etale germs} by discussing and defining the stacks $\varietiesc{\infty}$ and $\pointedvarietiesc{\infty}$ of unpointed and pointed \'etale germs, as well as the necessary variations. In Section \ref{sec: classifying stacks} we introduce the automorphism groups $G$ and $K$, as well as their quotients $G^{(c)}$ and $K^{(c)}$. We will see that there is a natural map $F$ from our stack $\varietiesc{\infty}$ to the classifying stack $BG$, and analogues $F^\cth$ for the quotients $G^\cth$. Pullback along these maps gives rise to the functors in column (B) of the main diagram. In Section \ref{sec: relative artin approximation}, we state and prove a generalisation of Artin's approximation theorem \cite{A} to the relative setting, and show as a corollary that the map $F^\cth$ is an isomorphism of stacks. It follows that the functors in the front part of column (B) are equivalences. 

In Section \ref{sec: groups of etale automorphisms}, we introduce the group-valued prestack $G^\et$ of \'etale-type automorphisms of the formal disc; it will be immediate from the relative Artin approximation theorem and the definition of $G^\et$ that the classifying stack $BG^\et$ is equivalent to the stack $\varietiesc{\infty}$ of \'etale germs of $n$-dimensional varieties. We will then study the representation theory of $G^\et$, and identify representations of $G$ as the subcategory of $\Rep(G^\et)$ of \emph{$K^\et$-locally-finite} representations. In Section \ref{sec: universal modules} we define the categories of universal $\CD$- and $\CO$-modules, as well as the functor $\Psi$ in the back part of column (C). We then prove that this functor is an equivalence. 

In Section \ref{sec: convergent and ind-finite universal modules}, we define the categories $\univcatDc{c}$ and $\univcatDconv$ of $c$th-order and convergent universal $\CD$-modules, and we prove that the functor $\Psi$ restricts to give the equivalences of the front part of column (C). We also characterise convergent universal $\CD$-modules as those universal $\CD$-modules which are \emph{locally finite} in the sense analogous to that in the study of $\Rep(G) \emb \Rep(G^\et)$---we shall call these universal $\CD$-modules of \emph{ind-finite type} to avoid confusion with the standard use of the word ``local'' in sheaf theory. In the final section, we discuss the extension of these definitions and results to the setting of $\infty$-categories. 

Combining our results, we obtain the following equivalences of categories:
\begin{center}
\begin{small}
\begin{tikzpicture}[>=angle 90,bij/.style={above,sloped,inner sep=0.5pt}]
\matrix(a)[matrix of math nodes, row sep=1em, column sep=1.3em, text height=1.5ex, text depth=0.25ex]
{       & \QCoh{\pointedvarieties}&          &         & \QCoh{\varieties}  & \\
 \Rep(K) &                        & \univcatOconv & \Rep(G) &           & \univcatDconv. \\};
\path[->]
 (a-1-2) edge node[bij]{$\sim$} (a-2-1)
         edge node[bij]{$\sim$} (a-2-3)
 (a-1-5) edge node[bij]{$\sim$} (a-2-4)
         edge node[bij]{$\sim$} (a-2-6);
\end{tikzpicture}
\end{small}
\end{center}

That is, we have proved a variation of the theorem suggested by Beilinson and Drinfeld [Proposition and Exercise 2.9.9, \cite{BD1}]:
\begin{thm}
We have the following equivalences of categories:
\begin{align*}
\univcatOconv \simeq \Rep(K)\\
\univcatDconv \simeq \Rep(G).
\end{align*}
\end{thm}

Composing the functors in the main diagram, we obtain a functor
\begin{align*}
\Rep(G) \to \univcatD.
\end{align*}
This functor agrees with the functor (2.9.9.1) of \cite{BD1}. It follows from our results that this functor is a fully faithful embedding, but it is not clear that it is essentially surjective. If it is not, then not all universal $\CD$-modules are convergent; we argue that in that case the category of convergent universal $\CD$-modules should be the preferred setting.  

Note that the statement of Proposition 2.9.9 \cite{BD1} is also discussed in a more restricted setting by Jordan and Orem (see Section 4 of \cite{JO}).

\subsection{Conventions and notation}
We fix $k = \overline{k}$, an algebraically closed field of characteristic zero. By $\Sch$, we will always mean the category of schemes over $k$. By $\PreStk$, we mean the category of functors
\begin{align*}
\left(\Sch^{\Aff}\right)^\op \to \infty\text{-}\Grpd.
\end{align*}

However, the examples of prestacks discussed in this paper take values in $\Grpd \subset \infty\text{-}\Grpd$. In fact, throughout most of the paper, we work with ordinary categories of sheaves, $\CD$-modules, and representations, rather than DG- or $\infty$-categories. Our categories are cocomplete (i.e. closed under colimits); in particular, colimits of categories are always taken as cocomplete categories. 

In the final section, we will show that our results extend to the setting of DG-categories after all, when we choose appropriate definitions for DG categories of representations, universal modules, etc. Even though this is the setting in which these results are most likely to find applications, we have chosen to present the arguments first for the non-infinity cases for concreteness (in particular when working with representations of groups and group-valued prestacks). 

Note that each of the stacks appearing in the main diagram can be defined by giving a prestack parametrising only the \emph{trivial} objects; then we take the stackification of the prestack to obtain the stacks in our diagram. (For example, the classifying stack of a group is the stackification of the prestack classifying only the trivial principal bundles; similarly in the third column of the diagram we can consider prestacks classifying the ``trivial'' pointed $n$-dimensional variety, $(\A{n}, 0)$.) These prestacks will be denoted by adding the subscript $\bullet_\triv$ to the symbol for the corresponding stack. In the case of the stacks of germs of varieties, it will also be convenient to consider an intermediate prestack, which has more objects and automorphisms than the trivial version of the prestack, but which still has the same stackification; this prestack will be denoted by decorating the symbol for the corresponding stack with $\widetilde{\bullet}$. Since taking quasi-coherent sheaves is independent of stackification, it does not matter which of these related prestacks we use in the above diagram.

\subsection{Acknowledgements}
I owe many thanks to Dominic Joyce and Dennis Gaitsgory for pointing out an error in an earlier version of this work, and to Pavel Safronov and David Ben-Zvi for other helpful comments on early drafts and discussions about these ideas. I thank Sasha Beilinson for a discussion pointing me in the direction of the ideas of Section \ref{sec: groups of etale automorphisms}, and Minhyong Kim for useful conversations about Artin's approximation theorem, and about representations of pro-unipotent groups. I am especially grateful to Kevin McGerty for many helpful discussions, but most of all to my advisor, Kobi Kremnizer, for his guidance, patience, and optimism. Finally, I thank the referee for carefully reading the paper and providing detailed comments. 

This work was supported by the Natural Sciences and Engineering Research Council of Canada, and by a scholarship from the Winston Churchill Society of Edmonton, Canada. 
 
\section{Stacks of \'etale germs}
\label{sec: stacks of etale germs}

We fix a natural number $n$, and are interested in studying smooth pointed varieties of dimension $n$ up to \'etale morphism. We will define the stack $\varietiesc{\infty}$ classifying families of such varieties---in fact, we will first introduce the prestack $\prestackVc{\infty}$ classifying \emph{trivial} $n$-dimensional pointed families, and then will define $\varietiesc{\infty}$ to be its stackification. We will also introduce an intermediate prestack $\prestackVbc{\infty}$, which has $\varietiesc{\infty}$ as its stackification as well, but which is somewhat more manageable. 

We begin in \ref{subsec: families of pointed varieties} with some preliminary definitions on $n$-dimensional families of varieties and common \'etale neighbourhoods between them. In \ref{subsec: groupoids of common etale neighbourhoods} we discuss equivalence relations which can be imposed on common \'etale neighbourhoods so that they form the morphisms of a groupoid, and in \ref{subsec: stacks of etale germs} we use these ideas to define stacks of $c$th-order and \'etale germs of $n$-dimensional varieties. In \ref{subsec: quasi-coherent sheaves} we consider the categories of quasi-coherent sheaves on these stacks. All of this material is related to the setting of universal $\CD$-modules, but we conclude in \ref{subsec: strict analogues} with some remarks about the \emph{strict} analogues of these definitions, which will be necessary for the setting of universal $\CO$-modules. 

\subsection{Families of pointed varieties and common \'etale neighbourhoods}
\label{subsec: families of pointed varieties}

\begin{defn}
Given two smooth families $\pi_i: X_i \to S_i$ a \emph{fibrewise \'etale} morphism between them is given by a commutative diagram
\begin{center}
\begin{tikzpicture}[>=angle 90]
\matrix(a)[matrix of math nodes, row sep=3em, column sep=3em, text height=1.5ex, text depth=0.25ex]
{X_1 & X_2 \\
S_1 & S_2\\};
\path[->, font=\scriptsize]
(a-1-1) edge node[above]{$f_X$} (a-1-2)
        edge node[left]{$\pi_1$} (a-2-1)
(a-1-2) edge node[right]{$\pi_2$} (a-2-2)
(a-2-1) edge node[below]{$f_S$} (a-2-2);
\end{tikzpicture}
\end{center}
such that for any $s \in S_1$ with $s^\prime\defeq f_S(s) \in S_2$, the induced morphism on fibres $\left(X_1\right)_{s} \to \left(X_2\right)_{s^\prime}$ is \'etale.
\end{defn} 

\begin{notation}
We will often use the subscripts $X$ and $S$ to distinguish between the two maps comprising a fibrewise \'etale morphism, even when neither of the smooth families involved is $X/S$.
\end{notation}

We are interested in pointed $n$-dimensional varieties; in the relative setting this is formalised as follows:
\begin{defn} \label{pointed family}
Fix a base scheme $S \in \Sch$. A \emph{pointed $n$-dimensional family} over $S$ is a scheme $X$ equipped with 
\begin{itemize}
\item a morphism $\pi: X \to S$, smooth of relative dimension $n$; and
\item a section $\sigma: S \to X$.
\end{itemize}
\end{defn}
\begin{notation}
We shall denote such a family by $\pi: X \rightleftarrows S: \sigma$, but will often abbreviate to $(\pi, \sigma)$ or $X \rightleftarrows S$ when there is no risk of confusion.
\end{notation}

A particular $n$-dimensional family which will be of special importance to us is the \emph{trivial} $n$-dimensional family $\A{n}_S = S \times \A{n}$ over $S$. We will often work with the \emph{zero section} $z: S \to S \times \A{n}$, induced by the inclusion of the origin in $\A{n}$. Whenever we write $S \times \A{n} \rightleftarrows S$ without specifying the maps, we will always mean the canonical projection and the zero section.

Another important pointed $n$-dimensional family is the following: let $X \to S$ be any smooth family of relative dimension $n$, and consider the pointed family
\begin{align}\label{universal family over X}
\pr_1: X \times_S X \rightleftarrows X: \Delta.
\end{align}
We think of this as the universal pointed family over $X$. Whenever we write $X\times_S X \rightleftarrows X$ without specifying the maps, we will always mean the projection onto the first factor and the diagonal embedding. 

We want to define a groupoid of pointed $n$-dimensional families over a fixed base scheme $S$ up to \'etale morphism. When $S=\Spec k$ is a point, we have the notion of a common \'etale neighbourhood of pointed varieties, which plays the role of isomorphism in the groupoid, and we generalise this notion to the relative setting as follows:

\begin{defn} \label{common etale neighbourhood}
Let $\pi_i:X_i \rightleftarrows S: \sigma_i\ (i=1,2)$ be smooth $n$-dimensional families over $S$. A \emph{common \'etale neighbourhood} is given by a third pointed $n$-dimensional family $(\rho: V \rightleftarrows S: \tau)$ together with a pair of \'etale maps $(\phi: V \to X_1, \psi: V \to X_2)$ such that $\phi$ and $\psi$ are compatible with the projections, and furthermore are compatible with the sections \emph{on the level of reduced schemes}. That is, we require
\begin{enumerate}
\item $\pi_1 \circ \phi = \rho = \pi_2 \circ \psi$
\item $\sigma_1 \circ \redEmb{S} = \phi \circ \tau \circ \redEmb{S}, \qquad \sigma_2 \circ \redEmb{S} = \psi \circ \tau \circ \redEmb{S}$,
\end{enumerate}
where $\redEmb{S}: S_\red \emb S$ denotes the canonical closed embedding. Diagrammatically, we depict this common \'etale neighbourhood as follows, where the section $\tau$ is denoted by a dotted line to remind us that it is only compatible with the sections $\sigma_i$ on the reduced part of $S$:
\begin{center}
\begin{tikzpicture}[>=angle 90]\label{unpointed common etale neighbourhood diagram}
\matrix(c)[matrix of math nodes, row sep=2em, column sep=2em, text height=1.5ex, text depth=0.25ex]
{    &  V  & \\
 X_1 && X_2 \\
  & S. & \\};
\path[->, font=\scriptsize]
(c-1-2) edge node[above left]{$\phi$} (c-2-1)
(c-1-2) edge node[above right]{$\psi$} (c-2-3)
(c-2-1) edge[bend left=10] (c-3-2)
(c-3-2) edge[bend left=15] (c-2-1)
(c-2-3) edge[bend right=10] (c-3-2)
(c-3-2) edge[bend right=15] (c-2-3)
(c-1-2) edge[bend left=10] (c-3-2);
\path[densely dotted, ->]
(c-3-2) edge[bend left=10] (c-1-2);
\end{tikzpicture}
\end{center}
\end{defn}

\begin{notation}
We will denote a common \'etale neighbourhood by $(V, \phi, \psi)$; when no confusion will result, we may use the notation $(\phi, \psi)$ or simply $V$. 
\end{notation}

\begin{defn}\label{strict common etale neighbourhood}
In the case that the diagram is actually commutative, and not just up to precomposing with $\redEmb{S}$, we will say that the common \'etale neighbourhood is \emph{strict}, and will use a solid rather than a dotted line for the section $S \to V$.
\end{defn}

\begin{rmk}\label{locally defined morphisms}
We can also introduce another variation on the definition of common \'etale neighbourhood: rather than requiring the middle family to live over $S$, we allow smooth families over any scheme $T$ equipped with an \'etale morphism to $S$ which is compatible with the projections. There are both strict and non-strict versions of such \emph{\'etale-locally-defined common \'etale neighbourhoods}.
\end{rmk}

\subsection{Groupoids of common \'etale neighbourhoods}
\label{subsec: groupoids of common etale neighbourhoods}

Our goal is to define a groupoid $\prestackVbc{\infty}(S)$ for each scheme $S$, whose objects are pointed $n$-dimensional families over $S$, and whose morphisms are represented by common \'etale neighbourhoods. In order to do this, we need to impose an equivalence relation on common \'etale neighbourhoods, so that the composition of morphisms is well-defined and associative, and the morphisms are invertible. 

Moreover, we expect that restricting a common \'etale neighbourhood by pulling back along another \'etale morphism should not change the corresponding morphism in our groupoid. More formally, let $(V, \phi, \psi)$ be a common \'etale neighbourhood between $X_1 \rightleftarrows S$ and $X_2 \rightleftarrows S$, and suppose that we have a pointed $n$-dimensional $S$-family $V^\prime$ \'etale over $V$:
\begin{center}
\begin{tikzpicture}[>=angle 90]
\matrix(a)[matrix of math nodes, row sep=3em, column sep=3em, text height=1.5ex, text depth=0.25ex]
{V^\prime & V\\
S         & S.\\};
\path[->, font=\scriptsize]
(a-1-1) edge node[above]{$f_X$} (a-1-2)
        edge[bend left=10] node[right]{$\rho^\prime$} (a-2-1)
(a-2-1) edge[bend left=10] node[left]{$\tau^\prime$} (a-1-1)
(a-1-2) edge[bend left=10] node[right]{$\rho$} (a-2-2)
(a-2-2) edge[bend left=10] node[left]{$\tau$} (a-1-2);
\path[-]
(a-2-1) edge[double,double distance=2pt] (a-2-2);
\end{tikzpicture}
\end{center}
Then this yields a second common \'etale neighbourhood $(V^\prime, \phi \circ f_X, \psi \circ f_X)$, which we would like to treat as the same morphism in our groupoid. 

Motivated by this, we introduce the following equivalence relation:
\begin{defn}\label{similarity}
We will say that two common \'etale neighbourhoods 
\begin{align*}
(V_i, \phi_i, \psi_i), \quad i=1,2
\end{align*}
between $X_1 \rightleftarrows S$ and $X_2 \rightleftarrows S$ are \emph{similar} if there exists a pointed family $W \rightleftarrows S$ and \'etale maps $f_i: W/S \to V_i/S$ compatible with the sections on the level of $S_\red$, such that
\begin{align*}
\phi_1 \circ f_1 = \phi_2 \circ f_2;\\
\psi_1 \circ f_1 = \psi_2 \circ f_2.
\end{align*}
\end{defn}

This is an equivalence relation, but it is slightly too restrictive for our purposes. We modify it as follows:
\begin{defn}\label{infinity equivalence relation}
We will say that two common \'etale neighbourhoods 
\begin{align*}
(V_i, \phi_i, \psi_i), \quad i=1,2
\end{align*}
are \emph{$(\infty)$-equivalent} if for each $s \in S$ there is a Zariski open neighbourhood $S^\prime$ of $s$ such that the restrictions of $(V_i, \phi_i, \psi_i)$ to $S^\prime$ give similar common \'etale neighbourhoods between $X_1 \times_S S^\prime \rightleftarrows S^\prime$ and $X_2 \times_S S^\prime \rightleftarrows S^\prime$. 
\end{defn}

This equivalence relation is exactly what we need to define a groupoid structure. Given two common \'etale neighbourhoods 
\begin{align*}
(V_i/S, \phi_i, \psi_i),\ i=1,2,
\end{align*}
representing morphisms $X_1/S \to X_2/S \to X_3/S$, we would like to define their composition using the fibre product $V_1 \times_{X_2} V_2$, but it requires a little care to show that this is well-defined. It is clear that this is a smooth scheme of relative dimension $n$ over $S$, but what is not immediate is the existence of a suitable section. However, we can define a map $S_\red \to V_1 \times_{X_2} V_2$ using the compatibility of the sections $\tau_1 \circ \redEmb{S}$ and $\tau_2 \circ \redEmb{S}$; this map then extends to the desired section using formal smoothness of $V_1 \times_{X_2} V_2 \to S$. Although the choice of extension is not unique, any two choices differ only up to nilpotence, so the resulting common \'etale neighbourhoods will be equivalent. Therefore, the composition of the morphisms $X_1/S \to X_2/S \to X_3/S$ is indeed represented by the pullback:
\begin{center}
\begin{tikzpicture}
[>=angle 90, 
cross line/.style={preaction={draw=white, -, line width=4pt}}]
\matrix(g)[matrix of math nodes, row sep=1.5em, column sep=3em, text height=1.5ex, text depth=0.25ex]
{     &     &   & V_1 \times_{X_2} V_2 &      &   \\
      & V_1 &   &                      & V_2  &   \\
 X_1  &     &X_2& S                    &      &X_3\\
      & S   &   &                      & S    &   \\
 S    &     & S &                      &      & S. \\};
\path[-]         
 (g-4-2) edge[ double, double distance= 2pt] (g-5-1)
         edge[ double, double distance= 2pt] (g-5-3)
 (g-3-4) edge[ double, double distance= 2pt] (g-4-2)
         edge[ double, double distance= 2pt] (g-4-5)
 (g-4-5) edge[ double, double distance= 2pt] (g-5-3)
         edge[ double, double distance= 2pt] (g-5-6);
\path[densely dotted, ->, font=\scriptsize]
 (g-4-2) edge[bend left=10] (g-2-2)  
 (g-4-5) edge[bend left=10] (g-2-5)
 (g-3-4) edge[bend left=10] (g-1-4);
\path[->, font=\scriptsize]
 (g-3-1) edge[bend left=10] (g-5-1)
 (g-5-1) edge[bend left=10] (g-3-1)
 (g-2-2) edge[bend left=10] (g-4-2)
 
 (g-5-3) edge[draw=white, -, line width=4pt, bend left=10] (g-3-3)
 (g-3-3) edge[cross line, bend left=10] (g-5-3)
 (g-5-3) edge[bend left=10] (g-3-3)
 (g-1-4) edge[bend left=10] (g-3-4)

 (g-2-5) edge[bend left=10] (g-4-5)

 (g-3-6) edge[bend left=10] (g-5-6)
 (g-5-6) edge[bend left=10] (g-3-6)

 (g-2-2) edge (g-3-1)
         edge (g-3-3)
 (g-1-4) edge (g-2-2)
         edge (g-2-5)
 (g-2-5) edge[cross line] (g-3-3)
         edge (g-3-6);
\end{tikzpicture}
\end{center}
It is not hard to check that this is associative. 

Next, given $X \rightleftarrows S$, it is straightforward to check that the identity morphism $\id_{X \rightleftarrows S}$ is simply represented by $(X, \id_X, \id_X)$. Moreover, any symmetric common \'etale neighbourhood $(V, \phi, \phi)$ is equivalent to the family $(X, \id_X, \id_X)$ and hence also represents the identity morphism. 

Finally, given a common \'etale neighbourhood, its inverse is represented by the mirror image diagram. Indeed, to show that the composition of a common \'etale neighbourhood with its mirror image represents the identity morphism, we need only remark that
\begin{align*}
\Delta: V \to V\times_{X_2} V
\end{align*}
is an open embedding (because $\psi$ is unramified and locally of finite type), and in particular is \'etale. Pulling back the composition along $\Delta$ gives $(V, \phi, \phi)$, which is equivalent to the identity common \'etale neighbourhood. 

We have proved the following:
\begin{prop}
Under the $(\infty)$-equivalence relation and with the composition and inverses described above, $\prestackVb^{(\infty)}(S)$ is a groupoid.
\end{prop}

\begin{rmk}
As we will see in Lemma \ref{lemma: uniqueness of liftings} and Proposition \ref{prop: the group homomorphisms are isomorphisms}, two common \'etale neighbourhoods $(V_i, \phi_i, \psi_i)$ $(i=1,2)$ are $(\infty)$-equivalent precisely when they induce the same isomorphism of the formal neighbourhoods of $S$ in the schemes $X_1$ and $X_2$:
\begin{align*}
\hat{\psi_1} \circ \hat{\phi_1}^{-1} = \hat{\psi_2} \circ \hat{\phi_2}^{-1}.
\end{align*}
\end{rmk}

Motivated by this observation, we introduce a family of coarser equivalence relations:
\begin{defn}
Let $c \in \BN$. Two common \'etale neighbourhoods $(V_i, \phi_i, \psi_i)$ are \emph{$(c)$-equivalent} if they induce the same isomorphisms on the $c$th infinitesimal neighbourhoods of $S$ in $X_1$ and $X_2$:

\begin{align*}
\psi_1^{(c)} \circ \left(\phi_1^{(c)}\right)^{-1} = \psi_2^{(c)} \circ \left(\phi_2^{(c)}\right)^{-1}: X_1^{(c)} \EquivTo X_2^{(c)}.
\end{align*}
\end{defn}

Then we let $\prestackVbc{c}(S)$ be the groupoid whose objects are pointed $n$-dimension\-al families over $S$ and whose morphisms are common \'etale neighbourhoods up to $(c)$-equivalence. 

Since $(c)$-equivalence is coarser than $(c+1)$-equivalence for any $c$, and also than $(\infty)$-equivalence, we obtain morphisms of groupoids
\begin{align*}
\prestackVbc{\infty}(S) \to \ldots \to \prestackVbc{c+1}(S) \to \prestackVbc{c}(S) \to \ldots.
\end{align*}

\subsection{Stacks of \'etale and \texorpdfstring{$c$th}{cth}-order germs of varieties}
\label{subsec: stacks of etale germs}
With these preliminary notions and definitions established, we can define the pre\-stacks of germs of varieties as follows:

\begin{defn}\label{prestack of varieties}
Given $c \in \BN \cup \{\infty\}$, let $\prestackVc{c}$ be the prestack that sends a test scheme $S$ to the groupoid whose only object is the trivial pointed $n$-dimensional variety $\pi: S \times \A{n} \rightleftarrows S: z$, and whose automorphisms are given by common \'etale neighbourhoods of $S \times \A{n}$ with itself, modulo $(c)$-equivalence.
\end{defn}

There is a distinguished class of common \'etale neighbourhoods of $S \times \A{n}$, characterised as follows:

\begin{defn}\label{standard etale neighbourhoods}
A common \'etale neighbourhood $V=(V, \phi, \psi)$ between the trivial pointed family and itself will be called \emph{split} if there exists an $n$-dimensional variety $W$ together with maps $\overline{\phi},\overline{\psi}: S \times W \to S \times \A{n}$ (not necessarily \'etale), an open embedding $V \emb S \times W$, and a point $w \in W$ such that the following diagram commutes:
\begin{center}
\begin{tikzpicture}
[>=angle 90,
cross line/.style={preaction={draw=white, -, line width=6pt}}]
\matrix(f)[matrix of math nodes, row sep=1.5em, column sep=3em, text height=1.5ex, text depth=0.25ex]
{               &   &            & V &     \\
                &   & S \times W &   &     \\
 S \times \A{n} &   &            & S & S \times \A{n} \\
                &   & S          &   &     \\
 S              &   &            &   & S.   \\};
\path[-]
 (f-4-3) edge[double, double distance = 2pt] (f-5-1)
         edge[double, double distance = 2pt] (f-5-5)
 (f-3-4) edge[double, double distance = 2pt] (f-4-3);
\path[densely dotted, ->, font=\scriptsize]
 (f-4-3) edge[bend left=10] node[left, inner sep=1pt]{$\id_S \times i_w$} (f-2-3)
 (f-3-4) edge[bend left=10] node[left]{$\tau$} (f-1-4);
\path[->, font=\scriptsize]
 (f-3-1) edge[bend left=10] node[right]{$\pi$} (f-5-1)
 (f-5-1) edge[bend left=10] node[left]{$z$} (f-3-1)
 (f-2-3) edge[bend left=10] node[right]{$\pr_S$} (f-4-3)

 (f-1-4) edge[bend left=10] node[right]{$\rho$} (f-3-4)

 (f-3-5) edge[bend left=10] node[right]{$\pi$} (f-5-5)
 (f-5-5) edge[bend left=10] node[left]{$z$} (f-3-5)

 (f-2-3) edge node[below right]{$\overline{\phi}$} (f-3-1)
         edge[cross line] node[right=5mm, inner sep=3pt]{$\overline{\psi}$} (f-3-5);
\path[right hook->]
 (f-1-4) edge (f-2-3);
\path[->, font=\scriptsize]
 (f-1-4) edge[bend right=15] node[above left]{$\phi$} (f-3-1)
         edge[bend left=15] node[above right]{$\psi$} (f-3-5); 
\end{tikzpicture}
\end{center}
\end{defn}

\begin{rmk}
Split common \'etale neighbourhoods can be simpler to work with, and will arise in our discussion of the Artin approximation theorem. Fortunately we will see in Lemma \ref{lemma: ubiquity of split neighbourhoods} that all common \'etale neighbourhoods of the trivial pointed variety are $(\infty)$-equivalent (and hence $\cth$-equivalent, for any $c$) to a common \'etale neighbourhood which is split. 
\end{rmk}

\begin{defn} 
Let $\varietiesc{c}$ be the stackification of $\prestackVc{c}$ in the \'etale topology. When $c = \infty$, we call $\varietiesc{c}$ the \emph{stack of \'etale germs of $n$-dimensional varieties}. For finite $c$, we call $\varietiesc{c}$ the \emph{stack of $c$th-order germs of $n$-dimensional varieties}. 
\end{defn}

In fact we will find it convenient to work with the intermediate prestack $\prestackVbc{c}$, which lies in between $\prestackVc{c}$ and its stackification $\varietiesc{c}$. 

\begin{defn}
For $c \in \BN \cup \{\infty\}$, let $\prestackVbc{c}$ be the subprestack of $\varietiesc{c}$ sending a test scheme $S$ to the subgroupoid $\prestackVbc{c}(S)$ of $\varietiesc{c}(S)$ defined above. Its objects are pointed $n$-dimensional varieties over $S$ and its morphisms are represented by common \'etale neighbourhoods up to $\cth$-equivalence.
\end{defn}

We see that this gives a prestack whose stackification is $\varietiesc{c}$: indeed, when constructing the stackification of $\prestackVc{c}$ explicitly (as in for example section 8.8 of \cite{stacks-project}), we must add in locally defined objects, which include all of the additional objects of $\prestackVbc{c}(S)$; we must also add in all of the locally defined morphisms between these new objects, and hence in particular all of the morphisms of $\prestackVbc{c}(S)$. The next stage in constructing the stackification is to identify all morphisms which agree locally; however, this has already been done in $\prestackVbc{c}(S)$ by our definition of $(c)$-equivalence. It follows that we can view $\prestackVbc{c}(S)$ as a (non-full) sub-groupoid of $\varietiesc{c}(S)$, and hence by the universal property, we obtain a map from the stackification of $\prestackVbc{c}$ into $\varietiesc{c}$. The quasi-inverse to this map is induced by the obvious inclusion of $\prestackVc{c}$ into $\prestackVbc{c}$.

\begin{rmk}
The crucial difference between the stack $\varietiesc{c}$ and the prestack $\prestackVbc{c}$ (and the reason that it is simpler to work with $\prestackVbc{c}$) is that the groupoid $\varietiesc{c}(S)$ contains isomorphisms represented by common \'etale neighbourhoods that are only defined \'etale-locally over the base as in Remark \ref{locally defined morphisms}.
\end{rmk}

\subsection{Quasi-coherent sheaves on \texorpdfstring{$\varietiesc{c}$}{the stack of germs}}
\label{subsec: quasi-coherent sheaves}
We will be interested in studying the categories of quasi-coherent sheaves on the stacks $\varietiesc{c}$ for $c \in \BN \cup \{\infty\}$. For details of the theory of quasi-coherent sheaves on stacks and more generally prestacks see Gaitsgory's notes \cite{G-QCoh}. Since the categories of quasi-coherent sheaves on a prestack and its stackification are equivalent, we have the following equivalences:
\begin{align*}
\QCoh{\prestackVc{c}} \simeq \QCoh{\varietiesc{c}} \simeq \QCoh{\prestackVbc{c}}.
\end{align*}

We will find it convenient to work in the realisation of the category given by $\QCoh{\prestackVbc{c}}$. Concretely, an object $M$ of $\QCoh{\prestackVbc{c}}$ consists of a collection of quasi-coherent sheaves together with coherences: for each map $S \to \prestackVbc{c}$ (i.e. for each $X \rightleftarrows S$ smooth of relative dimension $n$), we have an object $M_{X \rightleftarrows S} \in \QCoh{S}$. Moreover, we require compatibility under pullbacks in the following sense. Suppose that for $i=1,2$ we have $S_i \xrightarrow{(\pi_i,\sigma_i)} \prestackVbc{c}$, two $n$-dimensional families, together with a map $f: S_1 \to S_2$ and a commutative diagram of prestacks:
\begin{center}
\begin{tikzpicture}[>=angle 90]
\matrix(i)[matrix of math nodes, row sep=1em, column sep=1em, text height=1.5ex, text depth=0.25ex]
{ S_2 &   & \prestackVbc{c} .\\
      &{ }& \\
 S_1  &   & \\};
\path[->, font=\scriptsize]
 (i-1-1) edge node[above]{$(\pi_2, \sigma_2)$} (i-1-3)
 (i-3-1) edge node[left]{$f$} (i-1-1)
         edge node[below, sloped]{$(\pi_1, \sigma_1)$} (i-1-3);
\path[-implies, font=\scriptsize]
 (i-2-2) edge[double equal sign distance] node[below=3pt]{$\alpha$} (i-1-1);
\end{tikzpicture}
\end{center} 
Recall that in $\PreStk$, commutativity of a diagram is a structure, not a property, in this case amounting to an automorphism $\alpha$ in $\prestackVbc{c}(S_1)$ between the objects corresponding to $(\pi_1, \sigma_1)$ and $(\pi_2, \sigma_2) \circ f$, represented by a common \'etale neighbourhood of the form
\begin{center}
\begin{tikzpicture}[>=angle 90]
\matrix(e)[matrix of math nodes, row sep=1.5em, column sep=3em, text height=1.5ex, text depth=0.25ex]
{   &  V_\alpha  &  \\
 X_1&     &      S_1 \times_{S_2} X_2 \\
    &  S_1       &  \\
 S_1  &          & S_1.   \\};
\path[->, font=\scriptsize]
 (e-2-1) edge[bend left=10] node[right]{$\pi_1$} (e-4-1)
 (e-4-1) edge[bend left=10] node[left]{$\sigma_1$} (e-2-1)
 (e-1-2) edge[bend left=10] node[right]{$\rho_\alpha$} (e-3-2)
 (e-2-3) edge[bend left=10] node[right]{$f^*\pi_2$} (e-4-3)
 (e-4-3) edge[bend left=10] node[left]{$f^*\sigma_2$} (e-2-3)
 (e-1-2) edge node[above left]{$\phi_\alpha$} (e-2-1)
         edge node[above right]{$\psi_\alpha$} (e-2-3);
\path[-]
 (e-3-2) edge[double, double distance=2pt] (e-4-1)
         edge[double, double distance=2pt] (e-4-3);
\path[densely dotted, ->, font=\scriptsize]
 (e-3-2) edge[bend left=10] node[left]{$\tau_\alpha$} (e-1-2);
\end{tikzpicture}
\end{center}

We require that in such a situation, we have an isomorphism 
\begin{align*}
M(f, \alpha): f^*\left(M_{X_2 \rightleftarrows S_2}\right) \EquivTo M_{X_1 \rightleftarrows S_1}
\end{align*}
in $\QCoh{S_1}$. This isomorphism must be independent of the choice of representative $(V_\alpha, \phi_\alpha, \psi_\alpha)$ of the isomorphism $\alpha$ in $\prestackVbc{c}(S_1)$. We also require that these isomorphisms be compatible with compositions $S_1 \xrightarrow{f} S_2 \xrightarrow{g} S_3$.

\subsection{Strict analogues, for the \texorpdfstring{$\CO$}{O}-module setting}
\label{subsec: strict analogues}
Let us introduce the following strict analogues, which will be important in the setting of universal $\CO$-modules.

\begin{defn}
Fix $c \in \BN \cup \{\infty\}$. Let $\prestackPVc{c}$ be the prestack that sends a test scheme $S$ to the groupoid whose only object is the trivial pointed $n$-dimensional variety $\pi: S \times \A{n} \rightleftarrows S: z$, and whose automorphisms are given by strict common \'etale neighbourhoods of $S \times \A{n}$ with itself, up to $(c)$-equivalence.
\end{defn}

\begin{rmk}
In the case $c= \infty$, one might be tempted to consider a \emph{strict} version of $(\infty)$-equivalence, defined in the obvious way. It is straightforward to check that two strict common \'etale neighbourhoods are $(\infty)$-equivalent if and only if they are strictly $(\infty)$-equivalent, so in fact it is not necessary to introduce this latter notion. 
\end{rmk}

\begin{defn} 
Let $\pointedvarietiesc{c}$ be the stackification of $\prestackPVc{c}$ in the \'etale topology. When $c=\infty$, we call $\pointedvarietiesc{c}$ the \emph{stack of pointed \'etale germs of $n$-dimensional varieties}; when $c$ is finite, $\pointedvarietiesc{c}$ is the \emph{stack of pointed $c$th-order germs of $n$-dimensional varieties}. 
\end{defn}

As in the non-strict setting, we will also work with an intermediate prestack $\prestackPVbc{c}$, which lies in between the prestack $\prestackPVc{c}$ and its stackification. Name\-ly, for a given test scheme $S$, an object of the groupoid $\prestackPVbc{c}(S)$ is a pointed $n$-dimensional family over $S$, $\pi: X \rightleftarrows S: \sigma$. Given two such pointed families, a morphism between them is represented by a strict common \'etale neighbourhood $(V, \phi, \psi)$, modulo $(c)$-equivalence. Similarly to the groupoid $\prestackVbc{c}(S)$, composition is given by pullback and inverses are given by mirror-image diagrams. 

The difference between the strict and non-strict definitions lies in whether we require morphisms to preserve the distinguished points of the $n$-dimensional varieties (in the strict setting), or allow infinitesimal translations (in the non-strict setting). As we will see in Section \ref{sec: universal modules}, this is what gives quasi-coherent sheaves on $\pointedvarietiesc{c}$ the additional structure of an action of the sheaf of differential operators. 

\section{Groups of automorphisms and their classifying stacks}
\label{sec: classifying stacks}
In this section, we introduce certain groups $G$ and $K$ of automorphisms of the formal disc. We begin in \ref{subsec: G} by defining the group formal scheme $G$ and its finite-dimensional quotients $G^\cth$; in \ref{subsec: K} we introduce the reduced part $K = G_\red$. It is a pro-algebraic group, and contains a pro-unipotent subgroup $K_u$. In \ref{subsec: representations and classifying stacks} we give some general definitions and facts regarding representations of group-valued prestacks and classifying stacks, and in \ref{subsec: application to G and K} we apply these ideas to the groups $G$ and $K$. We also begin the comparison of the classifying stacks $BG$ and $BK$ with the stacks $\varietiesc{\infty}$ and $\pointedvarietiesc{\infty}$, which will be the motivation for the next several sections. 

\subsection{The group \texorpdfstring{$G$}{G} of continuous automorphisms of the formal disc}
\label{subsec: G}
\begin{defn}\label{def: G}
Let $\On = k\series{n}$, and let $G=\underline{\Aut} \On$ be the ind-affine group formal scheme of continuous automorphisms of $\On$. Explicitly, for $S = \Spec(R)$, $G(S)$ is the group of automorphisms of the $R$-algebra $R\series{n}$, continuous with respect to the topology corresponding to the ideal $\fm$ generated by $(t_1, \ldots, t_n)$. 
\end{defn}

A continuous homomorphism $\rho: R \series{n} \to R \series{n}$ is determined by its values on the topological generators $t_1, \ldots, t_n$. Given a multi-index $J=(j_1, \ldots, j_n) \in \BZ^n_{\ge 0}$, let us denote by $r^k_{J}$ the coefficient of $\underline{t}^J =t_1^{j_1}\cdots t_n^{j_n}$ in the series $\rho(t_k) \in R\series{n}$. For $k^\prime \in \{1, \ldots, n\}$, let $e_{k^\prime} = (0, \ldots, 0, 1, 0, \ldots, 0)$ be the multi-index with $1$ only in the $k^\prime$th place. With this notation,
\begin{align}\label{defining rho}
\rho: t_k \mapsto r^k_{\underline{0}} + \sum_{k^\prime=1}^n r_{e_{k^\prime}}^k t_{k^\prime} + \text{ higher order terms}.
\end{align}
The condition that this determines a continuous homomorphism is equivalent to requiring each $r^k_{\underline{0}}$ to be a nilpotent element of $R$. Then the homomorphism $\rho$ determined by the formula (\ref{defining rho}) is invertible precisely when the matrix $(r^k_{e_{k^\prime}})_{k,k^\prime} \in M_n (R)$ is invertible.

This allows us to describe the indscheme structure of $G$ explicitly:
\begin{align*}
G = \colim_{N \in \BN} {\Spec \left( {k[a^k_J, (\det{(a^k_{e_{k^\prime}})_{k, k^\prime}})^{-1}]/((a^k_{\underline{0}}) ^N) }\right) }.
\end{align*}

\begin{defn}\label{def: G^c}
Given $c \in \BN$, we can also consider $G^{\cth}$, a quotient of $G$:
\begin{align*}
G^\cth = \colim_{N \in \BN} {\Spec \left( {k[a^k_J, (\det{(a^k_{e_{k^\prime}})_{k, k^\prime}})^{-1}]_{|J|\le c}/((a^k_{\underline{0}}) ^N) }\right) },
\end{align*}
where 
\begin{align*}
|J| \defeq \sum_{i=1}^n j_i.
\end{align*}
\end{defn}

We refer to $G^\cth$ as the group formal scheme of automorphisms of the formal $n$-dimensional disc, although this is technically an abuse of terminology, because the formula $(2)$ taken modulo $\fm^{c+1}$ does not always determine a continuous automorphism of $\On/\fm^{c+1}$. Now $G^\cth$ is an indscheme of finite type and a quotient of $G$, and we can express $G$ as the limit
\begin{align*}
G = \lim_{c \in \BN} G^\cth.
\end{align*}
That is, $G$ is a pro-object in the category of ind-affine group formal schemes. 

\begin{eg}
It may be useful to keep in mind the notationally simpler one-dimensional setting. When $n=1$, an automorphism $\rho: R\seriest \to R\seriest$ is determined by its value on the single generator $t$:
\begin{align*}
\rho: t \mapsto r_0 + r_1 t + r_2 t^2 + \ldots,
\end{align*}
where $r_0 \in \Nil(R)$ and $r_1 \in R^\times$. 

The indscheme $G$ is the colimit (of schemes of infinite type)
\begin{align*}
G=\colim_{N \in \BN} \Spec{ k[a_0, a_1, a_1^{-1}, a_2, a_3, \ldots]/(a_0^N) }.
\end{align*}
On the other hand, it is also the limit of the indschemes $G^\cth$ of finite type, where
\begin{align*}
G^\cth = \colim_{N \in \BN} \Spec{k[a_0, a_1, a_1^{-1}, a_2, a_3, \ldots, a_c]/(a_0^N)}.
\end{align*}
The quotient maps $G^\cth \to G^{(c-1)}$ correspond to the inclusions
\begin{align*}
k[a_0, a_1, a_1^{-1}, a_2, a_3, \ldots, a_{c-1}]/(a_0^N) &\emb k[a_0, a_1, a_1^{-1}, a_2, a_3, \ldots, a_c]/(a_0^N)\\
a_i & \mapsto a_i.
\end{align*} 

%Note also that it is easy to see from this example that a continuous homomorphism $\rho: R\seriest \to R \seriest$ always descends to give a homomorphism $\rho^\cth: R[t]/\fm^{c+1} \to R[t]/\fm^{c+1}$. Moreover, $\rho$ is invertible if and only if $\rho^\cth$ is invertible for some (or equivalently for all) $c \ge 1$, because this is a condition on the coefficients of the degree $0$ and $1$ terms only. This is true for $n>1$ as well, for the same reasons. 
\end{eg}

\subsection{The reduced part \texorpdfstring{$K=G_\red$}{K}}
\label{subsec: K}
\begin{defn} \label{def: K}
Let $K=G_\red$ denote the reduced part of the indscheme $G$:
\begin{align*}
K = \Spec{ k[a^k_J, (\det{(a^k_{e_{k^\prime}})_{k, k^\prime}})^{-1}]_{|J| > 0} }.
\end{align*}
\end{defn}

It is an affine group scheme of infinite type. Geometrically, $(\Spec{R})$-points of $K$ correspond to continuous automorphisms $\rho: R\series{n} \to R \series{n}$ such that the constant term of each series $\rho(t_k)$ is zero. We think of $K$ as parametrising automorphisms of the formal disc $\Spf{k\series{n}}$ which fix the origin $0$, whereas the automorphisms parametrised by the larger group $G$ may involve infinitesimal translations of $0$. 

We view $K$ as a \emph{pro-algebraic group}: it has finite-dimensional quotients $K^\cth$, parametrising automorphisms of $k\series{n} / \fm^c$ which preserve the origin. 
\begin{defn}\label{def: K^c}
Explicitly, $K^\cth$ is the algebraic group
\begin{align*}
K^\cth = \Spec{ k[a^k_J, (\det{(a^k_{e_{k^\prime}})_{k, k^\prime}})^{-1}]_{0<|J|<c+1} },
\end{align*}
where the group structure comes from composition of the automorphisms $\rho$. 
\end{defn}

Note that we have obvious maps
\begin{align*}
K, K^\cth \to GL_n,
\end{align*}
where the map on $(\Spec{R})$-points sends an automorphism $\rho$ to the matrix
\begin{align*}
(r^k_{e_{k^\prime}})_{k, k^\prime} \in GL_n(R),
\end{align*}
in the notation of (\ref{defining rho}). These are homomorphisms of affine group schemes. (Notice that we might try to define a similar map for the groups $G, G^\cth$, but that this no longer respects the group structure.)

\begin{defn}\label{defn: unipotent subgroups of K}
Let $K_u$ and $K^\cth_u$ denote the kernels of the homomorphisms of group schemes $K \to GL_n$ and $K^\cth \to GL_n$ respectively. 
\end{defn}

Then $K^\cth_u$ is a unipotent algebraic group, and $K_u = \displaystyle\lim_{c \in \BN} K_u^\cth$ is a pro-unipotent group. We can write
\begin{align*}
K = GL_n \ltimes K_u, \qquad K^\cth = GL_n \ltimes K^\cth_u;
\end{align*} 
this will be helpful in Section \ref{sec: groups of etale automorphisms} in understanding the representation theory of $K$. 

\begin{eg}
Let us again consider the case $n=1$, where the notation is more pleasant. We have
\begin{align*}
K & = \Spec{ k[a_1, a_1^{-1}, a_2, a_3, \ldots ]},\\
K^\cth & = \Spec{ k[a_1, a_1^{-1}, a_2, a_3, \ldots, a_c]}.
\end{align*}
The maps $K, K^\cth \to GL_1 = \BG_m = \Spec k[x,x^{-1}]$ are induced by the algebra homomorphisms given by 
\begin{align*}
x \mapsto a_1.
\end{align*}
The unipotent groups are given by
\begin{align*}
K_u &= \Spec{ k[a_2, a_3, \ldots] },\\
K^\cth_u &= \Spec{ k[a_2, a_3, \ldots, a_c]}.
\end{align*}
Let us consider the coalgebra structure on the algebra of functions $k[a_2, a_3, \ldots]$, induced by the composition of automorphisms $\rho, \sigma \in K_u(k)$. Suppose that
\begin{align*}
\rho: t & \mapsto t + r_2 t^2 + r_3 t^3 + \ldots, \\
\sigma: t & \mapsto t + s_2 t^2 + s_3 t^3 + \ldots 
\end{align*}
Then
\begin{multline*}
\rho \circ \sigma: t \mapsto \\ t + (r_2 + s_2) t^2 + (r_3 + 2 r_2 s_2 + s_2) t^3 + (r_4 + 3 r_3 s_2 + r_2 s_2^2 + 2 r_2 s_3 + s_4) t^4 + \ldots
\end{multline*}
From this we see that the comultiplication satisfies
\begin{align*}
a_2 &\mapsto a_2 \otimes 1 + 1 \otimes a_2, \\
a_3 &\mapsto a_3 \otimes 1 + 2 a_2 \otimes a_2 + 1 \otimes a_3,\\
a_4 &\mapsto a_4 \otimes 1 + 3 a_3 \otimes a_2 + a_2 \otimes a_2^2 + 2 a_2 \otimes a_3 + 1 \otimes a_4,
\end{align*}
and so on. 

The action of $\BG_m$ on $K_u$ (and similarly on $K^\cth_u$ for any $c$) induces a grading on the algebra of functions as follows: a $k$-point of $\BG_m$ is of the form $z: t \mapsto zt$, for $z \in k^\times$. Conjugating $\rho \in K_u(k)$ by $z$ gives
\begin{align*}
z \circ \rho \circ z^{-1} : t \mapsto t + zr_2 t^2 + z^2 r_3 t^3 + \ldots;
\end{align*}
that is, the grading on $k[a_2, a_3, \ldots]$ is given by $\deg(a_j) = j-1$. 
\end{eg}

Returning to the general setting ($n \ge 1$), note that the diagonal inclusion $\BG_m \emb GL_n$ results in a grading of the algebra of functions $k[a^k_J]_{|J|>1}$ of $K_u$ (and again, similarly for $K_u^\cth$): we have $\deg(a^k_J) = |J| -1$. It will be important for us that the grading is \emph{non-negative}.

\subsection{Representations and classifying stacks}
\label{subsec: representations and classifying stacks}
\begin{defn}
By a \emph{group-valued prestack}, we mean a functor
\begin{align*}
H : (\Sch^\text{aff} )^\op \to \Grp.
\end{align*}
\end{defn}

The ordinary prestack underlying $H$ is given by composing with the forgetful functor $\Grp \to \Set$ and the inclusion $\Set \to \infty{-}\Grpd$.

Let $H$ be any group-valued prestack. We wish to consider the category $\Rep(H)$ of representations of $H$:

\begin{defn}\label{def: representation of a group-valued prestack}
A \emph{representation} of $H$ on a $k$-vector space 
$V$ is a morphism of group-valued functors
\begin{align*}
\bR: H \to GL_V;
\end{align*}
that is, for any $S = \Spec{R}$ we have
\begin{align*}
\bR_R: H(S) \to GL(V \otimes_k R), 
\end{align*}
natural in $R$. 
\end{defn} 

We can reformulate this definition in a more geometric manner as follows. Recall that given a group $H$ we can define the prestack $BH_{\triv}$ classifying trivial principal $H$-bundles: for a test scheme $S$, $BH_{\triv}(S)$ is a groupoid containing only one object, the trivial bundle $S \times H \to S$. The automorphism group $\Aut_{BH_{\triv}(S)} (S \times H \to S)$ is the group $H(S)$. 

\begin{defn}
The \emph{classifying stack} $BH$ of $H$ is the stackification of the prestack $BH_{\triv}$ in the \'etale topology. 
\end{defn}

\begin{rmk}
If $H$ is an algebraic group, this is the usual classifying stack: that is, $S$-points of $BH$ are principal $H$-bundles over $S$, and automorphisms are morphisms of $H$-bundles. 
\end{rmk}

Then we have that
\begin{align*}
\Rep(H) \simeq \QCoh{BH_\triv} \simeq \QCoh{BH},
\end{align*}
where the second equivalence is due to the fact that $\QCoh{\bullet}$ is preserved by stackification. 

In the case that $H$ is an affine group scheme, say $H = \Spec {A}$ with $A$ a Hopf algebra, then the data of a representation of $H$ on a vector space $V$ is equivalent to the structure of an $A$-comodule on $V$:
\begin{align*}
V \to V \otimes_k A.
\end{align*}
(See for example Chapter VIII, Prop. 6.1. in \cite{Mil-AGS}.)

\begin{observation} \label{observation: representations are locally finite}
From this definition we can show that any representation $V$ of an affine group scheme is \emph{locally finite}: that is, every vector $v \in V$ is contained in some finite-dimensional sub-representation. (For example, see Chapter VIII, Prop. 6.6. of \cite{Mil-AGS}.) This is not true of representations of more general group-valued prestacks, as we will see in Section \ref{subsec: non-locally finite representations}.
\end{observation}

Now suppose that $H$ is a pro-algebraic group, so that
\begin{align*}
H = \lim_{i} H_i, 
\end{align*}
where $H_i$ runs over all finite-dimensional quotients of $H$. (The example we have in mind is of course the group $K$ of Section \ref{subsec: K}.) If we forget for the moment about the scheme structure on these groups, the pro-structure of $H$ gives it a topology: a base for the open neighbourhoods of $1_H$ is given by the kernels $N_i$ of the quotient maps $H \surj H_i$. 

We might be interested in restricting our attention to only those representations of $H$ which are continuous with respect to this topology. If we give the vector space $V$ the discrete topology, this amounts to requiring that for each $v \in V$, the action of $H$ on $v$ factors through one of the finite-dimensional quotients $H_i$, or equivalently, that $V$ is the union of the subrepresentations $V_i$, where $V_i$ is the largest subspace of $V$ on which the action of $H$ factors through $H_i$ or on which the action of $N_i$ is trivial. 

If $V$ is finite-dimensional to begin with, the group-valued prestack $GL_V$ is also a finite-dimensional algebraic group, and so this condition is automatic. Combining this with Observation \ref{observation: representations are locally finite}, we conclude that all representations of $H$ are necessarily continuous with respect to the discrete topology on the underlying vector space. In other words:

\begin{prop}\label{prop: representations of pro-algebraic groups are continuous}
For $H = \lim_i{H_i}$ a pro-algebraic group, 
\begin{align*}
\Rep(H) \simeq \colim_i \Rep(H_i).
\end{align*}
\end{prop}

Given a pro-algebraic group $H$ we can always write it as the limit of its finite-dimensional quotients as above; however, as with our group $K = \displaystyle\lim_{c \in \BN} K^\cth$ we can often restrict our attention to a subset of these algebraic quotients. View the collection of all finite-dimensional quotients $H_i = \Spec{A_i}$ as a category $\CI$, whose morphisms are surjections compatible with the quotient maps from $H$, and suppose that we have a subcategory $\CJ \emb \CI$ such that 
\begin{align*}
H \simeq \lim_{j \in \CJ} H_j.
\end{align*}
Then for any $i \in \CI$ we can show that there exists $j \in \CJ$ such that $H_i$ is a quotient of $H_j$. It follows that $\CJ^\op$ is cofinal in $\CI^\op$ and in particular 
\begin{align*}
\Rep(H) \simeq \colim_{j \in \CJ^\op} \Rep(H_j).
\end{align*}

\subsection{Application to \texorpdfstring{$G$}{G} and \texorpdfstring{$K$}{K}}
\label{subsec: application to G and K}

From the above discussion, it is immediate that one has a commutative diagram of equivalences:

\begin{center}
\begin{tikzpicture}[>=angle 90,bij/.style={above,sloped,inner sep=0.5pt}]
\matrix(a)[matrix of math nodes, row sep=3em, column sep=3em, text height=1.5ex, text depth=0.25ex]
{\Rep(K) & \QCoh{BK} \\
\displaystyle{\colim_{c \in \BN} \Rep\left(K^\cth\right)} & \displaystyle{\colim_{c \in \BN} \QCoh{BK^\cth}}. \\};
\path[->, font=\scriptsize]
 (a-1-1) edge node[bij]{$\sim$} (a-1-2)
 (a-2-1) edge node[bij]{$\sim$} (a-1-1)
 (a-2-1) edge node[bij]{$\sim$} (a-2-2)
 (a-2-2) edge node[bij]{$\sim$} (a-1-2);
\end{tikzpicture}
\end{center}
Now we would like to make a similar comparison between the categories $\Rep(G)$ and $\displaystyle\colim_{c \in \BN} \Rep(G^\cth)$. We have the following:

\begin{prop}\label{prop: representations of G are continuous}
All representations of the group-valued prestack $G$ are continuous with respect to the the topology induced by the pro-structure of $G$:
\begin{align*}
\colim_{c \in \BN} \Rep(G^\cth) \EquivTo \Rep(G).
\end{align*}
\end{prop}
\begin{proof}
Let $V$ be a vector space and let
\begin{align*}
\bR: G \to GL_V
\end{align*}
be a representation of $G$ on $V$. Let $v \in V$; we want to show that the action of $G$ on $v$ factors through one of its finite-dimensional quotients $G^\cth$, i.e. that there exists some $c$ such that for every $S = \Spec{R}$ the induced map
\begin{align*}
\bR_{R,v}: G(S) &\to V \otimes_k R \\
 \rho & \mapsto \bR_R (\rho) (v \otimes 1_R) 
\end{align*}
factors through the quotient $G^\cth(S)$. 

This is equivalent to showing that the restriction of $\bR_{R,v}$ to the kernel $N_c(S)$ of the quotient map $G(S) \surj G^\cth(S)$ is the constant map
\begin{align*}
n \mapsto v \otimes 1_R
\end{align*}
for every $S = \Spec{R}$. 

However, the embedding $K \emb G$ allows us to view $V$ as a representation of $K$; then by Proposition \ref{prop: representations of pro-algebraic groups are continuous}, there exists $c$ such that the restriction of $\bR_{R,v}$ to $K(S)$ factors through $K^\cth(S)$ for every $S = \Spec{R}$. This implies that the restriction of $\bR_{R,v}$ to $\ker(K(S) \surj K^\cth(S))$ is the constant map---but this kernel is exactly $N_c(S)$.  \qed
\end{proof}

\begin{observation}
Fix a base scheme $S = \Spec(R)$ and suppose that we have a common \'etale neighbourhood $(V, \phi, \psi)$ of the trivial family. Taking completions along the embeddings of $S$, we obtain isomorphisms over $S$
\begin{align*}
\hat{\phi}, \hat{\psi}: \comp{V}{S} \EquivTo S \times \Ahat{n},
\end{align*}
and hence, composing, an isomorphism
\begin{align*}
\hat{\phi} \circ \hat{\psi}^{-1}: S \times \Ahat{n} \to S \times \Ahat{n},
\end{align*}
or equivalently, a continuous automorphism of $\Spec(R \series{n})$, i.e. an element $\omega_V$ of $G(S)$. Notice that $\omega_V$ lies in $K(S)$ precisely if the common \'etale neighbourhood is strict. 
\end{observation} 

Motivated by this observation we formulate the following:
\begin{prop} \label{prop: faithful morphism}
We have a natural morphism of prestacks:
\begin{align*}
F^{(\infty)}: \prestackVc{\infty} \longrightarrow BG_\triv.
\end{align*} 
\end{prop}
\begin{proof}
On objects, we define $F^{(\infty)}_S(S \times \A{n} \rightleftarrows S) \defeq (S \times G \to S)$.

On morphisms, we would like to set $F^{(\infty)}_S\left( [V, \phi, \psi] \right) \defeq \omega_V$ as in the above discussion. We need to show that this is well-defined and respects composition; both of these follow in a straightforward manner from the fact that taking completions of morphisms respects composition. \qed
\end{proof}

\begin{prop} \label{prop: faithful morphism for c}
Let $c \in \BN$. Then we have a natural morphism of prestacks
\begin{align*}
F^\cth: \prestackVc{c} \longrightarrow BG^\cth_\triv.
\end{align*}
\end{prop}
\begin{proof}
This morphism is defined analogously to $F^{(\infty)}$; the proof that it is well-defined on morphisms is immediate from the definition of $(c)$-equivalence, and the proof that it respects composition of morphisms is as above. \qed
\end{proof}

Restricting our attention to strict common \'etale neighbourhoods, we obtain the following analogous result:
\begin{prop}\label{prop: faithful morphism K}
For $c \in \BN$, we have morphisms of prestacks
\begin{align*}
F^{\prime \cth}: \prestackPVc{c} \longrightarrow BK^\cth_\triv.
\end{align*}
We also have a morphism
\begin{align*}
F^\prime: \prestackPVc{\infty} \longrightarrow BK_\triv.
\end{align*}
\end{prop}

Pulling back along the morphisms of Propositions \ref{prop: faithful morphism} and \ref{prop: faithful morphism for c} gives rise to functors
\begin{align*}
F^*: \Rep(G) \to \QCoh{\varietiesc{\infty}}, \\
F^{\cth,*}: \Rep(G^\cth) \to \QCoh{\varietiesc{c}}.
\end{align*}

In the subsequent sections, we study these functors. We will show that for finite $c$, $F^{\cth}$ is an equivalence of prestacks and hence $F^{\cth,*}$ is an equivalence of categories. On the other hand, we can show only that $F^*$ is a fully faithful embedding, but we give various characterisations of its essential image in Remark \ref{rmk: stack of analytic germs} and Section \ref{sec: convergent and ind-finite universal modules}.

\begin{rmk}[Remark on Harish-Chandra pairs]\label{remark: Harish-Chandra pairs}
Let $\fg$ denote the Lie algebra of $G$; it is equal to the Lie algebra $\Der{\On}$ of $k$-linear derivations of $\On$. The pair $(\fg, K = G_\red)$ forms a \emph{Harish-Chandra pair} (see \cite{BD1}, 2.9.7): $K$ is an affine group scheme; $\fg$ is a Lie algebra with a structure of Tate vector space; we have a continuous embedding $\Lie{K} \emb \fg$ of Lie algebras with open image; and we have an action of $K$ on $\fg$ which is compatible with the action of $\Lie{K}$ coming from the embedding. 

Given a Harish-Chandra pair $(\fg, K)$, we consider the category of $(\fg, K)$-modules: these are algebraic (and hence, by our earlier discussion, discrete) representations $V$  of $K$ equipped with an action of $\fg$ which is compatible with the induced action of $\Lie{K}$. When $K=G_\red$ as in our setting, this category is equivalent to the category of representations of $G$. 

Thus, for $G$ the group of automorphisms of the formal disc, the data of a representation of $G$ on a vector space $V$ is equivalent to the data of a representation of $K$ on $V$ together with a compatible action of $\fg=\Der{\On}$. This motivates one of the main results of this paper: we will see that we can associate to a representation of $K$ and a smooth $n$-dimensional variety $X$ an $\CO$-module $\SF$ on $X$. If our representation is in addition a representation of $G$, this amounts to having a compatible action of $\Der{\On}$, which in turn gives rise to a $\CD$-module structure on $\SF$. 
\end{rmk}

\section{Relative Artin approximation}
\label{sec: relative artin approximation}
In this section, our goal is to show that for finite $c$ the morphisms $F^\cth$ and $F^{\prime\cth}$ from Propositions \ref{prop: faithful morphism for c} and \ref{prop: faithful morphism K} are in fact isomorphisms of prestacks, and hence that we have 
\begin{align*}
BG^{\cth} \simeq \varietiesc{c}, \qquad BK^\cth \simeq \pointedvarietiesc{c}.
\end{align*} 
It suffices to show that the group homomorphisms $F_S^\cth$ and $F^{\prime\cth}_S$ are bijective---that is, given an automorphism of $S \times \Ahat{n}$ over $S$, we need to show that we can lift it to a common \'etale neighbourhood modulo $(c)$-equivalence; moreover we need to show that if the automorphism preserves the zero section $S \to S \times \Ahat{n}$ then we can lift it to a \emph{strict} common \'etale neighbourhood; and finally we need to show that in both cases the lifting is unique up to $(c)$-equivalence. 

\begin{rmk}
When $c = \infty$, the morphisms of prestacks are not isomorphisms: indeed, we will see that the corresponding group homomorphisms are injective, but not surjective. We will be able to use our understanding of these group homomorphisms to introduce yet another stack, the \emph{stack of formal germs of $n$-dimensional varieties}, which will be isomorphic to $BG$. See Remark \ref{rmk: stack of analytic germs}.
\end{rmk}

In \ref{subsec: statement of relative artin approximation}, we state the main result that we will need, which is a relative version of Artin's approximation theorem. The next two sections are devoted to the proof of this result: in \ref{subsec: preliminary material} we recall some important technical definitions and results, and in \ref{subsec: proof of relative artin approximation} we apply them to prove our result. Finally, in \ref{subsec: applications of relative artin approximation} we show how the relative version of Artin's approximation theorem implies that the morphisms $F^\cth$ and $F^{\prime\cth}$ are isomorphisms. 

\subsection{Statement of the main result}
\label{subsec: statement of relative artin approximation}
In the case that $S = \Spec{k}$ is a point, the results that we need follow from a well-known result of Artin:
\begin{thm}[Corollary 2.6, \cite{A}] \label{thm: artin approximation theorem} 
Let $X_1, X_2$ be schemes of finite type over $k$, and let $x_i \in X_i$ be points. Let $\fm_{x_1}$ denote the maximal ideal in the completed local ring $\widehat{\CO}_{X_1, x_1}$, and suppose there is an isomorphism of the formal neighbourhoods
\begin{align*}
\hat{\alpha}: \comp{(X_1)}{x_1} \EquivTo \comp{(X_2)}{x_2}
\end{align*}
over $k$. Then $X_1$ and $X_2$ are \'etale locally isomorphic: i.e., there is a common \'etale neighbourhood $((U,u), \phi, \psi)$ of $(X_i,x_i), i=1,2$.

Moreover, for any $c \in \BN$, we can choose $\phi$ and $\psi$ such that the resulting maps of completions satisfy
\begin{align*}
\hat{\psi} \circ \hat{\phi}^{-1} \equiv \hat{\alpha} \text{ (modulo $\fm_{x_1}^{c+1}$)}.
\end{align*}
\end{thm}

We are interested in the relative setting: $\pi_i\!: X_i \rightleftarrows S:\! \sigma_i\ ( i=1,2)$ are pointed $n$-dimensional families, and we ask when an isomorphism of the formal completions $\comp{(X_i)}{S}$ can be lifted to an actual morphism of schemes, at least \'etale locally. We are not able to prove a relative version of Theorem \ref{thm: artin approximation theorem} in full generality; however, we can show that it does hold when $X_1$ is a product $S \times Y$ for $Y$ any $n$-dimensional $k$-variety, and $\sigma_1$ is a constant section. This suffices for the applications we have in mind.

Therefore let us fix $S$ an affine scheme over $k$, and let $Y$ be a smooth $n$-dimensional variety over $k$, with $y \in Y$ some fixed point. Then we can form a pointed $n$-dimensional family $\pi_1: S \times Y \rightleftarrows S : \sigma_1$, where $\pi_1$ is the first projection, and $\sigma_1= \id_S \times i_y$ is induced by the inclusion of the point $y$ in $Y$. Let $\hat{Y} \defeq \comp{Y}{y}$ denote the completion of $Y$ at the point $y$, and note that $\comp{(S \times Y)}{S} \simeq S \times \hat{Y}$.

\begin{prop}[Relative Artin Approximation] 
\label{prop: relative artin approximation theorem}
Let $(\pi_2: X_2 \rightleftarrows S: \sigma_2)$ be any pointed $n$-dimensional family, and suppose that we have an isomorphism $\hat{\alpha}: S \times \hat{Y} \EquivTo \comp{(X_2)}{S}$ preserving both the projections to $S$ and the embeddings of $S$. Then there exists some affine \'etale neighbourhood $(U,u)\xrightarrow\phi (Y, y)$ that gives a strict split common \'etale neighbourhood of the $S$-families of $n$-dimensional varieties as follows:
\begin{equation}\label{diagram for relative artin approximation}
\begin{tikzpicture}[>=angle 90, baseline=(current bounding box.center)]
\matrix(c)[matrix of math nodes, row sep=2em, column sep=2em, text height=1.5ex, text depth=0.25ex]
{    &  V  & \\
 S \times Y && X_2 \\
  & S, & \\};
\path[->, font=\scriptsize]
(c-1-2) edge node[above left]{$\phi_S$} (c-2-1)
(c-1-2) edge node[above right]{$\psi_S$} (c-2-3)
(c-2-1) edge[bend left=10] (c-3-2)
(c-3-2) edge[bend left=15] (c-2-1)
(c-2-3) edge[bend right=10] (c-3-2)
(c-3-2) edge[bend right=15] (c-2-3)
(c-1-2) edge[bend left=10] node[right]{$\rho$} (c-3-2)
(c-3-2) edge[bend left=10] node[left]{$\tau$} (c-1-2);
\end{tikzpicture}
\end{equation}
where $V \subset S \times U$ is a Zariski open subset containing $S \times \left\{u\right\}$,  $\phi_S$ is the restriction of $\id_S \times \phi$ to $V$, and the section $\tau: S \hookrightarrow V$ is induced by the inclusion $i_u$ of the point $u$ in $U$. 

Furthermore, for any $c \in \BN$ this common \'etale neighbourhood can be chosen such that when we take completions along the closed embeddings of $S$, 
\begin{align}\label{compatibility with alpha}
\hat{\psi_S} \circ \hat{\phi_S}^{-1} \equiv \hat{\alpha} \text{ (modulo $\fm_S^{c+1}$)}.
\end{align}
(Here $\fm_S \subset \CO_{X_1}$ is the ideal sheaf corresponding to the closed embedding $\sigma_1: S \emb X_1$.) 
\end{prop}

The proof is very similar to the original proof of Theorem \ref{thm: artin approximation theorem} in \cite{A}; however we will give the generalisation explicitly below, in particular to demonstrate the equality (\ref{compatibility with alpha}), which is only implicit in \cite{A}. Both proofs rely on the notion of a functor locally of finite presentation, which we introduce in the next section.

\subsection{Preliminary material}
\label{subsec: preliminary material}
\begin{defn}
Let $Y$ be a scheme of finite type over $k$. A functor 
\begin{align*}
F: (\Sch_{/Y})^\op \to \Set
\end{align*}
is said to be \emph{locally of finite presentation} if it maps filtered limits of affine schemes over $Y$ to colimits of sets. That is, if $I$ is a filtered index category and $\left\{Y_i\right\}_{i\in I}$ is a diagram of affine schemes over $Y$, then 
\begin{align*}
\colim_{i \in I}{F(Y_i)} \simeq F(\limit_{i \in I}{Y_i}).
\end{align*}
\end{defn}

The following proposition gives a useful class of functors which are locally of finite presentation:
\begin{prop}[Proposition 2.3, \cite{A}]\label{prop: locally finite presentation}
Let
\begin{center}
\begin{tikzpicture}[>=angle 90]
\matrix(a)[matrix of math nodes, row sep=1em, column sep=1em, text height=1.5ex, text depth=0.25ex]
{Y_1 &   & Y_2\\
     & Z &    \\
     & X &    \\};
\path[->, font=\scriptsize]
 (a-1-1) edge (a-2-2)
 (a-1-3) edge (a-2-2)
 (a-2-2) edge (a-3-2);
\end{tikzpicture}
\end{center}
be a diagram of schemes over $X$ with $Z$ quasi-compact and quasi-separated, and $Y_i$ of finite presentation over $Z$ ($i=1,2$). Let $\Hom_Z(Y_1,Y_2)$ denote the functor:  
\begin{align*}
(\Sch_{/X})^\op &\to \Set\\
 T &\mapsto \Hom_{Z\times_X T}(Y_1 \times_X T, Y_2 \times_X T).  
\end{align*}
This functor is locally of finite presentation. 
\end{prop}

Now we give a proposition illustrating the usefulness of functors locally of finite presentation.

\begin{prop}[Corollary 2.2, \cite{A}]\label{thm: artin}
Fix a base scheme $Y$ over $k$, choose a point $y \in Y$, and let $\fm_y$ denote the maximal ideal of the completed local ring $\widehat{\CO}_{Y,y}$. Let $F: (\Sch_{/Y})^\op \to \Set$ be a contravariant functor locally of finite presentation, and assume we have $\hat{\xi} \in F(\hat{Y})$. Then for any $c\in \BN$, there exists an \' etale neighbourhood $(U,u)$ of $y$ in $Y$, and an element $\xi^\prime \in F(U)$ such that 
\begin{align}\label{artin-congruence}
\xi^\prime \equiv \hat{\xi} \text{ (modulo $\fm_y^{c+1}$)}.
\end{align}
\end{prop}

Here the congruence (\ref{artin-congruence}) is interpreted as follows: since $U$ is an \' etale neighbourhood of $Y$, we have a canonical morphism $\epsilon_{U}: \hat{Y} \to  U,$ inducing a function $F(U) \to F(\hat{Y})$. The content of (\ref{artin-congruence}) is that the images of $\xi^\prime$ and $\hat{\xi}$ agree after applying the canonical function
\begin{align*}
F(\hat{Y}) \to F(Y^{(c)}_y),
\end{align*}
where $Y^{(c)}_y$ denotes the $c$th infinitesimal neighbourhood of $y$ in $Y$. 

With this result in mind, we can prove Proposition \ref{prop: relative artin approximation theorem}.
\subsection{Proof of Proposition \ref{prop: relative artin approximation theorem}}
\label{subsec: proof of relative artin approximation}
Recalling the notation from the statement of the proposition, we define a functor
\begin{align*}
F: &(\Sch_{/Y})^\op \to \Set\\
 &T \mapsto \Hom_S(S\times T, X_2).
\end{align*}
Note that 
\begin{align*}
\Hom_S(S \times T, X_2) &\simeq \Hom_{S \times T}(S \times T, X_2 \times T)\\
 &\simeq \Hom_{(S \times Y) \times_Y T}\left( (S \times Y) \times _Y T, (X_2 \times Y) \times_Y T \right).
\end{align*}
Therefore, applying Proposition \ref{prop: locally finite presentation} to the diagram
\begin{center}
\begin{tikzpicture}[>=angle 90]
\matrix(a)[matrix of math nodes, row sep=1em, column sep=1em, text height=1.5ex, text depth=0.25ex]
{S \times Y &   & X_2 \times Y\\
     & S \times Y &    \\
     & Y, &    \\};
\path[->, font=\scriptsize]
 (a-1-1) edge (a-2-2)
 (a-1-3) edge (a-2-2)
 (a-2-2) edge (a-3-2);
\end{tikzpicture}
\end{center}
we conclude that $F$ is locally of finite presentation. 

In particular, $F(\hat{Y}) = \Hom_S(S \times \hat{Y}, X_2)$, and we have an element $\hat{\xi} \in F(\hat{Y})$ given by the composition:
\begin{align*}
S \times \hat{Y} \xrightarrow{\hat{\alpha}} \comp{(X_2)}{S} \xhookrightarrow{\epsilon_{X_2}} X_2.
\end{align*}

Now we apply Proposition \ref{thm: artin} and conclude that there exists an \'etale neighbourhood $\phi: (U,u) \to (Y, y)$ and an element $\xi^\prime \in F(U)$ approximating $\hat{\xi}$ modulo $\fm_y^{c+1}$. The element $\xi^\prime$ corresponds to a diagram of $S$-schemes:
\begin{center}
\begin{tikzpicture}[>=angle 90]
\matrix(f)[matrix of math nodes, row sep=2em, column sep=2em, text height=1.5ex, text depth=0.25ex]
{S \times U & & X_2\\
& S. & \\};
\path[->, font=\scriptsize]
(f-1-1) edge node[above]{$\xi^\prime$} (f-1-3)
(f-1-1) edge (f-2-2)
(f-1-3) edge (f-2-2);
\end{tikzpicture}
\end{center}

We can find an open neighbourhood $V$ of $S \times \left\{u\right\}$ in $S \times U$ such that $\xi^\prime$ is \'etale on $V$:  indeed, we know that $\xi^\prime$ induces an isomorphism $\comp{(S \times U)}{S} \simeq S \times \comp{U}{u} \xrightarrow{\sim} \comp{(X_2)}{S}$ because it agrees with the isomorphism $\hat{\alpha}$ on the $c$th infinitesimal neighbourhood. Therefore, for each $s \in S$, $\xi^\prime$ induces an isomorphism $\comp{(S \times U)}{(s,u)} \xrightarrow{\sim} \comp{(X_2)}{(\xi^\prime(s,u))}$, and so $\xi^\prime$ must be \'etale in some neighbourhood of $(s,u)$, since $\xi^\prime$ is locally of finite presentation. 

Having fixed such a neighbourhood $V$, we have a candidate for the left side of the diagram (\ref{diagram for relative artin approximation}) immediately:
\begin{center}
\begin{tikzpicture}[>=angle 90]
\matrix(e)[matrix of math nodes, row sep=2em, column sep=4em, text height=1.5ex, text depth=0.25ex]
{ & V\\
S \times Y&\\
& S. \\};
\path[->, font=\scriptsize]
(e-1-2) edge node[above,sloped]{$\phi_S = (\id_S \times \phi)|_{V}$} (e-2-1)
(e-2-1) edge[bend left=10] (e-3-2)
(e-3-2) edge[bend left=10] (e-2-1)
(e-1-2) edge[bend left=10] (e-3-2)
(e-3-2) edge[bend left=10] (e-1-2);
\end{tikzpicture}
\end{center}
Indeed, it is clear that $\phi_S$ is \'etale and respects the sections and the projections. 

To complete the right side of the diagram, it remains to show that the restriction $\psi_S$ of $\xi^\prime$ to $V$ commutes with the sections, i.e.
\begin{align*}
\xi^\prime \circ (\id_{S} \times i_u) = \sigma_2.
\end{align*}
Observe first that since $(U,u)$ is an \'etale neighbourhood of $(Y,y)$ we have a canonical morphism $\epsilon_U: \hat{Y} \to U$. Moreover, the fact that $\xi^\prime \equiv \hat{\xi}$ (mod $\mathfrak{m}^{c+1}$) amounts to the commutativity of the following diagram:
\begin{center}
\begin{tikzpicture}[>=angle 90]
\matrix(g)[matrix of math nodes, row sep=2em, column sep=2em, text height=2ex, text depth=0.25ex]
{                                     &S \times \hat{Y}&                               & S \times U &       \\
 S \times Y_{y}^{(c)}  &                             &                               &                         & X_2   \\
                                      &                             &  S \times \hat{Y}&                         &       \\};
\path[->, font=\scriptsize]
 (g-2-1) edge node[above, sloped]{$ \id_{S} \times \lambda_c$} (g-1-2)
 (g-1-2) edge node[above]{$ \id_{S} \times \epsilon_U $} (g-1-4)
 (g-1-4) edge node[above right]{$ \xi^\prime $} (g-2-5)

 (g-2-1) edge node[below, sloped]{$ \id_{S} \times \lambda_c$} (g-3-3)
 (g-3-3) edge node[below right]{$ \hat{\xi} $} (g-2-5);
\end{tikzpicture}
\end{center}
Now the result follows easily, once we note that the inclusion $i_u : \pt \emb U$ factors through the inclusion of the point $y$ in its $c$th infinitesimal neighbourhood $Y_{y}^{(c)}$ and its formal neighbourhood $\hat{Y}$ via the map $\epsilon_U$. Indeed, we have
\begin{align*}
\xi^\prime \circ \left(\id_S\times i_u\right) & = \xi^\prime \circ \left(\id_S \times (\epsilon_U \circ \lambda_c) \right) \circ \left(\id_S \times i_y^{(c)} \right) \\
 & = \hat{\xi} \circ \left(\id_{S} \times \lambda_c \right) \circ \left(\id_S \times i_y^{(c)} \right)\\
 & = \hat{\xi} \circ \left(\id_{S} \times \hat{i_y}\right) \\
 & = \epsilon_{X_2} \circ \hat{\alpha} \circ (\id_S \times \hat{i_y}) \\
 & = \epsilon_{X_2} \circ \hat{\sigma}_2\\
 & = \sigma_2.
\end{align*}

Finally, we have to check that $\hat{\psi}_S \circ \hat{\phi}_S^{-1} \equiv \hat{\alpha}$. Since $V$ is open in $S\times U$, it suffices to show that $\hat{\xi}^\prime \circ (\id_S \times \hat{\phi}^{-1}) \equiv \hat{\alpha}$. For this, we again use the compatibility of $\xi^\prime$ and $\hat{\alpha}$, which tells us that the following two compositions are equal:
\begin{align*}
S \times Y^{(c)}_y \to S \times \hat{Y} \xrightarrow{ \id_{S} \times \epsilon_U } S \times U \xrightarrow{\xi^\prime} X_2
\end{align*}
and
\begin{align*}
S \times Y^{(c)}_y \to S \times \hat{Y} \xrightarrow{\hat{\alpha}} \comp{(X_2)}{S} \xrightarrow{\epsilon_{X_2}} X_2. 
\end{align*}
Taking completions along $S$, we obtain
\begin{align*}
\hat{\xi^\prime} \circ (\id_S \times \hat{\epsilon}_U) \equiv \hat{\alpha} \text{ (modulo $\fm_S^{c+1}$)}.
\end{align*}
This completes the proof, because $\hat{\epsilon}_U = \hat{\phi}^{-1}$. \hfill $\square$

\subsection{Applications of the relative Artin approximation theorem}
\label{subsec: applications of relative artin approximation}
We shall need this theorem in the following two instances:

\begin{corollary}\label{cor: lifting formal isomorphisms}
\begin{enumerate}
\item Suppose we have an automorphism of $S \times \Ahat{n}$ over $S$ which does not necessarily preserve the section $\hat{z}$. Then for any $c \in \BN$ it can be lifted to a common \'etale neighbourhood $(V, \phi_S, \psi_S)$ such that $\psi_S^\cth \circ (\phi_S^\cth)^{-1} = \alpha^\cth$, as morphisms on the $c$th infinitesimal neighbourhoods. 
\item Suppose we have an automorphism of $S \times \Ahat{n}$ preserving the section $\hat{z}$. It can be lifted to a common \'etale neighbourhood as above which is also strict.
\end{enumerate}
\end{corollary}

\begin{proof}
To prove the first part, we apply Proposition \ref{prop: relative artin approximation theorem} to the following diagram:
\begin{center}
\begin{tikzpicture}[>=angle 90,bij/.style={above,sloped,inner sep=0.5pt}]
\matrix(b)[matrix of math nodes, row sep=2em, column sep=2em, text height=1.5ex, text depth=0.25ex]
{S \times \Ahat{n}&      & S \times \Ahat{n}\\
                  & S    &  \\};
\path[->, font=\scriptsize]
(b-1-1) edge node[bij]{$\sim$} node[below]{$\hat{\alpha}$} (b-1-3)
(b-1-1) edge[bend left=10] node[above]{$\pi$}(b-2-2)
(b-2-2) edge[bend left=15] node[below]{$\hat{z}$}(b-1-1)
(b-1-3) edge[bend right=10] node[above]{$\pi$}(b-2-2)
(b-2-2) edge[bend right=15] node[below]{$\hat{z}_2$}(b-1-3);
\end{tikzpicture}
\end{center} 
where $z_2= (\id_S \times \epsilon_{\Ahat{n}}) \circ \hat{\alpha} \circ \hat{z}$. The diagram commutes by construction, and Proposition \ref{prop: relative artin approximation theorem}  yields a strict common \'etale neighbourhood:
\begin{center}
\begin{tikzpicture}[>=angle 90]
\matrix(c)[matrix of math nodes, row sep=2em, column sep=2em, text height=1.5ex, text depth=0.25ex]
{    &  V  & \\
 S \times \A{n} && S \times \A{n} \\
  & S & \\};
\path[->, font=\scriptsize]
(c-1-2) edge node[above left]{$\phi_S$} (c-2-1)
(c-1-2) edge node[above right]{$\psi_S$} (c-2-3)
(c-2-1) edge[bend left=10] node[above]{$\pi$}(c-3-2)
(c-3-2) edge[bend left=15] node[below]{$z$}(c-2-1)
(c-2-3) edge[bend right=10] node[above]{$\pi$}(c-3-2)
(c-3-2) edge[bend right=15] node[below]{$z_2$}(c-2-3)
(c-1-2) edge[bend left=10] node[right]{$\rho$}(c-3-2)
(c-3-2) edge[bend left=10] node[left]{$\tau$}(c-1-2);
\end{tikzpicture}
\end{center}
such that $\psi_S^\cth \circ (\phi_S^\cth)^{-1} = \alpha^\cth$. Since $\psi_S \circ \tau = z_2$ and $z_2 \circ \redEmb{S} = z \circ \redEmb{S}$, this gives a common \'etale neighbourhood of the trivial pointed $n$-dimensional family. 

To prove the second part, notice that the additional assumption that $\hat{\alpha}$ preserves the section is equivalent to the statement that $z_2 = z$. It follows that the common \'etale neighbourhood is strict, as required. \qed
\end{proof}

Applying Propositions \ref{prop: relative artin approximation theorem}, \ref{prop: locally finite presentation}, and \ref{thm: artin}, we can also prove the following useful results:

\begin{lemma} \label{lemma: ubiquity of split neighbourhoods}
Every common \'etale neighbourhood of the trivial pointed family over $S$ is $(\infty)$-equivalent (and hence $(c)$-equivalent for any $c \in \BN$) to a split common \'etale neighbourhood.  
\end{lemma}

\begin{proof}
Let $\rho: V \rightleftarrows S: \tau$ be a pointed $n$-dimensional family with \'etale maps $\phi, \psi: V/S \to (S \times \A{n})/S$ giving a common \'etale neighbourhood. Using classical results on standard smooth and \'etale $n$-dimensional morphisms (see for example \cite{Mil} 3.14), we can show that for every $s \in S$ there is a Zariski open neighbourhood $T$ of $s$ in $S$, and an open neighbourhood $U$ of $\tau(s) \in \rho^{-1}(T)$ such that we have a commutative diagram as follows, with $\lambda$ \'etale:
\begin{center}
\begin{tikzpicture}[>=angle 90]
\matrix(f)[matrix of math nodes, row sep=3em, column sep=3em, text height=1.5ex, text depth=0.25ex]
{T \times \A{n} & U & V \\
T             & T & S. \\};
\path[-]
 (f-2-1) edge[double, double distance = 2pt] (f-2-2);
\path[->, font=\scriptsize]
 (f-1-2) edge node[above]{$\lambda$} (f-1-1)
 
 (f-1-1) edge[bend left=10] node[right]{$\pr_T$} (f-2-1)
 (f-2-1) edge[bend left=10] node[left]{$z$} (f-1-1)
 
 (f-1-2) edge[bend left=10] node[right]{$\rho_{|U}$} (f-2-2)
 (f-2-2) edge[bend left=10] node[left]{$\tau_{|T}$} (f-1-2)
 
 (f-1-3) edge[bend left=10] node[right]{$\rho$} (f-2-3)
 (f-2-3) edge[bend left=10] node[left]{$\tau$} (f-1-3);
\path[right hook->, font=\scriptsize]
 (f-1-2) edge (f-1-3)
 (f-2-2) edge (f-2-3); 
\end{tikzpicture}
\end{center}
By Proposition \ref{prop: relative artin approximation theorem} we can lift $\hat{\lambda}^{-1} : T \times \Ahat{n} \to U$: that is, we can find $(W,w)$ \'etale over $(\A{n}, 0)$ and $\lambda^\prime : (T \times W)/T \to U/T$. Moreover, $\lambda^\prime$ is \'etale on some open neighbourhood $V^\prime$ of $T \times \{w\}$, and $(V^\prime, \phi \circ \lambda^\prime_{|V^\prime}, \psi \circ \lambda^\prime_{|V^\prime})$ is a split common \'etale neighbourhood, similar to the restriction of $(V, \phi, \psi)$ to $T$. \qed
\end{proof}

\begin{lemma}\label{lemma: uniqueness of liftings}
Let $(y,Y)$ be a pointed $n$-dimensional variety, and $\pi: X \rightleftarrows S: \sigma$ be a pointed $n$-dimensional family over $S$. Suppose that we have two morphisms $\phi, \psi : Y \times S \to X$ compatible with the projections, and compatible with the sections on $S_\red$, such that $\hat{\phi} = \hat{\psi} : S \times \hat{Y} \EquivTo \comp{X}{S}$. Then there exists some $f_X: U \to Y$ \'etale such that $\phi \circ (f_X, \id_S) = \psi \circ (f_X, \id_S)$. In particular, the liftings provided by Corollary \ref{cor: lifting formal isomorphisms} are unique up to equivalence. 
\end{lemma}

\begin{proof}
We apply Proposition \ref{prop: locally finite presentation} to the diagram 
\begin{center}
\begin{tikzpicture}[>=angle 90]
\matrix(a)[matrix of math nodes, row sep=2em, column sep=2em, text height=1.5ex, text depth=0.25ex]
{Y \times S &                                  & (Y \times S) \times_{X} (Y \times S)\\
            & (Y \times S) \times (Y \times S) &    \\
            & Y &    \\};
\path[->, font=\scriptsize]
 (a-1-1) edge node[below left]{$\Delta$} (a-2-2)
 (a-1-3) edge node[below right]{$(p_1, p_2)$}(a-2-2)
 (a-2-2) edge node[left]{$\pr^{YSYS}_1$}(a-3-2);
\end{tikzpicture}
\end{center}
and obtain that the functor 
\begin{align*}
F: \Sch_{/Y}^\op &\to \Set\\
(T/Y) & \mapsto \Hom_{T \times_Y (Y \times S \times Y \times S)} \left(T \times_Y (Y \times S), T \times_Y(Y \times S) \times_{X} (Y \times S) \right)
\end{align*}
is locally of finite presentation. (Here $\Delta$ is the diagonal morphism, $p_1$ and $p_2$ are the projections from $(Y \times S) \times_{X} (Y \times S)$ to $Y \times S$ satisfying $\phi \circ p_1 = \psi \circ p_2$, and $\pr^{YSYS}_1$ is the projection onto the first $Y$ factor.)

The fact that $\hat{\phi} = \hat{\psi}$ implies that $\phi \circ (\epsilon_{Y}, \id_S) = \psi \circ (\epsilon_Y, \id_S)$ as maps from $\hat{Y} \times S$ to $X$. Hence we obtain a map from $\hat{Y} \times S$ to $(Y \times S) \times_X (Y \times S)$ over $Y$, which finally gives us a map $\hat{Y} \times S \to \hat{Y} \times_Y(Y \times S) \times_{X} (Y \times S)$ corresponding to an element $\hat{\xi}$ of $F(\hat{Y})$. 

Since $F$ is locally of finite presentation, Proposition \ref{thm: artin} applies, and we obtain an \'etale neighbourhood $f_X: (U, u) \to (Y, y)$ and an element $\xi^\prime \in F(U)$ which agrees with $\hat{\xi}$ modulo $\fm_y^2$. 

Now we remark that for any $Y$-scheme $f:T \to Y$, $F(T)$ is non-empty if and only if $\phi \circ (f, \id_S) = \psi \circ (f, \id_S)$ (and moreover, in that case $F(T)$ consists of a single point). 

Indeed, $F(T)$ is a subset of $\Hom \left(T \times S, T \times_Y(Y \times S) \times_{X} (Y \times S) \right)$. A map $\alpha: T \times S \to T \times_Y(Y \times S) \times_{X} (Y \times S)$ is given by three maps
\begin{align*}
\alpha_1: T \times S &\to T\\
\alpha_i: T \times S &\to Y \times S,\ i=2,3,
\end{align*}
satisfying 
\begin{align} \label{alpha properties}
f \circ \alpha_1 &= \pr_Y^{YS} \circ \alpha_2; \nonumber\\
\phi \circ \alpha_2 &= \psi \circ \alpha_3.
\end{align}
This $\alpha$ is an element of $F(T)$ if and only if it is compatible with the maps $T \times S \to T \times_Y (Y \times S \times Y \times S)$ and $T \times_Y (Y \times _S) \times_X (Y \times S)$, or equivalently if and only if
\begin{align*}
\alpha_1 &= \pr^{TS}_T;\\
\alpha_2 &= (f, \id_S);\\
\alpha_3 &= (f, \id_S).
\end{align*}
Therefore, the only possible candidate for an element of $F(T)$ corresponds to the triple $(\pr^{TS}_T, (f, \id_S), (f, \id_S))$, which only gives a map $\alpha$ in the case that the equations (\ref{alpha properties}) are satisfied. This amounts exactly to the condition $\phi \circ (f, \id_S) = \psi \circ (f, id_S)$.

It follows that the existence of $\xi^\prime \in F(U)$ means that $f_X:U \to Y$ gives the desired \'etale neighbourhood. \qed
\end{proof}

Combining Corollary \ref{cor: lifting formal isomorphisms} and Lemmas \ref{lemma: ubiquity of split neighbourhoods} and \ref{lemma: uniqueness of liftings}, we obtain
\begin{prop}\label{prop: the group homomorphisms are isomorphisms}
For any affine base-scheme $S$ and any $c \in \BN \cup \{\infty\}$, the group homomorphisms 
\begin{align*}
F^{\cth}_S : \Aut_{\left(\prestackVc{c}(S)\right)}(S \times G^{\cth} \to S) \to G^\cth(S) \\
F^{\prime\cth}_S : \Aut_{\left(\prestackPVc{c}(S)\right)}(S \times K^{\cth} \to S) \to K^\cth(S) 
\end{align*}
of Propositions \ref{prop: faithful morphism for c} and \ref{prop: faithful morphism K} are injective.\footnote{By a slight abuse of notation, we understand $K^{(\infty)}$ and $G^{(\infty)}$ to mean $K$ and $G$ respectively.} When $c$ is finite, the homomorphisms are surjective as well. 
\end{prop}
\begin{proof}
For finite $c$, injectivity follows immediately from the definition of $(c)$-equiva\-lence, and Corollary \ref{cor: lifting formal isomorphisms} shows that the homomorphisms are surjective. 

Now let $c = \infty$. Lemma \ref{lemma: uniqueness of liftings} implies that the homomorphisms are injective when restricted to the set of automorphisms represented by split common \'etale neighbourhoods. By Lemma \ref{lemma: ubiquity of split neighbourhoods}, this implies that they are injective. \qed
\end{proof}

It follows that for $c \in \BN$, $F^\cth$ and $F^{\prime \cth}$ give equivalences of prestacks, and using the uniqueness of stackification, we obtain the following:
\begin{thm} \label{thm: isomorphisms for order c stacks}
Let $c \in \BN$. We have isomorphisms of stacks:
\begin{align*}
\varietiesc{c} &\EquivTo BG^\cth\\
\pointedvarietiesc{c} &\EquivTo BK^\cth.
\end{align*}
\end{thm}

\begin{rmk}\label{rmk: stack of analytic germs}
We also have morphisms
\begin{align*}
\varietiesc{\infty} &\to BG\\
\pointedvarietiesc{\infty} &\to BK,
\end{align*}
but they are not isomorphisms. There are two approaches to modify the stacks involved to obtain an equivalence: we can either enlarge the automorphism groups of the stacks on the left hand side, or we can restrict the automorphism groups of those on the right hand side. 

\begin{enumerate}
\item Motivated by the above discussion, we see that we can define yet another stack of germs, this one equivalent to the classifying stack $BG$. It is the stackification of the prestack $\prestackV$ which again sends a test scheme $S$ to a groupoid whose only object is the trivial pointed $n$-dimensional family $S \times \A{n} \rightleftarrows S$. However, we would like the morphisms of this groupoid to correspond to elements of $G(S)$, i.e. automorphisms $\hat{\alpha}$ of $S \times \Ahat{n}$. Proposition \ref{prop: relative artin approximation theorem} tells us that we can represent such an automorphism by a \emph{sequence} of common \'etale neighbourhoods $\{(U_c, \phi_c, \psi_c)\}_{c=1}^\infty$ such that for each $c$, 
\begin{align*}
\psi_S^\cth \circ (\phi_S^\cth)^{-1} = \alpha^\cth.
\end{align*}
It follows that $(U_c, \phi_c, \psi_c)$ is uniquely determined up to $(c)$-equivalence, i.e. as a morphism in $\prestackVc{c}(S)$. Moreover, the sequence $\{(U_c, \phi_c, \psi_c)\}_{c=1}^\infty$ determines $\hat{\alpha}$.

That is, we should define $\varieties$ to be
\begin{align*}
\lim_{c \in\BN} \varietiesc{c}.
\end{align*}
We will call this the \emph{stack of formal germs of $n$-dimensional varieties}. We can similarly define
\begin{align*}
\pointedvarieties \defeq \lim_{c \in \BN} \pointedvarietiesc{c}.
\end{align*}

\item Let $G^\et$ be the group-valued prestack sending a test scheme $S$ to the image of $\Aut_{\prestackVc{\infty}} \left(S \times \A{n} \rightleftarrows S \right)$ in $G(S)$ under $F_S$: i.e. this is the group of all automorphisms of the formal disc which can be lifted precisely to common \'etale neighbourhoods. Then we have that
\begin{align*}
\varietiesc{\infty} \EquivTo BG^\et.
\end{align*}
It follows from Corollary \ref{cor: lifting formal isomorphisms} that $G^\et$ is dense in $G$, and hence we will be able to show that restriction from $G$ to $G^\et$ gives a fully faithful embedding $\Res_{G, G^\et}: \Rep(G) \emb \Rep(G^\et)$ (see Corollary \ref{cor: restriction is an embedding for G}). 

Similarly, we can define a sub-group $K^\et$ of $K$ such that
\begin{align*}
\pointedvarietiesc{\infty} \EquivTo BK^\et.
\end{align*}
In Section \ref{sec: groups of etale automorphisms} we study the representation theory of these group-valued prestacks $G^\et$ and $K^\et$. 
\end{enumerate}
\end{rmk}

\section{Groups of \'etale automorphisms and their representation theory}
\label{sec: groups of etale automorphisms}
This section is about the top rows of the main diagram (Figure 1), when $c= \infty$. We have the following stacks
\begin{align*}
\varietiesc{\infty} \EquivTo BG^\et \emb BG,
\end{align*}
giving rise to the following categories
\begin{align*}
\Rep(G) \to \Rep(G^\et) \EquivTo \QCoh{\varietiesc{\infty}};
\end{align*}
the composition of the morphisms is the functor $F^*$. We see that in order to understand the relationship between quasi-coherent sheaves on $\varietiesc{\infty}$ and representations of $G$, it suffices to study the restriction functor $\Res_{G, G^\et}: \Rep(G) \to \Rep(G^\et)$. In fact we will begin by working with the group $K$; then we will apply our results to the group $G$ as well. 

In \ref{subsec: Unipotent subgroups and polynomial submonoids}, we define some subgroups and submonoids of $G, G^\et, K,$ and $K^\et$. These will be technically easier to work with than the full groups, as we will see in the subsequent sections. In \ref{subsec: representations of K-etale} we study the restriction functor $\Res_{K, K^\et}$, and show that it gives an equivalence of the subcategories of finite-dimensional representations. Since $K$ is an affine group scheme, all of its representations are locally finite, from which we conclude that the functor $\Res_{K, K^\et}$ is fully faithful, with essential image the subcategory of locally finite representations of $K^\et$. 

We do not know whether there are any representations of $K^\et$ which are \emph{not} locally finite, but in \ref{subsec: non-locally finite representations} we give an example of a pair $H \supset H^\prime$ of a pro-algebraic group $H$ containing a dense group-valued subprestack  $H^\prime$, such that $H^\prime$ has representations which are not locally finite, and hence do not extend to representations of $H$. 

Finally, in \ref{subsec: representations of G-etale} we study the restriction functor 
\begin{align*}
\Res_{G, G^\et} : \Rep(G) \to \Rep(G^\et).
\end{align*}
Analogously to \ref{subsec: representations of K-etale} we show that it is fully faithful, and characterise its essential image as those representations of $G^\et$ satisfying a suitable finiteness condition. 

\subsection{Unipotent subgroups and polynomial submonoids}
\label{subsec: Unipotent subgroups and polynomial submonoids}
Recall from definition \ref{defn: unipotent subgroups of K} that the pro-unipotent group $K_u$ is the kernel of the natural map $K \to GL_n$; analogously, we define the sub-group-valued-prestack $K^\et_u$ of $K^\et$ to be the kernel of the the restriction of this map to $K^\et$. We have $K^\et = GL_n \ltimes K^\et_u$ as group-valued prestacks.

Recall also that in the proof of Proposition \ref{prop: representations of G are continuous} we defined for any $c$ a group-valued prestack $N_c$ by setting $N_c(S)$ to be the kernel of the homomorphism $G(S) \surj G^\cth(S)$. We noted then that $N_c(S)$ is also the kernel of the homomorphism $K(S) \surj K^\cth(S)$; 
now we remark in addition that it is contained in $K_u(S)$, and that its intersection with $K^\et(S)$ is the kernel of the maps
\begin{align*}
K^\et(S) \surj K^\cth(S), \qquad K^\et_u \surj K^\cth_u(S).
\end{align*}
(That these maps are surjective is a consequence of Corollary \ref{cor: lifting formal isomorphisms}.)

The prestack $N_c$ is an affine group scheme of infinite type. 
\begin{eg}
In the case $n=1$, $N_c$ parametrises automorphisms $\rho: R\seriest \to R\seriest$ of the form
\begin{align*}
\rho: t \mapsto t+ r_{c+1}t^{c+1} + r_{c+2}t^{c+2} + \ldots.
\end{align*}
\end{eg}

\begin{defn}
Consider for each $\Spec{R}$ the set $M(\Spec{R})\subset K(\Spec{R})$ of \emph{polynomial automorphisms}: these are automorphisms 
\begin{align*}
\rho: R\series{n} \to R\series{n}
\end{align*}
such that for each $k = 1, \ldots, n$, $\rho(t_k)$ is a polynomial in the variables $\{t_j\}$ with no constant term, rather than a power series. This defines a prestack $M \emb K^\et \emb K$.
\end{defn}
This is a monoid rather than a group-valued prestack, since it is not closed under taking inverses; however, it is in some ways easier to work with than $K^\et$ in that it is an indscheme: 
\begin{align*}
M = \colim_{\alpha \in \BN} M_\alpha,
\end{align*}
where $M_\alpha$ classifies polynomial automorphisms of degree at most $\alpha$. (Note that $M_\alpha$ is a scheme, but is not even a monoid, since composing two polynomial automorphisms of degree $\alpha$ gives a polynomial automorphism of degree $\alpha^2$.)

Similarly, we have the unipotent version $M_u = \colim_{\alpha \in \BN} M_{u, \alpha}$, where 
\begin{align*}
M_{u, \alpha} = \Spec{k [a^k_J]_{1<|J|\le \alpha}}
\end{align*}
classifies unipotent polynomial automorphisms of degree at most $\alpha$. Although the $M_{u, \alpha}$ are not closed under composition, they are still closed under the action of $\BG_m \emb GL_n$ by conjugation, and hence each algebra $k [a^k_J]_{1 < |J| \le \alpha}$ is still graded, with $\deg(a^k_J) = |J| -1$. 

We also have a monoid of polynomial automorphisms $M^G$ in $G$: 
\begin{align*}
M^G = \colim_{\alpha \in \BN} M^G_\alpha, 
\end{align*}
where $M^G_\alpha$ parametrises polynomial automorphisms of the formal disc of degree at most $\alpha$, whose constant terms are nilpotent. It is itself an indscheme:
\begin{align*}
M^G_\alpha = \colim_N M^G_{\alpha,N},
\end{align*}
where $M^G_{\alpha,N}$ parametrises only those polynomial automorphisms whose constant terms all satisfy $a^N =0$. As a scheme,
\begin{align*}
M^G_{\alpha,N} = \Spec{k[a^k_J]_{|J| \le \alpha}/((a^k_{\underline{0}}) ^N)}.
\end{align*}  

Finally, notice that for any $c$, the group scheme $N_c$ also contains a monoid $M^{N_c} = \colim_{\alpha \ge c+1} M^{N_c}_\alpha$ of polynomial automorphisms. 

Having established this notation, we can prove the following:

\begin{lemma}\label{restriction is injective}
Let $H \in \{ K, K_u, G, N_c \}$. Then the inclusion 
\begin{align*}
H^\et \emb H
\end{align*}
induces a map of sets
\begin{align*}
\Hom_{\PreStk}(H, \A{1}) \to \Hom_{\PreStk}(H^\et, \A{1}).
\end{align*}
This map is injective.

Moreover, the same is true of the restriction map
\begin{align*}
\Hom_{\PreStk}(H \times H, \A{1}) \to \Hom_{\PreStk}(H^\et \times H^\et, \A{1}).
\end{align*}
\end{lemma}

\begin{proof}
The intuition behind this statement is that Artin's approximation theorem tells us that $H^\et$ is dense in $H$. In order to give a rigorous proof it is useful to restrict further to the monoid introduced above: this is still dense in $H$ and we can exploit its indscheme structure to study its functions. It is clearly sufficient to show that restriction from $H$ to the monoid, which we'll denote by $M^H$, (respectively from $H \times H$ to $M^H \times M^H$) is injective: if two maps agree on $H^\et$, they certainly agree on $M^H$. We will carry out the proof for the case $H=G$; the remaining cases are very similar (and where different, simpler). 

By the universal property of colimits, a map $\phi: M^G \to \A{1}$ is given by a compatible family of polynomials 
\begin{align*}
\left( f_{\alpha,N} \in k[a^k_J]_{|J| \le \alpha}/((a^k_{\underline{0}})^N) \right)_{\alpha,N}.
\end{align*}

For fixed $N$, the compatibility between $f_{\alpha,N}$ is the following: if $\alpha_1 > \alpha_2$, we require that the polynomial obtained from $f_{\alpha_1, N}$ by setting $a^K_J =0$ for all $J$ with $|J| > \alpha_2$ be equal to $f_{\alpha_2, N}$. 

Similarly, a map $\psi: G \to \A{1}$ is defined by a compatible family of polynomials
\begin{align*}
\left( g_{N} \in k[a^k_J]_{J \in \BZ_{\ge 0} ^n} / ((a^k_{\underline{0}})^N) \right) _N.
\end{align*}

Let $\phi^1=(f^1_{\alpha,N})$ and $\phi^2=(f^2_{\alpha,N})$ be the restriction of two maps $\psi^1=(g^1_N)$ and $\psi^2=(g^2_N)$ to $M$. For any $\alpha$, the polynomial $f^i_{\alpha,N}$ is obtained from $g^i_N$ by setting $a^k_J = 0$ for all $J$ with $|J| > \alpha$. It is clear that if $f^1_{\alpha,N} = f^2_{\alpha,N}$ for every $\alpha$, then $g^1_N = g^2_N$. That is, if $\phi^1 = \phi^2$, then $\psi^1 = \psi^2$, and so the restriction map is injective, as required. 

The argument for $G^\et \times G^\et \emb G \times G$ is similar. \qed
\end{proof}

\subsection{Representations of \texorpdfstring{$K^\et$}{K-etale}}
\label{subsec: representations of K-etale}
\begin{thm}\label{theorem: K extensions}
Let $V$ be a finite-dimensional vector space, and let $\bR: K^\et \to GL_V$ be a representation of $K^\et$ on $V$. Then the natural transformation $\bR$ extends uniquely to a representation
\begin{align*}
\overline{\bR}: K \to GL_V.
\end{align*}
Equivalently, there exists some $c$ such that $\bR$ factors through the finite-di\-men\-sion\-al quotient $K^\cth$; then the extension $\overline{\bR}$ is defined via the quotient $K \surj K^\cth$. 
\end{thm}

\begin{proof}
Choose a basis $\{v_1, \ldots, v_m\}$ of $V$ such that the action of $\BG_m \emb GL_n \emb K^\et$ is diagonal:
\begin{align*}
z \mapsto \left( 
\begin{array}{cccc}
z^{d_1} & & \\
        & \ddots & \\
        &        & z^{d_m}
         \end{array}
         \right)
\end{align*}
with $d_1 \le d_2 \le \ldots \le d_m$ an increasing sequence of integers.

Now consider the restriction of $\bR$ to the monoid $M_u = \colim_{\alpha}M_{u, \alpha} \emb K^\et_u$, and let $\bR_{ij}$ denote the $(i,j)$th matrix coefficient with respect to the basis $\{v_1, \ldots, v_m\}$:
\begin{align*}
\bR_{ij}: M_u \to \A{1}.
\end{align*}
As in the proof of Proposition \ref{restriction is injective}, $\bR_{ij}$ is given by an infinite family of polynomials
\begin{align*}
\left\{f_{ij, \alpha} \in k[a^k_J]_{1< |J| \le \alpha} \right\}_{\alpha},
\end{align*}
satisfying the compatibility condition
\begin{align*}
f_{ij, \alpha +1} |_{a^k_J \equiv 0, |J|= \alpha +1} = f_{ij, \alpha}.
\end{align*}
\begin{description}
\item[Step 1]
Let $\alpha_0 = d_m - d_1 + 1$. We will show that in fact the polynomials $f_{ij, \alpha}$ depend only on the variables $\{a_1, \ldots, a_{\alpha_0}\}$. In particular, 
\begin{align*}
f_{ij, \alpha_0 + 1} = f_{ij, \alpha_0 + 2} = \cdots,
\end{align*}
and the function $\bR_{ij}$ is the restriction of a function $\overline{\bR_{ij}} : K_u \to \A{1}$, corresponding to the polynomial $f_{ij, \alpha_0 +1} \in k[a^k_J]_{1 < |J|}$.

To prove this claim, let $z \in \BG_m(k)$ be an arbitrary $k$-point, and recall that conjugation by $z$ gives maps
\begin{align*}
\gamma_z: M_u \to M_u, \qquad \gamma_{z, \alpha}: M_{u, \alpha} \to M_{u, \alpha},
\end{align*}
or equivalently maps $\gamma_z^\sharp$ and $\gamma_{z, \alpha}^\sharp$ on the corresponding algebras of functions. 

We know that for any $k$-algebra $R$ and for any $m \in M_u(\Spec{R})$, we have 
\begin{align*}
\bR(z m z^{-1}) = \bR(z) \bR(m) \bR(z)^{-1}.
\end{align*}
In terms of matrix coefficients, this implies that
\begin{align*}
\bR_{ij} \circ \gamma_z = z^{d_i - d_j} \bR_{ij},
\end{align*}
or equivalently that for every $\alpha$,
\begin{align*}
\gamma_{z, \alpha}^\sharp (f_{ij, \alpha}) = z^{d_i - d_j} f_{ij, \alpha}. 
\end{align*}
It follows that $f_{ij, \alpha}$ is homogeneous of degree $d_i - d_j$, and hence cannot depend on any variable $a^k_J$ with $|J|> d_i - d_j +1$. 

Allowing $(i, j)$ to vary, we obtain the global bound $\alpha_0 = d_m - d_1 + 1 = \max\left\{d_i - d_j +1 \right\}$ on the degree of the variables appearing in the matrix coefficients of $\bR$.

\item[Step 2] We have shown that each matrix coefficient extends to a function on $K_u$, but it is not a priori clear that these assemble to give a representation $\overline{\bR}$ of $K_u$, which in turn extends to all of $K$. We show this now. 

By Step 1 and Lemma \ref{restriction is injective}, if we take $c > \alpha_0$, the restriction of $\bR$ to $N_c^\et$ is constant. That is, $N_c^\et$ is in the kernel of the representation $\bR$, and hence we have a factorisation
\begin{center}
\begin{tikzpicture}[>=angle 90]
\matrix(a)[matrix of math nodes, row sep=3em, column sep=1em, text height=1.5ex, text depth=0.25ex]
{ K^\et &          & GL_V. \\
          & K^\cth &      \\};
\path[->, font=\scriptsize]
 (a-1-1) edge node[above]{$\bR$} (a-1-3)
 (a-2-2) edge[dashed] node[below right]{$\bR^\cth$} (a-1-3);
 
\path[->>]
 (a-1-1) edge (a-2-2);
\end{tikzpicture}
\end{center}

So we can define an extension of $\bR: K^\et \to GL_V$ to all of $K$ by defining $\overline{\bR}$ to be the composition
\begin{align*}
K \surj K^\cth \to GL_V.
\end{align*}
By applying Lemma \ref{restriction is injective} to each of the matrix coefficients of $\overline{\bR}$, we see that this extension is unique. \qed
\end{description}
\end{proof}

\begin{corollary}\label{cor: restriction is an embedding for K}
The restriction functor 
\begin{align*}
\Res_{K, K^\et}: \Rep(K) \to \Rep(K^\et)
\end{align*}
is fully faithful. Its essential image is the full subcategory $\Rep^\lf(K^\et)$ of locally finite representations of $K^\et$. 
\end{corollary}

\begin{proof}
Let $(V,\overline{\bR})$ and $(W,\overline{\bS})$ be two representations of $K$, and let $(V, \bR), (W, \bS)$ denote their image in $\Rep(K^\et)$. We consider the map
\begin{align*}
\Hom_K(V, W) \to \Hom_{K^\et}(V, W).
\end{align*}
It is clear that this is injective. To see that it is surjective, notice that a map $V \to W$ compatible with $\bR$ and $\bS$ is necessarily compatible with the extensions $\overline{\bR}$ and $\overline{\bS}$ of the representations to $K$: this amounts to the fact that for sufficiently large $c$, we have a commutative diagram
\begin{center}
\begin{tikzpicture}[>=angle 90]
\matrix(b)[matrix of math nodes, row sep=3em, column sep=3em, text height=1.5ex, text depth=0.25ex] 
{ K^\cth & GL_V \\
  GL_W   & \End(V, W).\\};
\path[->, font=\scriptsize]
(b-1-1) edge node[above]{$\bR^\cth$} (b-1-2)
        edge node[left]{$\bS^\cth$} (b-2-1)
(b-2-1) edge (b-2-2)
(b-1-2) edge (b-2-2);
\end{tikzpicture}
\end{center}
So $\Res_{K,K^\et}$ is fully faithful as claimed. 

To identify the essential image, first recall that because $K$ is an affine group scheme, all of its representations are locally finite: any representation $V$ can be written as a union of finite-dimensional representations. This clearly still gives a decomposition of $V$ when we restrict to the action of $K^\et$, and so the essential image is a subcategory of $\Rep^\lf(K^\et)$. 

On the other hand, suppose that $(V, \bR) \in \Rep^\lf(K^\et)$. Then we can write $V = \bigcup_i V_i$, where $V_i \subset V$ is a finite-dimensional subrepresentation of $K^\et$. Let $\bR_i: K^\et \to GL_{V_i}$ denote the restriction of $\bR$ to the subspace $V_i$. Since $V_i$ is finite-dimensional, Theorem \ref{theorem: K extensions} provides us with a unique extension of $\bR_i$:
\begin{align*}
\overline{\bR_i}: K \to GL_{V_i}.
\end{align*}
The uniqueness of these extensions means that they agree on intersections $V_i \cap V_j$, and hence give a representation
\begin{align*}
\overline{\bR}: K \to GL_V.
\end{align*}
By construction, $\Res_{K, K^\et} (V, \overline{\bR}) = (V, \bR)$ and hence $(V, \bR)$ is indeed in the essential image of $\Res_{K, K^\et}$. \qed
\end{proof}

\begin{rmk}
At the time of writing, we do not know whether there exist any representations of $K^\et$ which are \emph{not} locally finite. If there are no such representations, then the functor $\Res_{K, K^\et}: \Rep(K) \to \Rep(K^\et)$ is an equivalence. 
\end{rmk}

\subsection{Non-locally-finite representations}
\label{subsec: non-locally finite representations}
In this section, we work in a similar set-up to show that it is possible to have a group-valued prestack which is dense in a pro-algebraic group and which still has representations which are not locally finite. This means that if the functor $\Res_{K, K^\et}$ is to be an equivalence, and all representations of $K^\et$ are locally finite, it is a property very particular to these group-valued prestacks. 

Consider $\A{\infty} = \displaystyle\colim_{i \in \BN} {\A{i}} \emb \overline{\BA^{\infty}} = \displaystyle\lim_{i \in \BN} {\A{i}}$, viewed as additive groups in the obvious way: for $S = \Spec{R}$, $\A{\infty}(S)$ is the set of finite sequences in $R$, while $\overline{\A{\infty}}(S)$ is the set of infinite sequences, both equipped with term-wise addition. 

Just as in the situation of $K^\et \emb K$ the maps from the subgroup to the finite-dimensional quotients of the pro-group are all surjective: we have
\begin{align*}
\A{\infty} \surj \A{c},
\end{align*}
corresponding to truncating the sequences at the $c$th term. Restriction of functions from $\overline{\A{\infty}}$ to $\A{\infty}$ is still injective, by the same argument as in the proof of Lemma \ref{restriction is injective}.

Now consider the regular representation of $\A{\infty}$: this is the infinite-di\-men\-sion\-al vector space 
\begin{align*}
V&= \Hom(\A{\infty}, \A{1})\\
 &= \lim_{n \in \BN} (k[t_1, \ldots, t_n]),
\end{align*} 
with the action of $\A{\infty}$ given by precomposition with the addition map. Consider the element $v \in V$ given by the infinite family of polynomials
\begin{align*}
f_n = \sum_{j=1}^n \prod_{i=1}^j t_i \in k[t_1, \ldots, t_n].
\end{align*}

Then the subspace generated by $v \in V$ under the action of $\A{\infty}(k)$ is infinite-dimensional. For example, if ${\bf e}_{i_0} \in \A{\infty}(k)$ is the sequence with $1$ in the $i_0$th position and 0 everywhere else, then ${\bf e}_{i_0}.v$ is given by the infinite sequence of polynomials
\begin{align*}
\sum_{j=1}^n \prod_{\substack{i=1\\i \ne i_0}}^j t_i,
\end{align*}
and so the ${\bf e}_{i_0}.v$ are linearly independent as $i_0$ varies. Hence the representation $V$ is not locally finite. 

We conclude that
\begin{align*}
\Rep^\lf(\A{\infty}) \subsetneq \Rep(\A{\infty}).
\end{align*}

On the other hand, there even exist finite-dimensional representations of $\A{\infty}$ which do not extend to $\overline{\A{\infty}}$, which cannot happen for the case of $K^\et \emb K$. Indeed, consider the two-dimensional unipotent representation corresponding to the assignment
\begin{align*}
(x_i)_i \mapsto \left( \begin{array}{cc}
1 & \displaystyle{\sum_i x_i} \\
0 & 1 \\
\end{array}
\right).
\end{align*}
It cannot be extended to $\overline{\A{\infty}}$.

\subsection{Representations of \texorpdfstring{$G^\et$}{G-etale}}
\label{subsec: representations of G-etale}
In this subsection, we compare the representations of $G$ to those of $G^\et$. We distinguish the full subcategory $\Rep^{K^\et\text{-}\lf}(G^\et)\subset \Rep(G^\et)$ of representations of $G^\et$ which are locally finite when viewed as representations of $K^\et$ via the inclusion of $K^\et$ inside $G^\et$. 

\begin{defn}
We shall refer to such representations as \emph{$K^\et$-locally-finite} representations of $G^\et$. 
\end{defn}

\begin{prop}\label{prop: G extensions}
Let $(V, \bR) \in \Rep^{K^\et\text{-}\lf}(G^\et)$ where $V$ is a vector space and $\bR: G^\et \to GL_V$. Then $\bR$ extends uniquely to a morphism $\overline{\bR}: G \to GL_V$ of group-valued prestacks. 
\end{prop}

\begin{proof}
Let $\bS: K^\et \to V$ be the composition of $\bR$ with the inclusion $K^\et \emb G^\et$. Then $(V, \bS)$ is a locally finite representation of $K^\et$, and hence extends uniquely to a representation $\overline{\bS}: K \to GL_V$ by Corollary \ref{cor: restriction is an embedding for K}.

Consider the following diagram of group-valued prestacks:
\begin{center}
\begin{tikzpicture}[>=angle 90]
\matrix(b)[matrix of math nodes, row sep=1em, column sep=1em, text height=1.5ex, text depth=0.25ex]
{ K^\et & & K & \\
\\
  G^\et & & G & \\
        & &   & GL_V .\\};
\path[right hook->, font=\scriptsize]
 (b-1-1) edge (b-1-3)
         edge (b-3-1)
 (b-3-1) edge (b-3-3)
 (b-1-3) edge (b-3-3);
\path[->, font=\scriptsize]
 (b-3-1) edge[bend right=15] node[below]{$\bR$} (b-4-4)
 (b-1-3) edge[bend left=20] node[right]{$\overline{\bS}$} (b-4-4);
\path[->,dashed, font=\scriptsize]
 (b-3-3) edge node[above right]{$\overline{\bR}$} (b-4-4);     
\end{tikzpicture}
\end{center}
We wish to define the group homomorphism $\overline{\bR}$ completing this diagram, and to show that it is unique. 

If there is to be a function $\overline{\bR}$ which respects the group structures as in the above diagram, it must satisfy
\begin{align} \label{eq: uniqueness of R}
\overline{\bR}(xk) = \bR(x)\overline{\bS}(k)
\end{align}
for every $k$-algebra $R$ and for all $R$-points $x \in G^\et(\Spec{R}), k \in K(\Spec{R})$. Notice that the map
\begin{align*}
G^\et \times K \to G
\end{align*}
given by composition of automorphisms (i.e. multiplication inside $G$) is surjective on $(\Spec{R})$-points. It follows that equation (\ref{eq: uniqueness of R}) determines $\overline{\bR}$ uniquely, if it exists. We need to prove that the assignment
\begin{align*}
xk \mapsto \bR(x)\overline{\bS}(k)
\end{align*}
gives a well-defined function $G(\Spec{R}) \to GL(V \otimes_k R)$, and moreover that this is functorial in the $k$-algebra $R$. 

For the first point, suppose that $xk = x^\prime k^\prime$ for some $x, x^\prime \in G^\et(\Spec{R})$, $k, k^\prime \in K(\Spec{R})$. Then $(x^\prime)^{-1} x = k^\prime k^{-1}$, and this is an element of $K^\et(\Spec{R})$, so that
\begin{align*}
\bR(x^\prime)^{-1}\bR(x) = \bS(x^\prime)^{-1} \bS(x) = \bS(k^\prime) \bS(k)^{-1} = \overline{\bS}(k^\prime) \overline{\bS}(k)^{-1},
\end{align*}
and hence
\begin{align*}
\bR(x) \overline{\bS}(k) = \bR(x^\prime) \overline{\bS}(k^\prime),
\end{align*}
as required. Functoriality is also straightforward to check. 

We conclude that there is a unique morphism of prestacks $\overline{\bR}: G \to GL_V$ making the above diagram commute. It remains to show that this morphism respects the group structures. Fix a (possibly infinite) basis $\{v_i\}_{i \in I}$ for $V$, and for any pair $(i,j) \in I \times I$ define the projection
\begin{align*}
\pi_{ij}: GL_V \to \A{1}
\end{align*} 
in the obvious way. To show that $m_{GL_V} \circ \overline{\bR} = (\overline{\bR} \times \overline{\bR}) \circ m_G$, it suffices to show that
\begin{align*}
\pi_{ij} \circ m_{GL_V} \circ \overline{\bR} = \pi_{ij} \circ (\overline{\bR} \times \overline{\bR}) \circ m_G  : G \times G \to \A{1}
\end{align*}
for every pair $(i,j)$. By Lemma \ref{restriction is injective}, it suffices to show that these functions agree when restricted to $G^\et \times G^\et$; but this is clear because the restriction of $\overline{\bR}$ to $G^\et$ is the homomorphism $\bR$. \qed
\end{proof}

We have the following analogue to Corollary \ref{cor: restriction is an embedding for K}:
\begin{corollary} \label{cor: restriction is an embedding for G}
The restriction functor
\begin{align*}
\Res_{G,G^\et}: \Rep(G) \to \Rep(G^\et)
\end{align*}
is fully faithful, with essential image $\Rep^{K^\et\text{-}\lf}(G^\et)$.
\end{corollary}
\begin{proof}
The only part that is not immediate is the surjectivity of the maps of hom-spaces. For $t=1,2$, let $(V_t, \overline{\bR}_t)$ be representations of $G$, with $(V_t, \bR_t)$ the restrictions to $G^\et$. Suppose that we have a linear map $f: V_1 \to V_2$ compatible with the maps $\bR_t$; then we need to show that it is also compatible with $\overline{\bR}_t$. That is, we need to show that the following diagram commutes:
\begin{center}
\begin{tikzpicture}[>=angle 90]
\matrix(b)[matrix of math nodes, row sep=3em, column sep=3em, text height=1.5ex, text depth=0.25ex] 
{ G         & GL_{V_1} \\
  GL_{V_2}  & \End(V_1, V_2).\\};
\path[->, font=\scriptsize]
(b-1-1) edge node[above]{$\overline{\bR}_1$} (b-1-2)
        edge node[left]{$\overline{\bR}_2$} (b-2-1)
(b-2-1) edge (b-2-2)
(b-1-2) edge (b-2-2);
\end{tikzpicture}
\end{center}
This is similar to the argument in the last part of the proof of  Proposition \ref{prop: G extensions}. Fix bases $\{v_i\}_{i \in I}$ and $\{w_j\}_{j \in J}$ for $V_1$ and $V_2$; then it suffices to show that the diagram commutes after composing with the $(i,j)$th projection $\End(V_1, V_2) \to \A{1}$ for 
every $(i, j) \in I \times J$. But then by Lemma \ref{restriction is injective} once more, it is enough to show that each of the resulting diagrams commutes after pre-composing with the inclusion $G^\et \emb G$, and this is true by our assumption on $f$. \qed
\end{proof}

\section{Universal modules}
\label{sec: universal modules}

In this section, we finally begin our analysis of universal $\CD$-modules. We will see that they are equivalent to quasi-coherent sheaves on the stack $\varietiesc{\infty}$ of \'etale germs. Similarly, we will see that universal $\CO$-modules are just quasi-coherent sheaves on the stack $\pointedvarietiesc{\infty}$. We begin by recalling the definitions of universal modules, given by Beilinson and Drinfeld (see \cite{BD1} 2.9.9); in \ref{subsec: from universal D modules to quasi-coherent sheaves}, \ref{subsec: from quasi-coherent sheaves to universal D-modules}, and \ref{subsec: compatibility of the functors} we prove that the category $\univcatD$ of universal $\CD$-modules of dimension $n$ is equivalent to $\QCoh{\varietiesc{\infty}}$. In \ref{subsec: the O-module setting} we discuss the corresponding equivalence in the case of universal $\CO$-modules. 

\begin{defn}\label{defn: universal O module}
A \emph{universal $\CO$-module} $\SF$ of dimension $n$ consists of the following data:
\begin{enumerate}
\item For each smooth family $X \xrightarrow{\pi} S$ of relative dimension $n$, we have an $\CO_X$-module $\SF_{X/S} \in \QCoh{X}$.
\item For each fibrewise \'etale morphism $f=(f_X,f_S): (X/S) \to (X^\prime/S^\prime)$ of smooth $n$-dimensional families, we have an isomorphism 
\begin{align*}
\SF(f): \univ \xrightarrow{\sim} (f_X)^*\univv 
\end{align*}
of $\CO_{X}$-modules. 
\end{enumerate}
These isomorphisms are required to be compatible with composition in the following sense. Suppose we are given three smooth $n$-dimensional families with fibrewise \'etale morphisms between them:
\begin{align*}
X/S \xrightarrow{f} X^\prime/S^\prime \xrightarrow{g} X^{\prime\prime}/S^{\prime\prime}.
\end{align*}
Then we require that $\SF(g \circ f)$ agrees with the composition $f_X^*\SF(g) \circ \SF(f)$.
\end{defn}

A morphism $\phi: \SF \to \SG$ of universal $\CO$-modules is a collection of morphisms 
\begin{align*}
\phi_{X/S}: \SF_{X/S} \to \SG_{X/S},
\end{align*}
indexed by $n$-dimensional families, compatible with the structure isomorphisms. 

In this way we obtain a category $\univcat$ of universal $\CO$-modules of dimension $n$. Similarly, we can define the category $\univcatD$ of universal $\CD$-modules of dimension $n$:

\begin{defn}\label{defn: universal D module}
A \textit{universal $\CD$-module} of dimension $n$ is a rule $\SF$ which assigns:
\begin{enumerate}
\item to each smooth $X \to S$ of relative dimension $n$ a left  $\CD_{X/S}$-module $\SF(X/S)$;
\item to each fibrewise \'etale morphism $f=(f_X, f_S):  (X/S) \to (X^\prime/S^\prime)$ of smooth $n$-dimensional families, an isomorphism
\begin{align*}
\SF(f): \SF(X/S) \EquivTo f^*\SF(X^\prime/S^\prime),
\end{align*}
in a way compatible with composition.
\end{enumerate}
\end{defn}

Let us be a little more precise. The category $\CD(X/S)$ is (by definition) the category of quasi-coherent sheaves on $(X/S)_\dR$, where $(X/S)_\dR$ is the following fibre product:\footnote{Recall that we associate to any prestack $\CY$ its \emph{de Rham prestack} $\CY_\dR$:
\begin{align*}
\CY_\dR: T \mapsto \CY_\dR(T) \defeq \CY(T_\red).
\end{align*}
Then we define the category of \emph{left $\CD$-modules} on $\CY$ to be $\QCoh{\CY_\dR}$. The forgetful functor from $\CD$-modules to $\CO$-modules is given by pullback along the natural map $\fdR{\CY}: \CY \to \CY_\dR$.} 
\begin{center}
\begin{tikzpicture}[>=angle 90]
\matrix(i)[matrix of math nodes, row sep=3em, column sep=3em, text height=1.5ex, text depth=0.25ex]
{(X/S)_\dR & X_\dR \\
 S         & S_\dR \\};
\path[->, font=\scriptsize]
 (i-1-1) edge (i-1-2)
         edge (i-2-1)
 (i-1-2) edge node[right]{$\pi_\dR$} (i-2-2)
 (i-2-1) edge node[below]{$\fdR{S}$} (i-2-2);
\end{tikzpicture}
\end{center}

This means that the object $\SF(X/S)$ is given by a collection of objects 
\begin{align*}
\SF(X/S)_{T \to (X/S)_\dR} \in \QCoh{T}
\end{align*}
indexed by affine schemes $T$ and morphisms $T \to (X/S)_\dR$, along with isomorphisms describing compatibility with pullbacks. (See \cite{GR} and also III.4 of \cite{GR-book}, for more details on categories of $\CD$-modules over prestacks as well as the full definition in the $(\infty,1)$-categorical and derived setting.)

If we have a fibrewise \'etale morphism of $n$-dimensional families $f=(f_X, f_S): (X/S) \to (X^\prime/S^\prime)$, we obtain a morphism $f_{X/S}:(X/S)_\dR \to (X^\prime/S^\prime)_\dR$ as follows:

\begin{center}
\begin{tikzpicture}[>=angle 90]
\matrix(j)[matrix of math nodes, row sep=1.5em, column sep=1.5em, text height=1.5ex, text depth=0.25ex]
{ (X/S)_\dR &                         & X_\dR   &               \\
            & (X^\prime/S^\prime)_\dR &          & X^\prime_\dR  \\
   S        &                         &          &               \\
            & S^\prime                &          & S^\prime_\dR . \\};
\path[->, font=\scriptsize]
 (j-1-1) edge (j-1-3)
 (j-1-3) edge node[above right]{$f_{X, \dR}$} (j-2-4)
 (j-1-1) edge (j-3-1)
 (j-3-1) edge node[below left]{$f_S$} (j-4-2)
 
 (j-2-2) edge (j-2-4)
         edge (j-4-2)
 (j-2-4) edge node[right]{$\pi^\prime_\dR$} (j-4-4)
 (j-4-2) edge node[below]{$\fdR{S}$} (j-4-4);
\path[dashed, ->, font=\scriptsize]
 (j-1-1) edge node[below left]{$f_{X/S}$} (j-2-2);
\end{tikzpicture}
\end{center}
Then the compatibility isomorphism $\SF(f)$ associated to $\SF$ is an isomorphism between $\SF(X/S)$ and $f_{X/S}^*\SF(X^\prime/S^\prime)$ in $\QCoh{(X/S)_\dR}$.  

\begin{notation}
We will always use the subscript $\bullet_{X/S}$ to denote the morphism of relative de Rham prestacks induced by a fibrewise \'etale morphism between two smooth families, even when neither of the smooth families involved is actually denoted by $X/S$.
\end{notation}

\begin{rmk}\label{rmk: left vs right relative D-modules}
At this stage, the reader may wonder why we have chosen to use \emph{left} relative $\CD$-modules, rather than \emph{right}, which is the more usual category in which to work. (See for example the discussion of the category $\IndCoh{(X/S)_\dR}$ of \emph{relative crystals} in III.4, Section 3.3 of \cite{GR-book}.) We have several reasons for this choice.
\begin{enumerate}
\item We wish, at least for the moment, to remain consistent with the definition of universal $\CD$-module given by Beilinson and Drinfeld.
\item Also for the moment, we wish to work with abelian categories rather than DG-categories; the $!$-pullback functors needed in the definition of the category of ind-coherent sheaves on a prestack are inherently derived. 
\item Suppose for the sake of argument that we are working with DG-categories rather than abelian categories. As will be seen in the following sections, we are going to compare universal $\CD$-modules to sheaves on the stack $\varietiesc{\infty}$. %Here is a bit I've added instead of all the nice paragraphs below that I took out.%
If we work with quasi-coherent sheaves, we have no problems, because $\QCoh{\CY}$ is defined for an arbitrary prestack. On the other hand, $\IndCoh{\CY}$ is defined only for prestacks which are \emph{locally of finite type}; in particular, the definition does not make sense for the stack $\varietiesc{\infty}$.  %We will see that, given a universal $\CD$-module $\SF$ and a map $S \to \varietiesc{\infty}$ corresponding to a pointed family $\pi: X \rightleftarrows S: \sigma$, we pull back the $\CD$-module $\SF(X/S)$ by the section $\sigma$, thus obtaining a $\CD$-module on $S$. These should be compatible with pullback along maps $S \to S^\prime$. 

\end{enumerate}

\end{rmk}

\begin{thm}\label{thm: universal D-modules are quasi-coherent sheaves}
The category $\univcatD$ of universal $\CD$-modules of dimension $n$ is equivalent to the category of quasi-coherent sheaves on $\varietiesc{\infty}$. 
\end{thm}

\begin{proof}
Recall that it suffices to show that $\QCoh{\prestackVbc{\infty}} \simeq \univcatD$. The idea behind the proof is quite simple, but there are many technical details to be checked. We proceed by defining functors in both directions and checking that they are quasi-inverse to each other. 
\end{proof}

\subsection{From universal \texorpdfstring{$\CD$}{D}-modules to quasi-coherent sheaves}
\label{subsec: from universal D modules to quasi-coherent sheaves}
First we define the functor $\theta: \univcatD \to \QCoh{\prestackVbc{\infty}}$. Given a universal $\CD$-module $\SF$, we wish to define a quasi-coherent sheaf $\theta(\SF)$ on $\prestackVbc{\infty}$. That is, given any $S$ and any morphism $(\pi,\sigma) : S \to \prestackVbc{\infty}$ representing a pointed $n$-dimensional family $(\pi: X \rightleftarrows S: \sigma)$, we need to define a quasi-coherent sheaf $\theta(\SF)_{X \rightleftarrows S}$ on $S$. We have $\SF(X/S) \in \QCoh{ (X/S)_\dR }$, and $\sigma: S \to X$ induces a section $\overline{\sigma}: S \to (X/S)_\dR$, so we simply set 
\begin{align*}
\theta(\SF)_{X \rightleftarrows S} \defeq \overline{\sigma}^* \SF(X/S).
\end{align*}

(We adopt the convention of denoting the map into the relative de Rham stack induced by a section by $\overline{\bullet}$.) Next we need to define the compatibility isomorphisms. Suppose we have a commutative diagram in $\PreStk$ of schemes mapping to $\prestackVbc{\infty}$:

\begin{center}
\begin{tikzpicture}[>=angle 90]
\matrix(e)[matrix of math nodes, row sep=1.5em, column sep=2em, text height=1.5ex, text depth=0.25ex]
{     &    &                     & & &       &  V_\alpha  &                      &         \\
  S_2 &    &\prestackVbc{\infty} & & &   X_1 &            & S_1 \times_{S_2} X_2 & X_2     \\
      & {} &                     & & &       &  S_1       &                      &         \\
  S_1 &    &   {}                & & &   S_1 &            & S_1                  & S_2.     \\};
\path[->, font=\scriptsize]
 (e-2-1) edge node[above]{$(\pi_2,\sigma_2)$} (e-2-3)
 (e-4-1) edge node[left]{$f$} (e-2-1) 
         edge node[below right]{$(\pi_1, \sigma_1)$} (e-2-3)
         
 (e-3-2) edge[-implies, double, double distance=2pt] node[right=3pt]{$\alpha$} (e-2-1)        

 (e-2-6) edge[bend left=10] node[right]{$\pi_1$} (e-4-6)
 (e-4-6) edge[bend left=10] node[left]{$\sigma_1$} (e-2-6)
 (e-1-7) edge[bend left=10] node[right]{$\rho_\alpha$} (e-3-7)
 (e-3-7) edge[densely dotted, bend left=10] node[left]{$\tau_\alpha$} (e-1-7)
 (e-2-8) edge[bend left=10] node[right]{$f^*\pi_2$}(e-4-8)
 (e-4-8) edge[bend left=10] node[left]{$f^*\sigma_2$}(e-2-8)
 (e-1-7) edge node[above left]{$\phi_{\alpha}$} (e-2-6)
         edge node[above right]{$\psi_{\alpha}$} (e-2-8)
 (e-2-9) edge[bend left=10] node[right]{$\pi_2$} (e-4-9)
 (e-4-9) edge[bend left=10] node[left]{$\sigma_2$} (e-2-9)
 
 (e-2-8) edge node[above]{$\pr_{X_2}$} (e-2-9)
 (e-4-8) edge node[below]{$f$} (e-4-9);
\path[-]
 (e-3-7) edge[double, double distance=2pt] (e-4-6)
         edge[double, double distance=2pt] (e-4-8);
\end{tikzpicture}
\end{center}

We need to specify an isomorphism 
\begin{align*}
\theta(\SF)(f, \alpha): f^*\left(\theta(\SF)_{X_2 \rightleftarrows S_2} \right) \EquivTo \theta(\SF)_{X_1 \rightleftarrows S_1};
\end{align*}
it arises naturally from the universality of $\SF$. Indeed, from the definition of $\SF$ we have isomorphisms
\begin{align*}
\SF(\phi_\alpha, \id_{S_1}) :& \SF\left(V_\alpha / S_1\right) \EquivTo \left(\phi_\alpha, \id_{S_1} \right)_{X/S}^* \SF \left(X_1 / S_1\right);\\
\SF(\pr_{X_2} \circ \psi_\alpha, f) :& \SF\left(V_\alpha / S_1\right) \EquivTo \left(\pr_{X_2} \circ \psi_\alpha, f\right)_{X/S}^* \SF \left(X_2/S_2\right).
\end{align*}
Note that $(\phi_{\alpha}, \id_{S_1})_{X/S} \circ \overline{\tau_\alpha} = \overline{\sigma_1}$ and $(\pr_{x_2} \circ \psi_{\alpha}, f)_{X/S} \circ \overline{\tau_\alpha} = \overline{\sigma_2} \circ f$, so that setting
\begin{align*}
\theta(\SF)(f, \alpha) \defeq \overline{\tau_\alpha}^*\left( \SF(\phi_\alpha, \id_{S_1}) \circ \SF(\pr_{X_2} \circ \psi_\alpha, f)^{-1} \right),
\end{align*}
we obtain an isomorphism
\begin{align*}
f^* \overline{\sigma_2}^* \SF(X_2/S_2) \EquivTo \overline{\sigma_1}^* \SF(X_1/S_1),
\end{align*}
as required.

Let us now check that $\theta(\SF)(f, \alpha)$ is independent of the choice of common \'etale neighbourhood $(V_\alpha, \phi_\alpha, \psi_\alpha)$ taken to represent the isomorphism $\alpha$ between $(X_1 \rightleftarrows S_1)$ and $(S_1 \times_{S_2} X_2 \rightleftarrows S_1)$ in $\prestackVbc{\infty}(S_1)$. It suffices to consider the case that $(V_\alpha, \phi_\alpha, \psi_\alpha)$ and $(V_\alpha^\prime, \phi_\alpha^\prime, \psi_\alpha^\prime)$ are common \'etale neighbourhoods with $\phi_\alpha^\prime = \phi_\alpha \circ g$ and $\psi_\alpha^\prime = \psi_\alpha \circ g$ for some $g: V^\prime/S_1 \to V/S_1$ \'etale and compatible with the sections on $S_{1,\red}$. Then we need to show that 
\begin{align*}
\overline{\tau}^*\left( \SF(\phi_\alpha, \id_{S_1}) \circ \SF(\psi_\alpha, \id_{S_1})^{-1} \right)=\overline{\tau}^*\left( \SF(\phi^\prime_\alpha, \id_{S_1}) \circ \SF(\psi^\prime_\alpha, \id_{S_1})^{-1} \right)
\end{align*}
as morphisms of quasi-coherent sheaves on $S_1$, and this follows immediately from the compatibility of $\SF{(\bullet)}$ with respect to composition. 

\begin{rmk}\label{rmk: not well-defined w.r.t. c-equivalence}
Note that in general this assignment is \emph{not} well-defined with respect to $(c)$-equivalence for any finite $c$. We will return to this point in Section \ref{sec: convergent and ind-finite universal modules}.
\end{rmk}

One can also check that the $\theta(\SF)(f, \alpha)$ are compatible under composition, and hence $\theta(\SF)$ is indeed an object of $\QCoh{\prestackVbc{\infty}}$. See Appendix \ref{appendix} for the details of the proof. 

The definition of $\theta$ on morphisms is straightforward: given a morphism $F: \SF \to \SG$ of universal $\CD$-modules, the morphism $\theta(F): \theta(\SF) \to \theta(\SG)$ of quasi-coherent sheaves is given by $\theta(F)_{X \rightleftarrows S}\defeq \overline{\sigma}^* (F(X/S)):$
\begin{align*}
  \theta(\SF)_{X \rightleftarrows S} = \overline{\sigma}^*(\SF(X/S)) \longrightarrow \theta(\SG)_{X \rightleftarrows S} = \overline{\sigma}^*(\SG(X/S)).
\end{align*}
This completes the construction of the functor $\theta$.

\subsection{From quasi-coherent sheaves to universal \texorpdfstring{$\CD$}{D}-modules}
\label{subsec: from quasi-coherent sheaves to universal D-modules}
Now we will construct the quasi-inverse functor 
\begin{align*}
\Psi: \QCoh{\prestackVbc{\infty}} \to \univcatD.
\end{align*}

(Some details have been omitted from the following for legibility of exposition; an extended version of the paper can be found on the arXiv, in which all details are included.)

Let $M \in \QCoh{\prestackVbc{\infty}}$, and let $X \to S$ be smooth of relative dimension $n$. We need to define an object $\Psi(M)(X/S) \in \QCoh{(X/S)_\dR}$. More precisely, for any $T \to (X/S)_\dR$, we need to define $\Psi(M)(X/S)_{T \to (X/S)_\dR}$, together with isomorphisms describing the compatibility with pullbacks. 

By definition of $(X/S)_\dR$, a morphism $T \to (X/S)_\dR$ is given by a pair of morphisms $(g,h)$ as in the following commutative diagram:
\begin{center}
\begin{tikzpicture}[>=angle 90]
\matrix(i)[matrix of math nodes, row sep=3em, column sep=3em, text height=1.5ex, text depth=0.25ex]
{T_\red & X \\
 T      & S. \\};
\path[->, font=\scriptsize]
 (i-1-1) edge node[above]{$g$}(i-1-2)
 (i-1-2) edge node[right]{$\pi$} (i-2-2)
 (i-2-1) edge node[below]{$h$} (i-2-2);
\path[right hook->, font=\scriptsize]
 (i-1-1) edge node[left]{$\redEmb{T}$}(i-2-1);
\end{tikzpicture}
\end{center}
To define an object of $\QCoh{T}$ using $M$, we need an object of $\prestackVbc{\infty}(T)$, i.e. a pointed $n$-dimensional family over $T$. An obvious candidate is $T \times_S X$, which is smooth of dimension $n$ over $T$. To define a section, note that we can define $\sigma^\circ \defeq(\redEmb{T}, g): T_\red \to T \times_S X$. Then we have the following commutative diagram:
\begin{center}
\begin{tikzpicture}[>=angle 90]
\matrix(i)[matrix of math nodes, row sep=3em, column sep=3em, text height=1.5ex, text depth=0.25ex]
{T_\red & T \times_S X \\
 T      & T \\};
\path[->, font=\scriptsize]
 (i-1-1) edge node[above]{$\sigma^\circ$}(i-1-2)
         edge[right hook->] node[left]{$\redEmb{T}$}(i-2-1)
 (i-1-2) edge (i-2-2)
 (i-2-1) edge[dashed] node[above left]{$\sigma$}(i-1-2);
\path[-]
 (i-2-1) edge[double, double distance=2pt] (i-2-2);
\end{tikzpicture}
\end{center}
where formal smoothness of $T\times_S X \to T$ allows us to lift $\sigma^\circ$ to a section $\sigma$. This gives us an object of $\prestackVbc{\infty}(T)$.

Of course, $\sigma$ is not unique, but any other choice $\sigma^\prime$ of lifting will yield an object of $\prestackVbc{\infty}(T)$ canonically isomorphic to the original one. Hence up to a canonical isomorphism, we obtain $M_{T\times_S X \rightleftarrows T} \in \QCoh{T}$, and we define 
\begin{align*}
\Psi(M)(X/S)_{T \to (X/S)_\dR} \defeq M_{T \times_S X \rightleftarrows T}.
\end{align*}

To complete the construction of $\Psi(M)(X/S) \in \QCoh{(X/S)_\dR}$, we need to specify the compatibilities under pullback. Assume we have a commutative diagram in $\PreStk$:
\begin{center}
\begin{tikzpicture}[>=angle 90]
\matrix(i)[matrix of math nodes, row sep=3em, column sep=3em, text height=1.5ex, text depth=0.25ex]
{ T_2 & (X/S)_\dR \\
 T_1 & \\};
\path[->, font=\scriptsize]
 (i-1-1) edge node[above]{$(g_2,h_2)$} (i-1-2)
 (i-2-1) edge node[left]{$f$} (i-1-1)
         edge node[below, sloped]{$(g_1,h_1)$} (i-1-2);
\end{tikzpicture}
\end{center} 
i.e. $(g_1,h_1) = f^*(g_2,h_2) = (g_2 \circ f_\red, h_2 \circ f)$. We need to exhibit an isomorphism 
\begin{align*}
f^*\left(\Psi(M)(X/S)_{T_2 \to (X/S)_\dR}\right) \EquivTo \Psi(M)(X/S)_{T_1 \to (X/S)_\dR}
\end{align*}
or equivalently $f^*M_{T_2 \times_S X \rightleftarrows T_2} \EquivTo M_{T_1 \times_S X \rightleftarrows T_1}$. To do this, it suffices to show that the following diagram commutes canonically in $\PreStk$:
\begin{center}
\begin{tikzpicture}[>=angle 90]
\matrix(i)[matrix of math nodes, row sep=3em, column sep=3em, text height=1.5ex, text depth=0.25ex]
{ T_2 & \prestackVbc{\infty} \\
 T_1 & \\};
\path[->, font=\scriptsize]
 (i-1-1) edge node[above]{$(\pr_{T_2},\sigma_2)$} (i-1-2)
 (i-2-1) edge node[left]{$f$} (i-1-1)
         edge node[below, sloped]{$(\pr_{T_1}, \sigma_1)$} (i-1-2);
\end{tikzpicture}
\end{center} 
i.e. to exhibit a canonical (up to $(\infty)$-equivalence) common \'etale neighbourhood between $T_1 \times_S X \rightleftarrows T_1$ and $T_1 \times_{T_2} (T_2 \times_S X) \rightleftarrows T_1$. The obvious candidate is
\begin{equation}\label{diagram: candidate}
\begin{tikzpicture}[>=angle 90, baseline=(current bounding box.center)]
\matrix(e)[matrix of math nodes, row sep=1.5em, column sep=3em, text height=1.5ex, text depth=0.25ex]
{              &  T_1 \times_S X  &  \\
 T_1\times_S X &                  & T_1 \times_{T_2}(T_2 \times_S X) \\
               &  T_1             &  \\
 T_1           &                  & T_1.   \\};
\path[->, font=\scriptsize]
 (e-2-1) edge[bend left=10]  (e-4-1)
 (e-4-1) edge[bend left=10] node[left]{$\sigma_1$} (e-2-1)
 (e-1-2) edge[bend left=10] (e-3-2)
 (e-2-3) edge[bend left=10]  (e-4-3)
 (e-4-3) edge[bend left=10] node[left]{$f^*\sigma_2$} (e-2-3)
 (e-1-2) edge node[above, sloped]{$\sim$} (e-2-3);
\path[-]
 (e-1-2) edge[double, double distance=2pt] (e-2-1)
 (e-3-2) edge[double, double distance=2pt] (e-4-1)
         edge[double, double distance=2pt] (e-4-3);
\path[densely dotted, ->, font=\scriptsize]
 (e-3-2) edge[bend left=10] node[left]{$\sigma_1$} (e-1-2);
\end{tikzpicture}
\end{equation}
The (reduced) commmutativity of this diagram follows from noting that $\sigma_1 \circ \redEmb{T_1} = (\redEmb{T_1},g_1)$ and $f^*\sigma_2 \circ \redEmb{T_1} = (\redEmb{T_1}, g_2 \circ f_\red)$ as maps $(T_1)_\red \to T_1 \times_S X$.

This yields the desired isomorphism\footnote{Here and in the following we suppress the \'etale morphisms and write simply $(T_1 \times_S X)$ for the common \'etale neighbourhood (\ref{diagram: candidate}).} 
\begin{align*}
M(f, T_1 \times_S X):f^*M_{T_2 \times_S X \rightleftarrows T_2} \EquivTo M_{T_1 \times_S X \rightleftarrows T_1},
\end{align*}
and the compatibility of these isomorphisms comes from the structure of $M$.  %Indeed, suppose we have the following commutative diagram
%\begin{center}
%\begin{tikzpicture}[>=angle 90]
%\matrix(e)[matrix of math nodes, row sep=2em, column sep=5em, text height=1.5ex, text depth=0.25ex]
%{ T_3 &   \\
%  T_2 & (X/S)_{\dR}. \\
%  T_1 &   \\};
%\path[->, font=\scriptsize]
% (e-3-1) edge node[left]{$f_1$} (e-2-1)
% (e-2-1) edge node[left]{$f_2$} (e-1-1)
% (e-1-1) edge node[above, sloped]{$(g_3,h_3)$} (e-2-2)
% (e-2-1) edge node[above=-3pt]{$(g_2, h_2)$} (e-2-2)
% (e-3-1) edge node[below, sloped]{$(g_1, h_1)$} (e-2-2);  
%\end{tikzpicture}
%\end{center}
%Then we need to show that the isomorphism $f_1^* f_2^* M_{T_3 \times_S X \rightleftarrows T_3} \EquivTo M_{T_1 \times_S X \rightleftarrows T_1}$ is equal to the composition 
%\begin{align*}
%f_1^* f_2^* M_{T_3 \times_S X \rightleftarrows T_3} \EquivTo f_1^* M_{T_2 \times_S X \rightleftarrows T_2} \EquivTo M_{T_1 \times_S X \rightleftarrows T_1}.
%\end{align*}
%That is, we need to prove that
%\begin{align*}
%M(f, T_1 \times_S X \rightleftarrows T_1) \circ f_1^* M(f_2, T_2 \times_S X \rightleftarrows T_2) = M(f_2 \circ f_1, T_1 \times_S X \rightleftarrows T_1).
%\end{align*}

%This follows from the compatibility of $M$ with respect to pullbacks, and the fact that the composition of the morphisms represented by the common \'etale neighbourhoods given by $(T_1 \times_S X)$ and $(T_2 \times_S X)$ is represented by the common \'etale neighbourhood $(T_1 \times_S X)$ between $(T_1 \times_S X \rightleftarrows T_1)$ and $(f_2 \circ f_1)^*\left(T_3 \times_S X \rightleftarrows T_3 \right)$. 

Therefore $\Psi(M)(X/S) \in \QCoh{(X/S)_\dR}$, as claimed.

Finally, to show that $\Psi(M) \in \univcatD$, we need to define the isomorphisms $\Psi(M)(f)$ associated to fibrewise \'etale morphisms $f=(f_X,f_S): (X/S) \to (X^\prime/S^\prime)$. We need to define an isomorphism 
\begin{align*}
\Psi(M)(X/S) \EquivTo f_{X/S}^*\Psi(M)(X^\prime/S^\prime)
\end{align*}
of sheaves on $(X/S)_\dR$, i.e. a compatible family of isomorphisms
\begin{align} \label{universal D-module compatibilities for Psi}
\Psi(M)(X/S)_{T\to (X/S)_\dR} \EquivTo \left(f_{X/S}^*\Psi(M)(X^\prime/S^\prime)\right)_{T \to (X/S)_\dR}
\end{align}
for each $T \to (X/S)_\dR$.

Let $T \to (X/S)_\dR$ be the morphism corresponding to the pair $(g:T_\red \to X, h: T \to S)$. Unwinding the definitions, we see that 
\begin{align*}
\left(f_{X/S}^*\Psi(M)(X^\prime/S^\prime)\right)_{T\to (X/S)_\dR} = \Psi(M)(X^\prime/S^\prime)_{T \to (X^\prime/S^\prime)_\dR}
\end{align*}
where the morphism $T \to (X^\prime/S^\prime)_\dR$ corresponds to the pair $(f_X \circ g, f_S \circ h)$. So we can rewrite (\ref{universal D-module compatibilities for Psi}) as
\begin{align*}
M_{T \times_S X \rightleftarrows T} \EquivTo M_{T \times_{S^\prime} X^\prime \rightleftarrows T}.
\end{align*}

To define such an isomorphism, it suffices to exhibit an isomorphism of the corresponding objects in $\prestackVbc{\infty}(T)$, i.e. a common \'etale neighbourhood between $(T \times_S X \rightleftarrows T)$ and $(T \times_{S^\prime} X^\prime \rightleftarrows T)$. We can take the following representative:
\begin{center}
\begin{tikzpicture}[>=angle 90]
\matrix(e)[matrix of math nodes, row sep=1.5em, column sep=3em, text height=1.5ex, text depth=0.25ex]
{                            &  T\times_{S} X  &  \\
 T \times_S X                &                               & T\times_{S^\prime} X^\prime \\
                             &  T                            &  \\
 T                           &                               & T.   \\};
\path[->, font=\scriptsize]
 (e-2-1) edge[bend left=10]  (e-4-1)
 (e-4-1) edge[bend left=10] node[left]{$\sigma$} (e-2-1)
 (e-1-2) edge[bend left=10] (e-3-2)
 (e-2-3) edge[bend left=10]  (e-4-3)
 (e-4-3) edge[bend left=10] node[left]{$\sigma^\prime$} (e-2-3)
 (e-1-2) edge node[above, sloped]{$(\id_T,f_X)$} (e-2-3);
\path[-]
 (e-1-2) edge[double, double distance=2pt] (e-2-1)
 (e-3-2) edge[double, double distance=2pt] (e-4-1)
         edge[double, double distance=2pt] (e-4-3);
\path[densely dotted, ->, font=\scriptsize]
 (e-3-2) edge[bend left=10] node[left]{$\sigma$} (e-1-2);
\end{tikzpicture}
\end{center}
The reduced commutativity of the right side of the diagram follows from noting that $\sigma \circ \redEmb{T} = (\redEmb{T}, g)$ while $\sigma^\prime \circ \redEmb{T} = (\redEmb{T}, f_X \circ g)$. 

Because of the structure of $M$, these isomorphisms are compatible with pullback along maps $T^\prime \to T$, and hence give the desired isomorphism 
\begin{align*}
\Psi(M)(f): \Psi(M)(X/S) \EquivTo f_{X/S}^*\Psi(M)(X^\prime/S^\prime).
\end{align*}

In turn, the maps $\Psi(M)(f)$ are themselves compatible with composition. We conclude that $\Psi(M)$ is indeed a universal $\CD$-module. 

The definition of $\Psi$ on morphisms of $\QCoh{\prestackVbc{\infty}}$ is clear: a morphism $F: M \to N$ of quasi-coherent sheaves on $\prestackVbc{\infty}$ amounts to a compatible family of morphisms $F_{X \rightleftarrows S}: M_{X \rightleftarrows S} \to N_{X \rightleftarrows S} \in \QCoh{S}$ indexed by morphisms $S \to \prestackVbc{\infty}$. Then we define $\Psi(F): \Psi(M) \to \Psi(N)$ by setting
\begin{align*}
\Psi(F)(X/S)_{T \to (X/S)_\dR} : M_{T\times_S X \rightleftarrows T} \to N_{T \times_S X \rightleftarrows T}
\end{align*}
to be equal to $F_{T \times_S X \rightleftarrows T}$. It is not hard to see that this definition is compatible with pullback by morphisms $T^\prime \to T$ as well as with the structure morphisms $\Psi(M)(f)$ and $\Psi(N)(f)$ corresponding to fibrewise \'etale morphisms $f=(f_X, f_S): X^\prime/S^\prime \to X/S$. It is also immediate that $\Psi(F \circ G) =\Psi(F) \circ \Psi(G)$, and so $\Psi$ gives a functor $\QCoh{\prestackVbc{\infty}} \to \univcatD$.

\subsection{Compatibility of \texorpdfstring{$\theta$ and $\Psi$}{the functors}}
\label{subsec: compatibility of the functors}
It remains to check that $\theta$ and $\Psi$ are indeed quasi-inverse. First suppose that we have $M \in \QCoh{\prestackVbc{\infty}}$ and consider $\theta \circ \Psi(M) \in \QCoh{\prestackVbc{\infty}}$. For $(\pi:X \rightleftarrows S: \sigma) \in \prestackVbc{\infty}$, we have
\begin{align*}
\left(\theta \circ \Psi(M)\right)_{X \rightleftarrows S} = \overline{\sigma}^*\Psi(M)(X/S).
\end{align*}
Here $\overline{\sigma}: S \to (X/S)_\dR$ corresponds by definition to the pair $(\sigma \circ \redEmb{S}, \id_S)$, so it follows that 
\begin{align*}
\overline{\sigma}^*\Psi(M)(X/S) = M_{S \times_S X \rightleftarrows S},
\end{align*}
and $(S \times_S X \rightleftarrows S) \simeq (X \rightleftarrows S)$. Therefore
\begin{align*}
\left(\theta \circ \Psi(M)\right)_{X \rightleftarrows S} \simeq M_{X \rightleftarrows S},
\end{align*}
which gives the natural isomorphism between $\theta \circ \Psi$ and $\text{Id}_{\QCoh{\prestackVbc{\infty}}}$. 

Conversely, let $\SF \in \univcatD$ and consider $\Psi \circ \theta (\SF)$. Take $\pi:X \to S$ smooth of dimension $n$ and $T \to (X/S)_\dR$ corresponding to a compatible pair of morphisms $(g: T_\red \to X, h: T \to S)$. Then
\begin{align*}
\left(\Psi \circ \theta (\SF)\right)(X/S)_{T \to (X/S)_\dR} &= \theta(\SF)_{\pr_T:T \times_S X \rightleftarrows T: \sigma}\\
 &= \overline{\sigma}^* \left(\SF (T \times_S X/T )\right),
\end{align*}
where $\sigma$ is a section $T \to T\times_S X$ such that $\sigma \circ \redEmb{T} = (\id_T, g)$. Notice that $f\defeq(\pr_X, h)$ gives a fibrewise \'etale map $(T \times_S X)/T \to (X/S)$, so that we have 
\begin{align*}
\SF(f):\SF(T \times_S X/T) \EquivTo f_{X/S}^*\SF(X/S).
\end{align*} 
Finally, unwinding the definitions of $f_{X/S}$ and $\overline{\sigma}$ shows that $f_{X/S} \circ \overline{\sigma}: T \to (X/S)_\dR$ agrees with $(g,h)$; hence we have
\begin{align*}
\overline{\sigma}^* \left(\SF (T \times_S X/T )\right) &\simeq \overline{\sigma}^*f_{X/S}^*\SF(X/S) \\
 & \simeq \SF(X/S)_{T \to (X/S)_\dR}
\end{align*}
as required. These isomorphisms gives the desired natural isomorphisms between $\Psi \circ \theta$ and $\text{Id}_{\univcatD}$. The proof is complete. \hfill $\square$

\subsection{The \texorpdfstring{$\CO$}{O}-module setting}
\label{subsec: the O-module setting}
We have an analogous result in the case of universal $\CO$-modules:
\begin{thm}
\label{thm: universal O modules are sheaves}
The category $\univcat$ of universal $\CO$-modules of dimension $n$ is equivalent to the category $\QCoh{\prestackPVbc{\infty}}$ and hence to $\QCoh{\pointedvarietiesc{\infty}}$.
\end{thm}

\begin{proof}
The idea behind the proof is similar to the case of universal $\CD$-modules: we proceed by defining functors in both directions and checking that they are quasi-inverse to each other. In brief, we have
\begin{align}\label{universal O module to sheaf}
\theta: \univcat &\to \QCoh{\prestackPVbc{\infty}}\\
\SF &\mapsto \left( (\pi:X \rightleftarrows S : \sigma) \mapsto \sigma^*(\SF_{X/S}) \right); \nonumber
\end{align}
\begin{align}\label{sheaf of universal O module}
\Psi: \QCoh{\prestackPVbc{\infty}} &\to \univcat\\
 M & \mapsto \left( (X \to S) \mapsto M_{\text{pr}_1: X \times_S X \rightleftarrows X: \Delta} \right). \nonumber
\end{align}
We refrain from spelling out the details---the definitions and arguments are along the same lines as those used in the proof of Theorem \ref{thm: universal D-modules are quasi-coherent sheaves}, but simpler. \qed
\end{proof}

\section{Convergent and ind-finite universal modules} 
\label{sec: convergent and ind-finite universal modules}
So far, we have identified the category of universal $\CD$-modules with the category of representations of the group-valued prestack $G^\et$. Furthermore, we have identified $\Rep(G)$ as a full subcategory of $\Rep(G^\et)$. In this section, we study the corresponding full subcategory of $\univcatD$. 

We have two approaches: in \ref{subsec: convergent universal modules} we take the first approach, via the description of $\Rep(G)$ as $\displaystyle\colim_{c \in \BN} \Rep(G^\cth)$, as in Proposition \ref{prop: representations of G are continuous}. The second method uses the characterisation of representations of $G$ as those representations of $G^\et$ which are locally finite when viewed as representations of $K^\et$, as in Corollary \ref{cor: restriction is an embedding for G}. We discuss this in \ref{subsec: ind-finite universal modules} and \ref{subsec: ind-finite universal D-modules}. Comparing the results obtained from each of these approaches allows us to provide two characterisations of those universal $\CD$-modules which lie in the essential image of $\Rep(G)$ under the equivalence $\Rep(G^\et) \EquivTo \univcatD$. We will call these the \emph{convergent} universal $\CD$-modules. 
\subsection{Convergent universal modules}
\label{subsec: convergent universal modules}
Recall the stack of formal germs introduced in Remark \ref{rmk: stack of analytic germs}:
\begin{align*}
\varieties \defeq \lim_{c \in \BN} \varietiesc{c}.
\end{align*}
Combining the results of Proposition \ref{prop: representations of G are continuous}, Corollary \ref{cor: restriction is an embedding for G}, and Theorem \ref{thm: universal D-modules are quasi-coherent sheaves}, we obtain the following diagram:
\begin{center}
\begin{small}
\begin{tikzpicture}[>=angle 90,bij/.style={above,sloped, inner sep=0.5pt},cross line/.style={preaction={draw=white, -, line width=6pt}}] 
\matrix(a)[matrix of math nodes, row sep=1em, column sep=.5em, text height=2.5ex, text depth=0.25ex]
{ & \Rep(G^\et) & & \QCoh{\varietiesc{\infty}} & \qquad & \univcatD \\
\Rep(G)                                       && \QCoh{\varieties}                                      & & & \\
                                              &&                                                        & & & \\
\displaystyle{\colim_{c \in \BN} \Rep(G^\cth)}&& \displaystyle{\colim_{c \in \BN} \QCoh{\varietiesc{c}}}& & & \\
                                              &&                                                        & & & \\
\Rep(G^\cth)                                  && \QCoh{\varietiesc{c}}                                  & & & \\};

\path[right hook-stealth]
(a-2-1) edge (a-1-2)
(a-2-3) edge (a-1-4)

(a-6-1) edge (a-4-1)
(a-6-3) edge (a-4-3)

(a-4-3) edge[bend right=15] (a-1-4)
(a-6-3) edge[bend right=25] (a-1-4);

\path[-stealth, bij, font=\scriptsize]
(a-4-1) edge node{$\sim$} (a-2-1)
(a-4-3) edge node{$\sim$} (a-2-3)

(a-1-2) edge node{$\sim$} (a-1-4)
(a-2-1) edge node{$\sim$} (a-2-3)
(a-4-1) edge node{$\sim$} (a-4-3)
(a-6-1) edge node{$\sim$} (a-6-3)

(a-1-4) edge node{$\sim$} node[below=2pt]{$\Psi$} (a-1-6);
\end{tikzpicture}
\end{small}
\end{center}

Since our goal is to identify $\Rep(G)$ with a category of universal $\CD$-modules, we must now study the essential image of $\QCoh{\varietiesc{c}}$ under $\Psi$ for each $c \in \BN$. 

Suppose that $M \in \QCoh{\varietiesc{c}} \emb \QCoh{\varietiesc{\infty}}$. Then we know that $M$ consists of the data of a sheaf
\begin{align*}
M_{X \rightleftarrows S} \in \QCoh{S}
\end{align*}
for each object $(X \rightleftarrows S) \in \prestackVbc{\infty}$, together with isomorphisms 
\begin{align*}
M(f, \alpha): f^* M_{X^\prime \rightleftarrows S^\prime} \EquivTo M_{X \rightleftarrows S} \in \QCoh{S},
\end{align*}
for any commutative diagram in $\PreStk$ of the form
\begin{center}
\begin{tikzpicture}[>=angle 90]
\matrix(e)[matrix of math nodes, row sep=1.5em, column sep=2em, text height=1.5ex, text depth=0.25ex]
{     &    &                     & & &       &  V_\alpha  &                      &         \\
  S_2 &    &\prestackVbc{\infty} & & &   X_1 &            & S_1 \times_{S_2} X_2 & X_2     \\
      & {} &                     & & &       &  S_1       &                      &         \\
  S_1 &    &   {}                & & &   S_1 &            & S_1                  & S_2.    \\};
\path[->, font=\scriptsize]
 (e-2-1) edge node[above]{$(\pi_2,\sigma_2)$} (e-2-3)
 (e-4-1) edge node[left]{$f$} (e-2-1) 
         edge node[below right]{$(\pi_1, \sigma_1)$} (e-2-3)
         
 (e-3-2) edge[-implies, double, double distance=2pt]node[right]{$\ \alpha$} (e-2-1)        

 (e-2-6) edge[bend left=10] node[right]{$\pi_1$} (e-4-6)
 (e-4-6) edge[bend left=10] node[left]{$\sigma_1$} (e-2-6)
 (e-1-7) edge[bend left=10] node[right]{$\rho_\alpha$} (e-3-7)
 (e-3-7) edge[densely dotted, bend left=10] node[left]{$\tau_\alpha$} (e-1-7)
 (e-2-8) edge[bend left=10] node[right]{$f^*\pi_2$}(e-4-8)
 (e-4-8) edge[bend left=10] node[left]{$f^*\sigma_2$}(e-2-8)
 (e-1-7) edge node[above left]{$\phi_{\alpha}$} (e-2-6)
         edge node[above right]{$\psi_{\alpha}$} (e-2-8)
 (e-2-9) edge[bend left=10] node[right]{$\pi_2$} (e-4-9)
 (e-4-9) edge[bend left=10] node[left]{$\sigma_2$} (e-2-9)
 
 (e-2-8) edge node[above]{$\pr_{X_2}$} (e-2-9)
 (e-4-8) edge node[below]{$f$} (e-4-9);
\path[-]
 (e-3-7) edge[double, double distance=2pt] (e-4-6)
         edge[double, double distance=2pt] (e-4-8);
\end{tikzpicture}
\end{center}

The fact that $M$ is an object of $\QCoh{\prestackVbc{c}}$ amounts to the condition that $M(f, \alpha) = M(f, \alpha^\prime)$ whenever any representatives of $\alpha$ and $\alpha^\prime$ are $(c)$-equivalent (c.f. Remark \ref{rmk: not well-defined w.r.t. c-equivalence}). Let us consider the implications of this condition for the corresponding universal $\CD$-module $\Psi(M)$. Suppose that we have two \'etale maps $f_1, f_2$ of $n$-dimensional families over some base scheme $S$
\begin{center}
\begin{tikzpicture}[>=angle 90]
\matrix(f)[matrix of math nodes, row sep=3em, column sep=3em, text height=1.5ex, text depth=0.25ex]
{X   & X^\prime   \\
 S & S, \\};
\path[-]
 (f-2-1) edge[double, double distance = 2pt] (f-2-2);
\path[->, font=\scriptsize]
 (f-1-1) edge node[above]{$f_i$} (f-1-2)
 
 (f-1-1) edge[bend left=10] (f-2-1)
 (f-2-1) edge[densely dotted, bend left=10] node[left]{$\sigma$} (f-1-1)
 
 (f-1-2) edge[bend left=10] (f-2-2)
 (f-2-2) edge[bend left=10] node[left]{$\sigma^\prime$} (f-1-2)
; 
\end{tikzpicture}
\end{center}
inducing the same isomorphism of the $c$th infinitesimal neighbourhoods of $S$ in $X$ and $X^\prime$:
\begin{align*}
f_1^\cth = f_2^\cth : X^\cth_S \EquivTo X^{\prime \cth}_S.
\end{align*}

Then we obtain isomorphisms
\begin{align*}
\Psi(M)(f_i): \Psi(M)(X/S) \EquivTo f_{i,X/S}^* \Psi(M)(X^\prime/S) \in \QCoh{(X/S)_\dR}, \ i=1,2,
\end{align*}
and pulling back along $\overline{\sigma}: S \to (X/S)_\dR$ yields maps
\begin{align*}
\overline{\sigma}^*\Psi(M)(f_i): \overline{\sigma}^*\Psi(M)(X/S) \EquivTo \overline{\sigma}^*f_{i,X/S}^*\Psi(M)(X^\prime/S) \in \QCoh{S}, \ i=1,2.
\end{align*}
Note that $f_{1,X/S} \circ \overline{\sigma} = \overline{\sigma^\prime} = f_{2,X/S} \circ \overline{\sigma}$, and so the maps $\overline{\sigma}^*\Psi(M)(f_i)$ are maps between the same sheaves on $S$. From the definition of $\Psi(M)$, we identify
\begin{align*}
\overline{\sigma}^*\Psi(M)(X/S) = M_{X \rightleftarrows S}; \qquad \overline{\sigma^\prime}^* \Psi(M)(X^\prime/S\prime) = M_{X^\prime \rightleftarrows S},
\end{align*}
and we see that for $i=1,2$ the map $\overline{\sigma}^*\Psi(M)(f_i)$ is given by the structure isomorphism $M(\id_S, \alpha_i)$ of M, where $\alpha_i$ is the isomorphism in $\prestackVbc{\infty}$ represented by the common \'etale neighbourhood
\begin{center}
\begin{tikzpicture}[>=angle 90]
\matrix(c)[matrix of math nodes, row sep=2em, column sep=2em, text height=1.5ex, text depth=0.25ex]
{   & X & \\
 X  &   & X^\prime \\
    & S. & \\};
\path[->, font=\scriptsize]
(c-1-2) edge node[above left]{$\id_X$} (c-2-1)
(c-1-2) edge node[above right]{$f_i$} (c-2-3)
(c-2-1) edge[bend left=10] (c-3-2)
(c-3-2) edge[bend left=15] (c-2-1)
(c-2-3) edge[bend right=10] (c-3-2)
(c-3-2) edge[bend right=15] (c-2-3)
(c-1-2) edge[bend left=10] (c-3-2);
\path[densely dotted, ->]
(c-3-2) edge[bend left=10] (c-1-2);
\end{tikzpicture}
\end{center}
Since these two common \'etale neighbourhoods are $(c)$-equivalent, it follows that $M(\id_S, \alpha_1) = M(\id_S, \alpha_2)$.

Motivated by this observation, we formulate the following definition:
\begin{defn}
A universal $\CD$-module $\SF$ is \emph{of $c$th order} if whenever we have two \'etale morphisms $f_1, f_2$ of $n$-dimensional families over $S$ such that 
\begin{align*}
f_1^\cth = f_2^\cth : X^\cth_S \EquivTo X^{\prime \cth}_S,
\end{align*}
 then we have that
\begin{align*}
\overline{\sigma}^*\SF(f_1) = \overline{\sigma}^*\SF(f_2): \overline{\sigma}^*\SF_{X/S} \EquivTo \overline{\sigma^\prime}^*\SF_{X^\prime/S}.
\end{align*}
\end{defn}
We let $\univcatDc{c}$ denote the full subcategory of $\univcatD$ whose objects are the universal $\CD$-modules of $c$th order.

\begin{prop}
The functor $\Psi: \QCoh{\prestackVbc{\infty}} \EquivTo \univcatD$ restricts to an equivalence
\begin{align*}
\Psi^\cth: \QCoh{\prestackVbc{c}} \EquivTo \univcatDc{c}.
\end{align*}
\end{prop}

\begin{proof}
By the above discussion, the restriction $\Psi^\cth$ of $\Psi$ to $\QCoh{\prestackVbc{c}}$ is a fully faithful embedding into $\univcatDc{c}$. To complete the proof, it suffices to show that the functor $\theta: \univcatD \EquivTo \QCoh{\prestackVbc{\infty}}$ restricts to a functor
\begin{align*}
\theta^\cth: \univcatDc{c} \to \QCoh{\prestackVbc{c}}.
\end{align*}
So let us assume that $\SF \in \univcatDc{c}$, and consider the quasi-coherent sheaf $\theta(\SF) \in \QCoh{\prestackVbc{\infty}}$. Suppose that we have a diagram in $\PreStk$ of the form:
\begin{center}
\begin{tikzpicture}[>=angle 90]
\matrix(i)[matrix of math nodes, row sep=3em, column sep=3em, text height=1.5ex, text depth=0.25ex]
{ S^\prime & \prestackVbc{\infty}. \\
 S & \\};
\path[->, font=\scriptsize]
 (i-1-1) edge node[above]{$(\pi^\prime, \sigma^\prime)$} (i-1-2)
 (i-2-1) edge node[left]{$f$} (i-1-1)
         edge node[below, sloped]{$(\pi, \sigma)$} (i-1-2);
\end{tikzpicture}
\end{center} 
Assume in addition that we have two isomorphisms 
\begin{align*}
\alpha_i \in \Hom_{\prestackVbc{\infty}(S)}\left( (\pi, \sigma), (\pi^\prime, \sigma^\prime) \circ f \right)
\end{align*}
which make this diagram commute, and which are $(c)$-equivalent, although not necessarily $(\infty)$-equivalent. In order to show that $\theta(\SF)$ is a quasi-coherent sheaf on $\prestackVbc{c}$, we need to show that the structure isomorphisms $\theta(\SF)(f, \alpha_i)$ agree with each other. 

Let us choose representatives of the isomorphisms $\alpha_i$ as follows:
 \begin{center}
\begin{tikzpicture}[>=angle 90]
\matrix(e)[matrix of math nodes, row sep=1.5em, column sep=3em, text height=1.5ex, text depth=0.25ex]
{   &  V_i  &                              & \\
 X  &       & S \times_{S^\prime} X^\prime & X^\prime \\
    &  S    &                              & \\
 S  &       & S                            & S^\prime. \\};
\path[->, font=\scriptsize]
 (e-2-1) edge[bend left=10] (e-4-1)
 (e-4-1) edge[bend left=10] (e-2-1)
 (e-1-2) edge[bend left=10] node[right]{$\rho_i$} (e-3-2)
 (e-2-3) edge[bend left=10] (e-4-3)
 (e-4-3) edge[bend left=10] (e-2-3)
 (e-2-4) edge[bend left=10] (e-4-4)
 (e-4-4) edge[bend left=10] (e-2-4)
 (e-1-2) edge node[above left]{$\phi_i$} (e-2-1)
         edge node[above right]{$\psi_i$} (e-2-3)
 (e-2-3) edge (e-2-4)
 (e-4-3) edge node[below]{$f$} (e-4-4);
\path[-]
 (e-3-2) edge[double, double distance=2pt] (e-4-1)
         edge[double, double distance=2pt] (e-4-3);
\path[densely dotted, ->, font=\scriptsize]
 (e-3-2) edge[bend left=10] node[left]{$\tau_i$} (e-1-2);
\end{tikzpicture}
\end{center}

Then it suffices to show that 
\begin{align*}
\overline{\tau_1}^* \left(\SF(\phi_1, \id_S) \circ \SF(\psi_1, \id_S)^{-1} \right) = \overline{\tau_2}^* \left(\SF(\phi_2, \id_S) \circ \SF(\psi_2, \id_S)^{-1} \right).
\end{align*}

Equivalently, we can compose the representatives of $\alpha_1$ and $\alpha_2^{-1}$ to obtain a common \'etale neighbourhood which we'll call $\alpha_{12}$:
\begin{center}
\begin{tikzpicture}
[>=angle 90, 
cross line/.style={preaction={draw=white, -, line width=4pt}}]
\matrix(g)[matrix of math nodes, row sep=1.5em, column sep=3em, text height=1.5ex, text depth=0.25ex]
{     &     &                             & V_1 \times_{S \times_{S^\prime} X^\prime} V_2 &     &     \\
      & V_1 &                             &                                               & V_2 &     \\
 X_1  &     & S\times_{S^\prime} X^\prime & S                                             &     & X_1 \\
      & S   &                             &                                               & S   &     \\
 S    &     & S                           &                                               &     & S.   \\};
\path[-]         
 (g-4-2) edge[ double, double distance= 2pt] (g-5-1)
         edge[ double, double distance= 2pt] (g-5-3)
 (g-3-4) edge[ double, double distance= 2pt] (g-4-2)
         edge[ double, double distance= 2pt] (g-4-5)
 (g-4-5) edge[ double, double distance= 2pt] (g-5-3)
         edge[ double, double distance= 2pt] (g-5-6);
\path[densely dotted, ->, font=\scriptsize]
 (g-4-2) edge[bend left=10] (g-2-2)  
 (g-4-5) edge[bend left=10] (g-2-5)
 (g-3-4) edge[bend left=10] node[left]{$\tau_{12}$} (g-1-4);
\path[->, font=\scriptsize]
 (g-3-1) edge[bend left=10] (g-5-1)
 (g-5-1) edge[bend left=10] (g-3-1)
 (g-2-2) edge[bend left=10] (g-4-2)
 
 (g-5-3) edge[draw=white, -, line width=4pt, bend left=10] (g-3-3)
 (g-3-3) edge[cross line, bend left=10] (g-5-3)
 (g-5-3) edge[bend left=10] (g-3-3)
 (g-1-4) edge[bend left=10] (g-3-4)

 (g-2-5) edge[bend left=10] (g-4-5)

 (g-3-6) edge[bend left=10] (g-5-6)
 (g-5-6) edge[bend left=10] (g-3-6)

 (g-2-2) edge node[above left]{$\phi_1$}(g-3-1)
         edge node[above right]{$\psi_1$}(g-3-3)
 (g-1-4) edge node[above left]{$\pr_{V_1}$} (g-2-2)
         edge node[above right]{$\pr_{V_2}$} (g-2-5)
 (g-2-5) edge[cross line] node[left=12pt, inner sep=05.ex]{$\psi_2$}(g-3-3)
         edge node[above right]{$\phi_2$} (g-3-6);
\end{tikzpicture}
\end{center}
Now we need to show that 
\begin{align*}
\overline{\tau_{12}}^* \left(\SF(\phi_1 \circ \pr_{V_1}, \id_S) \circ \SF(\phi_2 \circ \pr_{V_2}, \id_S)^{-1} \right) = \id_{\overline{\tau}^* \SF(X/S)}.
\end{align*}
We can use the fact that $\SF$ is of $c$th order: it suffices to show that $\alpha_{12}$ is $(c)$-equivalent to the identity. But of course
\begin{align*}
\left(\phi_1 \circ \pr_{V_1} \right)^\cth \circ \left(\left(\phi_2 \circ \pr_{V_2} \right)^\cth\right)^{-1} &= \phi_1^\cth \circ \pr_{V_1}^\cth \circ \left(\pr_{V_2}^\cth \right)^{-1} \circ \left(\phi_2^\cth \right)^{-1} \\
 &= \phi_1^\cth \circ \left(\psi_1^\cth\right)^{-1} \circ \psi_2^\cth \circ \left(\phi_2^\cth\right)^{-1} \\
 &= \id_{X^\cth_S}.\\
\end{align*}
So $\theta(\SF)$ is indeed an object of the subcategory $\QCoh{\prestackVbc{c}}$, and the proof is complete. \qed
\end{proof}

The following is immediate:
\begin{corollary}
We have an equivalence of categories
\begin{align*}
\univcatDc{c} \simeq \Rep(G^\cth).
\end{align*}
\end{corollary}

We have the following nested sequence of subcategories of $\univcatD$:
\begin{align*}
\ldots \emb \univcatDc{c} \emb \univcatDc{c+1} \emb \ldots \emb \univcatD.
\end{align*}

\begin{defn}
Let 
\begin{align*}
\univcatDconv \defeq \colim_{c \in \BN} \univcatDc{c}.
\end{align*}
It is a full subcatgory of $\univcatD$. An object of $\univcatDconv$ will be called a \emph{convergent} universal $\CD$-module of dimension $n$. 
\end{defn}

\begin{corollary}
The essential image of $\displaystyle\colim_{c \in \BN} \QCoh{\varietiesc{c}}$ in $\univcatD$ is $\univcatDconv$. We have an equivalence of categories
\begin{align*}
\Rep(G) \EquivTo \univcatDconv.
\end{align*}
\end{corollary}

We can similarly define the category $\univcatc{c}$ of $c$th-order universal $\CO$-modules and can show that 
\begin{align*}
\univcatc{c} \simeq \QCoh{\pointedvarietiesc{c}}
\end{align*}
for any $c \in \BN$. Letting $\univcatOconv \defeq \displaystyle\colim_{c \in \BN} \univcatc{c}$, we obtain the following:

\begin{prop}
We have an equivalence of categories
\begin{align*}
\Rep(K) \EquivTo \univcatOconv.
\end{align*}
\end{prop}

\subsection{Universal \texorpdfstring{$\CO$}{O}-modules of ind-finite type}
\label{subsec: ind-finite universal modules}
In this subsection, we take a different approach, beginning with the identification
\begin{align*}
\Rep(G) \EquivTo \Rep^{K^\et{\text{-}\lf}}(G^\et).
\end{align*}
As has generally been the case when working with the groups $G^\et$ and $K^\et$, we will begin by studying the picture for $K^\et$, and then extend our results to $G^\et$.

Our first step is to identify the subcategory of $\QCoh{\pointedvarietiesc{\infty}}$ corresponding to the locally finite representations of $K$. First notice that a representation $V$ of $K^\et$ is finite-dimensional if and only if the corresponding sheaf $M$ on $\QCoh{\BK^\et}$ and hence on $\QCoh{\pointedvarietiesc{\infty}}$ is of \emph{finite type}: that is, for every $S=\Spec{R} \to \BK^\et \simeq \pointedvarietiesc{\infty}$, the corresponding sheaf $M_S \in \QCoh{S}$ is of finite type.

\begin{rmk}
This is equivalent to requiring the sheaf $M_S$ to be of finite presentation, and in fact to be locally free. This follows from three facts: these properties are all equivalent for $\QCoh{\pt} \simeq \Vect$; they are preserved by pullback;  and every morphism $S \to \prestackPVc{\infty}$ factors through $S \to \pt$.

However, it is not equivalent to requiring the sheaf $M_S$ to be coherent, because coherence is not preserved under arbitrary pullbacks. For example, if $V=k$ is the trivial representation, with $M$ the corresponding sheaf on $S$, then $M_S = R$ for every $S = \Spec{R}$. If $R$ is a $k$-algebra which is not coherent as an $R$-module, then $M_S$ is finitely generated, but it is not coherent.
\end{rmk}

It follows that the essential image of $\Rep^\lf(K^\et)$ in $\QCoh{\pointedvarietiesc{\infty}}$ is the full subcategory generated by those sheaves $M \in \QCoh{\pointedvarietiesc{\infty}}$ which can be written as a union $M = \bigcup_{i} M_i$ of sheaves $M_i$ of finite type. 

\begin{defn}
Let $\CY$ be a prestack, and $M \in \QCoh{\CY}$ be a sheaf that can be written as the colimit of its subsheaves of finite type. Then we say that $M$ is of \emph{ind-finite type}. We denote the full subcategory of sheaves of ind-finite type by $\QCohIF{\CY}$. 
\end{defn} 

Of course, for any sheaf $M \in \QCoh{\CY}$ and for any $S = \Spec{R}$, we can always write $M_S$ as a union of finitely generated subsheaves; however, this cannot always be done in a way compatibly with all pullbacks and automorphisms of $S$-points of $\QCoh{\CY}$. 
\begin{eg}
Recall the notation of Section \ref{subsec: non-locally finite representations}. Let $\CY = B \A{\infty}$, and let $M \in \QCoh{\CY}$  be the sheaf corresponding to the regular representation $V$. Then $M$ is not an object of $\QCohIF{\CY}$. 
\end{eg}

We do not know if $\QCohIF{\pointedvarietiesc{\infty}}$ is equal to $\QCoh{\pointedvarietiesc{\infty}}$, or if it is a proper subcategory. By construction, this question is equivalent to the question of whether $\Rep^{\text{l.f.}}(K^\et)$ is a proper subcategory of $\Rep(K^\et)$.  

Now we can study the essential image of $\QCohIF{\pointedvarietiesc{\infty}}$ in $\univcat$. It is the full subcategory whose objects are those universal $\CO$-modules which can be written as a union of their submodules of finite type:
\begin{align*}
\SF = \bigcup_i \SF_i, 
\end{align*}
where $\SF_i \in \univcat$ is such that for any $X/S$ smooth of dimension $n$, $\SF_{i, X/S} \in \QCoh{X}$ is of finite type. This is equivalent to requiring $\SF_{i, X/S}$ to be locally free of finite rank: this is because the translation-invariance of $\SF$ ensures that $\SF_{i, X/S}$ is of constant rank. 

\begin{defn}
If $\SF$ is a universal $\CO$-module satisfying this condition, then we shall say that $\SF$ is a universal $\CO$-module of \emph{ind-finite type}. We denote the subcategory of universal $\CO$-modules of ind-finite type by $\univcatOif$.
\end{defn}

The following result is clear by definition:
\begin{prop}
The equivalence $\Rep(K^\et) \EquivTo \univcat$ restricts to give an equivalence of categories
\begin{align*}
\Rep(K) \EquivTo \univcatOif.
\end{align*}
\end{prop}

\subsection{Universal \texorpdfstring{$\CD$}{D}-modules of ind-finite type}
\label{subsec: ind-finite universal D-modules}
Similarly, it is clear that the essential image in $\QCoh{\varietiesc{\infty}}$ of $\Rep{G}$ is the subcategory of sheaves $M \in \QCoh{\varietiesc{\infty}}$ such that the pullback of $M$ along the map
\begin{align*}
\pointedvarietiesc{\infty} \to \varietiesc{\infty}
\end{align*}
is of ind-finite type. We denote this category by
\begin{align*}
\QCoh{\varietiesc{\infty}}_{\pointedvarietiesc{\infty}\text{-ift}} \emb \QCoh{\varietiesc{\infty}}.
\end{align*}
Again, we do not know whether this is in fact a proper subcategory.

We can again characterise the image of this subcategory in $\univcatD$:

\begin{defn}
A universal $\CD$-module $\SF$ is of \emph{ind-finite type} if it is of ind-finite type when regarded as a universal $\CO$-module. We denote the full subcategory of universal $\CD$-modules of ind-finite type by $\univcatDif$. 
\end{defn}

\begin{rmk}
Let us emphasise that the decomposition of $\SF$ into subsheaves of finite type only needs to respect the $\CO$-module structures; we do not expect the subsheaves $\SF_i \subset \SF$ to be sub-$\CD$-modules of $\SF$. 
\end{rmk}

\begin{prop}
The equivalence $\Rep(G^\et) \EquivTo \univcatD$ restricts to an equivalence of subcategories
\begin{align*}
\Rep(G) \EquivTo \univcatDif.
\end{align*}
\end{prop}

Combining this with our previous results (essentially, travelling to the left and then back again to the right along the middle two rows of the main diagram in Figure 1), we deduce the following:

\begin{prop}
A universal $\CO$- or $\CD$-module is of ind-finite type if and only if it is convergent. 
\end{prop}

Our proposal is that these categories of convergent universal modules, rather than the full categories of universal modules as defined in \cite{BD1}, are the more natural categories with which to work. Of course, we do not know if in fact the categories are equivalent. 

\section{Remarks on \texorpdfstring{$\infty$-}{infinity }categories}
\label{sec: infinity categories}

In this section, we extend our results to the $\infty$-categories of representations and universal modules. In fact, none of these categories are previously well-established in the literature, so we have some freedom to choose our definitions to allow our results to extend. We will provide some justification of our choices as we proceed, but the very fact that our results extend so naturally is in itself a good defence for these definitions.

One motivation for working with $\infty$-categories rather than ordinary categories is the following. Recall that we are interested in the study of universal chiral algebras of dimension $n$, under the hypothesis that these give the correct notion of an $n$-dimensional vertex algebra. In particular, a universal chiral algebra of dimension $n$ is a universal $\CD$-module. Although the (ordinary) categories of universal $\CD$-modules and representations of $G$ may be interesting in their own right, if we wish to work with universal chiral algebras of dimension two or higher, we immediately see that it is necessary to work in the derived setting. For example, if we insist on remaining in the abelian categories, the definitions of universal chiral algebras and Lie $\star$ algebras (as in \cite{FG}) become equivalent.

\subsection{Conventions}
\label{subsec: conventions for dg categories}
Henceforth all categories of sheaves, modules, vector spaces, and $\CD$-modules will be assumed to be the $(\infty,1)$-categories, unless otherwise specified. We shall appropriate notation established earlier in the paper for abelian categories without further decoration by symbols such as ``d.g.'' or ``$\infty$''; when we wish to refer to the abelian hearts of these categories, we shall indicate it with a superscript $\heartsuit$. When we say ``category'', we mean cocomplete $(\infty,1)$-category; it is in this sense that we take colimits, for example.

\subsection{\texorpdfstring{$\infty$-}{Infinity }categories of universal modules and representations}
\label{subsec: infinity categories of universal D-modules}

We can begin by na\"ively extending the definition of a universal $\CD$-module (and similarly a universal $\CO$-module), repeating the definition \ref{defn: universal D module} in the setting of $\infty$-categories. (We will carry this out explicitly for $c$th-order $\CD$-modules in \ref{subsec: infinity category of convergent universal D-modules}.) The functors $\theta$ and $\Psi$ of Theorem \ref{thm: universal D-modules are quasi-coherent sheaves} admit $\infty$-categorical extensions, and provide an equivalence between the categories $\QCoh{\varietiesc{\infty}}$ and $\univcatD$. The equivalence between $\varietiesc{\infty}$ and $BG^\et$ is purely geometric, valid before considering categories or $\infty$-categories of sheaves, and hence we still have the equivalence
\begin{align*}
\QCoh{BG^\et} \simeq \QCoh{\varietiesc{\infty}}.
\end{align*}
In fact, apart from the first column, the entire content of the main diagram lifts immediately from  $\Cat$ to $\DGCat$ with no serious modifications. 

However, it is not immediately clear what we should take for the $\infty$-category of representations of our groups. For an algebraic group $H$, we take
\begin{align*}
\Rep(H) \defeq \QCoh{BH},
\end{align*}
as in e.g. 6.4.3, \cite{DG}. Since $BH$ is a smooth Artin stack, we can show that  
\begin{align*}
\Upsilon_{BH}: \QCoh{BH} \to \IndCoh{BH}
\end{align*}
is an equivalence of categories, so we could also have defined 
\begin{align*}
\Rep(H) = \IndCoh{BH}.
\end{align*}
Recall that the functors $\Upsilon$ intertwine the $*$-pullback on $\QCoh{\bullet}$ with the $!$-pullback on $\IndCoh{\bullet}$. 

On the other hand, given a pro-algebraic group $H$, we can consider the category $\QCoh{BH}$, but the category $\IndCoh{BH}$ is not defined, because the stack $BH$ is not locally of finite type. It can, however, be written as the limit of stacks which are locally of finite type, and this leads us to a second potential definition for the category of representations. 

More precisely, recall that if we write $H = \lim_{i} H_i$, with $H_i$ finite-dimensional quotients, and all maps $H_i \to H_j$ smooth surjections of algebraic groups, then by Proposition \ref{prop: representations of pro-algebraic groups are continuous} 
\begin{align*}
\Rep^\heartsuit(H) \simeq \colim_{i} \Rep^\heartsuit(H_i).
\end{align*}
Motivated by this fact, it is natural to consider the category
\begin{align*}
\colim_i \Rep(H_i) = \colim_i \QC^*(BH_i) \simeq \colim_i \IC^!(BH_i),
\end{align*}
and unlike in the abelian categories, this category is not all of $\QCoh{BH}$. Instead, we think of $\QCoh{BH}$ as the DG-category of representations of $H$, and $\colim_i \Rep(H_i)$ as the subcategory of representations of $H$ which are locally finite. In the $\infty$-categorical setting, this condition is not automatically satisfied, but it is one which we are happy to impose. In other words, we set
\begin{align*}
\Rep(H) \defeq \colim_i \Rep(H_i). 
\end{align*} 

Similarly, for a group formal scheme that can be written as $L = \lim_{i} L_i$ (such as the group $G$ of automorphisms of the formal disc) with $H=L_\red = \lim_i L_{i, \red}$ a pro-algebraic group, we set
\begin{align*}
\Rep(L) \defeq \colim_{i} \Rep(L_i). 
\end{align*}

Although the $L_i$ are themselves indschemes, they are of finite type, and hence $BL_i$ is locally of finite type and $\Rep(L_i)$ is given by $\QCoh{BL_i} \simeq \IndCoh{BL_i}$. 

We do not know how to define a corresponding category for an arbitrary group-valued prestack. In the case of $G^\et$ and $K^\et$, for example, the stacks $BG^\et$ and $BK^\et$ seem to be quite intractable. Fortunately the relative Artin approximation theorem and its corollaries from Section \ref{subsec: applications of relative artin approximation} allow us to approximate these stacks using the stacks $BG^\cth$ and $BH^\cth$, which are easier to work with. By restricting our attention to representations which are sufficiently finite-dimensional in flavour, we can avoid working with the stacks $BG^\et$ and $BK^\et$ entirely. The cost, however, is that the corresponding categories of representations do not correspond to the categories of arbitrary universal modules, but only those of convergent universal modules. In fact, though, we view this as an advantage rather than a cost. 

As a consequence of this discussion, we shall henceforth ignore the back two rows of the main diagram, and shall only work with the front part of the diagram, which can be described entirely using stacks which are locally of finite type. Thus far, we have the following:

\begin{center}
\begin{tikzpicture}[>=angle 90, bij/.style={above,sloped, inner sep=0.5pt}, cross line/.style={preaction={draw=white, -, line width=6pt}}] 
\matrix(a)[matrix of math nodes, row sep=2em, column sep=.8em, text height=2.5ex, text depth=0.25ex]
{\Rep(G)                                        &&                                                  & &                                   \\
\displaystyle{\colim_{c \in \BN} \Rep\left(G^\cth\right)}  && \displaystyle{\colim_{c \in \BN} \QCoh{BG^\cth}} & & \displaystyle{\colim_{c \in \BN} \QCoh{\varietiesc{c}}}        \\
\displaystyle{\Rep\left(G^\cth\right)                                   } && \displaystyle{\QCoh{BG^\cth}}                                   & & \QCoh{\varietiesc{c}}   .          \\};

\path[-, font=\scriptsize]
(a-1-1) edge[double, double distance=2pt] (a-2-1) 
(a-2-1) edge[double, double distance=2pt] (a-2-3)
(a-3-1) edge[double, double distance=2pt] (a-3-3);

\path[right hook-stealth, font=\scriptsize]
(a-3-1) edge (a-2-1)
(a-3-3) edge (a-2-3)
(a-3-5) edge (a-2-5);

\path[-stealth, bij]
(a-2-3) edge node{$\sim$} (a-2-5)
(a-3-3) edge node{$\sim$} (a-3-5);
\end{tikzpicture}
\end{center}

Recall the discussion in Remark \ref{rmk: left vs right relative D-modules} on the use of the categories $\QCoh{\bullet}$ as compared to $\IndCoh{\bullet}$. At that stage, we defended the use of $\QCoh{(X/S)_\dR}$, corresponding to relative left $\CD$-modules, rather than its more well-studied and better-behaved counterpart $\IndCoh{(X/S)_\dR}$, corresponding to relative right $\CD$-modules. Our reasons were threefold: we wished to remain consistent with the definitions of Beilinson and Drinfeld and to work with abelian categories rather than $\infty$-categories, and moreover we needed to work with prestacks which were not locally of finite type, and so we could not rely on the theory of ind-coherent sheaves. 

Indeed, all of the stacks appearing in the back rows of the main diagram from Figure 1 are of infinite type and hence not well-suited to being studied using the theory of ind-coherent sheaves---but by restricting our attention to the category of convergent universal $\CD$-modules, as we have just decided to do, we can avoid using these stacks, instead using only the stacks in the front rows of the diagram, which are locally of finite type. In particular, we can work with ind-coherent sheaves on these prestacks.

Furthermore, we argued that the correct notion of universal $\CD$-module should include the convergence condition, regardless of whether it agrees with the definition given by Beilinson and Drinfeld even in the abelian setting. In other words, our three motivations for working with quasi-coherent rather than ind-coherent sheaves have disappeared, and consequently we now feel free to use the better-behaved theory of ind-coherent sheaves in the $\infty$-categorical setting and to define an $\infty$-category which will correspond to universal right $\CD$-modules.

By taking ind-coherent sheaves rather than quasi-coherent sheaves at each stage, we obtain the ``right $\CD$-module'' version of the above diagram. However, note that each of the stacks appearing in the diagram is (equivalent to) a smooth Artin stack, so that the categories of quasi-coherent sheaves and ind-coherent sheaves are in fact equivalent via the funtors $\Upsilon$. In other words, the diagrams are actually equivalent, termwise, and the functors $\Upsilon$ between the terms intertwine the morphisms of the diagram as well. 

\subsection{\texorpdfstring{$\infty$}{Infinity}-categories of convergent universal modules}
\label{subsec: infinity category of convergent universal D-modules}
We have (equivalent) categories
\begin{align*}
\colim_{c \in \BN} \QCoh{\varietiesc{c}},\\
\colim_{c \in \BN} \IndCoh{\varietiesc{c}}
\end{align*}
which should, formally, correspond to categories of universal left and right $\CD$-modules, but which in flavour belong to the third column of the main diagram from Figure 1. In order to give an equivalent description of these categories in the language of the fourth column, of ``universal modules'', we must apply a construction analogous to that of the functor $\Psi$. We first study the $\infty$-categorical analogue of universal right $\CD$-modules of $c$th order; that is, we apply a version of the functor $\Psi$ to the category $\IndCoh{\varietiesc{c}}$. (The $\QCoh{\varietiesc{c}}$ setting is completely analogous.) We do this here only informally, as the full technical definition is no more enlightening. 

Given an object $M \in \IndCoh{\varietiesc{c}}$, we begin to argue as in the proof of Theorem \ref{thm: universal D-modules are quasi-coherent sheaves} in Section \ref{subsec: from quasi-coherent sheaves to universal D-modules} and obtain the following data:

\begin{enumerate}
\item For any $X \to S$ smooth of dimension $n$ with $S$ a scheme of finite type, we have
\begin{align*}
\SF(X/S) \in \IndCoh{(X/S)_\dR},
\end{align*}
given by the compatible family
\begin{align*}
\left\{ \SF(X/S)_{T \to (X/S)_\dR} \defeq M_{T \times_S X \rightleftarrows T} \in \IndCoh{T} \right\}_{T \in \Sch_{/(X/S)_\dR}}  .
\end{align*}
\item For any any pair $(X/S), (X^\prime,S^\prime)$ of smooth $n$-dimensional families, and for any fibrewise \'etale morphism $f:(X/S) \to (X^\prime / S^\prime)$, an isomorphism
\begin{align*}
\SF(f): \SF_{X/S} \EquivTo f_{X/S}^! \SF_{X^\prime/S^\prime}
\end{align*}
in $\IndCoh{(X/S)_\dR}$, defined for each $T$-point $T\to (X/S)_\dR$ to be equal to the compatibility isomorphism
\begin{align*}
M_{T \times_S X \rightleftarrows T} \EquivTo M_{T \times_{S^\prime} X^\prime \rightleftarrows T} \in \IndCoh{T}.
\end{align*}
The fact that $M$ depends only on morphisms in $\varietiesc{c}(T)$ (as compared to in $\varietiesc{\infty}(T)$) tells us that for any sections $\sigma: S \to X$, $\sigma: S^\prime \to X^\prime$ compatible with $f$ on the level of $S^\prime_\red$, the map
\begin{align*}
\overline{\sigma}^!\SF(f): \overline{\sigma}^! \SF(X/S) \EquivTo \overline{\sigma^\prime}^! \SF(X^\prime/S^\prime)
\end{align*}
depends only on the restriction of $f$ to the $c$th infinitesimal neighbourhood of $S \emb X$. 
\end{enumerate}

In Section \ref{subsec: from quasi-coherent sheaves to universal D-modules} we then showed that the isomorphisms $\SF(f)$ were compatible with composition. Since we are now working in $\infty$-categories, this compatibility is now a structure rather than a condition, and so we obtain additional data. The first few stages look like this: 
\begin{enumerate}[resume]
\item Given three smooth families with fibrewise maps between them
\begin{align*}
(X/S) \xrightarrow{f} (X^\prime /S^\prime) \xrightarrow{g} (X^{\prime\prime}/S^{\prime\prime}),
\end{align*}
we have a natural isomorphism
\begin{align*}
a_{f,g}: f_{X/S}^! \SF(g) \circ \SF(f) \Rightarrow \SF(g \circ f)
\end{align*}
of isomorphisms $\SF(X/S) \to (g \circ f)^! \SF(X^{\prime\prime}/S^{\prime\prime})$. (Note that we have omitted from our notation the canonical isomorphism $(g \circ f)^! \Rightarrow f^! \circ g^!$.) 

This natural isomorphism is defined for each $T$-point $T \to (X/S)_\dR$ to be the natural transformation between the two maps
\begin{align*}
M_{T \times_S X \rightleftarrows T} \to M_{T \times_{S^\prime} X^\prime \rightleftarrows T} \to M_{T \times_{S^{\prime\prime}} X^{\prime\prime} \rightleftarrows T}
\end{align*}
and 
\begin{align*}
M_{T \times_S X \rightleftarrows T} \to M_{T \times_{S^{\prime\prime}} X^{\prime\prime} \rightleftarrows T},
\end{align*}
which comes from the structure of $M$ as an object of $\IndCoh{\varietiesc{c}}$.

\item Given four smooth families with fibrewise maps between them 
\begin{align*}
(X/S) \xrightarrow{f} (X^\prime/S^\prime) \xrightarrow{g} (X^{\prime\prime}/S^{\prime\prime}) \xrightarrow{h} (X^{\prime\prime\prime} /S^{\prime\prime\prime}), 
\end{align*}
the data of (3) gives us two natural isomorphisms between the maps
\begin{align*}
f^! g^! \SF(h) \circ f^! \SF(g) \circ \SF(f) \text{ and }\SF(h \circ g \circ f)
\end{align*}
as follows:

\begin{center}
\begin{tikzpicture}[>=latex]
\matrix(b)[matrix of math nodes, row sep=2em, column sep=2em, text height=1.5ex, text depth=0.25ex]
{ f^! g^! \SF(h) \circ f^! \SF(g) \circ \SF(f) & f^! g^! \SF(h) \circ \SF(g \circ f) \\
  f^! \SF(h \circ g) \circ \SF(f)              & \SF(h \circ g \circ f). \\};

\path[-implies, font=\scriptsize]
(b-1-1) edge[double equal sign distance] node[above]{$a_{f,g}$} (b-1-2)
(b-1-2) edge[double equal sign distance] node[right]{$a_{g \circ f, h}$} (b-2-2)
(b-1-1) edge[double equal sign distance] node[left]{$a_{g,h}$} (b-2-1)
(b-2-1) edge[double equal sign distance] node[below]{$a_{f, h\circ g}$} (b-2-2);

\end{tikzpicture}
\end{center}

There is a 3-morphism $b_{f,g,h}$ making this diagram commute.
\item \ldots and so on \ldots
\end{enumerate}

\begin{defn}
A collection $\SF=\left(\{\SF(X/S)\}, \{\SF(f)\}, \{a_{f, g}\}, \{b_{f,g,h}\}, \ldots \right)$ as above will be called a \emph{universal right $\CD$-module of $c$th order}. Such objects form an $\infty$-category ${}^r\univcatDc{c}$, equivalent by construction to the category $\IndCoh{\varietiesc{c}}$.   
\end{defn}

\begin{defn}
The $\infty$-category of \emph{convergent universal right $\CD$-modules} is by definition the colimit
\begin{align*}
{}^r\univcatDconv \defeq \colim_{c \in \BN} {}^r\univcatDc{c}.
\end{align*}
\end{defn}

We think of an object of ${}^r\univcatDconv$ as a family 
\begin{align*}
\SF = \left(\{\SF(X/S)\}, \{\SF(f)\}, \{a_{f, g}\}, \{b_{f,g,h}\}, \ldots \right)
\end{align*}
as in the above description, except with the condition in (2) pertaining to $(c)$-equivalence for morphisms omitted; instead, we assume that $\SF$ has an exhaustive filtration by subobjects $\SF^\cth$, where for each $c$, $\SF^\cth$ does satisfy the condition in (2). 

\begin{rmk}
Of course, it is impossible to specify such an object completely in this manner. However, this description in terms of families of sheaves has a significantly more geometric feel than that of the category of representations of $G$. This is a particular aspect of the difference between the study of vertex algebras (living on the same side of the story as $\Rep(G)$) and the study of chiral algebras (living on the geometric side).
\end{rmk}

We have, immediately, the following diagram (with equivalences along the rows):
\begin{center}
\begin{tiny}
\begin{tikzpicture}[>=angle 90, bij/.style={above,sloped, inner sep=0.5pt}, cross line/.style={preaction={draw=white, -, line width=6pt}}] 
\matrix(a)[matrix of math nodes, row sep=2em, column sep=1em, text height=2.5ex, text depth=0.25ex]
{\Rep(G) &&  && && {}^r\univcatDconv \\
\displaystyle{\colim_{c \in \BN} \Rep\left(G^\cth\right)}  && \displaystyle{\colim_{c \in \BN} \IndCoh{BG^\cth}} & & \displaystyle{\colim_{c \in \BN} \IndCoh{\varietiesc{c}}} && \displaystyle{\colim_{c \in \BN} {}^r\univcatDc{c}}   \\
\displaystyle{\Rep\left(G^\cth\right) } && \displaystyle{\IndCoh{BG^\cth}}  & & \IndCoh{\varietiesc{c}} && {}^r\univcatDc{c}. \\};

\path[-, font=\scriptsize]
(a-1-1) edge[double, double distance=2pt] (a-2-1) 
(a-2-1) edge[double, double distance=2pt] (a-2-3)
(a-3-1) edge[double, double distance=2pt] (a-3-3)
(a-1-7) edge[double, double distance=2pt] (a-2-7);

\path[right hook-stealth, font=\scriptsize]
(a-3-1) edge (a-2-1)
(a-3-3) edge (a-2-3)
(a-3-5) edge (a-2-5)
(a-3-7) edge (a-2-7);

\path[-stealth, bij]
(a-2-3) edge (a-2-5)
(a-3-3) edge (a-3-5)
(a-2-5) edge (a-2-7)
(a-3-5) edge (a-3-7)
(a-1-1) edge[dashed] (a-1-7);

\end{tikzpicture}
\end{tiny}
\end{center}

We have a completely analogous diagram for convergent universal left $\CD$-modules, using quasi-coherent sheaves. Because all of the categories in the first three columns are equivalent whether we use $\IndCoh{\bullet}$ or $\QCoh{\bullet}$, we deduce that the categories of universal right and left $\CD$-modules are equivalent as well. 

This is somewhat surprising: we do not have, in general, that the categories $\QCoh{(X/S)_\dR}$ and $\IndCoh{(X/S)_\dR}$ are equivalent. We can only say that 
\begin{align*}
\Upsilon_{(X/S)_\dR}: \QCoh{(X/S)_\dR} \to \IndCoh{(X/S)_\dR} 
\end{align*}
is a fully faithful embedding. However, as a consequence of the convergence condition, we can see that any universal family 
\begin{align*}
(X/S) \mapsto \SF(X/S) \in \IndCoh{(X/S)_\dR}
\end{align*}
will actually take values in the essential image of these functors $\Upsilon$. Consequently, the categories of convergent universal right and left $\CD$-modules are in fact canonically equivalent. 

A final remark on the convergence condition is the following: even in the right $\CD$-modules setting, where we do work with ind-coherent sheaves,  we still do not obtain a category whose objects are all compatible families of ind-coherent sheaves (or relative right $\CD$-modules) indexed by smooth $n$-dimensional families $X/S$. Instead, what we obtain is, informally, closer to the ind-completion of a category of universal families of coherent sheaves. This is close in spirit to the description of convergent universal $\CD$-modules as ind-finite families of modules from Section \ref{subsec: ind-finite universal D-modules}.

\setcounter{section}{0}
\renewcommand{\thesection}{\Alph{section}}
\section{Appendix: Proof of compatibility of \texorpdfstring{$\theta$}{theta} with composition}
\label{appendix}
In this section, we show that for $\SF \in \univcatD$ the assignment $(f, \alpha) \mapsto \theta(\SF)(f, \alpha)$ is compatible with composition. 

Suppose that we have a commutative diagram as follows
\begin{center}
\begin{tikzpicture}[>=angle 90]
\matrix(e)[matrix of math nodes, row sep=.8em, column sep=3em, text height=1.5ex, text depth=0.25ex]
{ S_3 &    &  \\
      & {} & \\
  S_2 & {} & \prestackVbc{\infty} \\
      & {} &  \\
  S_1 &    &  \\};
\path[->, font=\scriptsize]
 (e-5-1) edge node[left]{$f_1$} (e-3-1)
 (e-3-1) edge node[left]{$f_2$} (e-1-1)
 (e-1-1) edge node[above, sloped]{$(\pi_3,\sigma_3)$} (e-3-3)
 (e-3-1) edge node[below=-3pt]{$(\pi_2, \sigma_2)$} (e-3-3)
 (e-5-1) edge node[below, sloped]{$(\pi_1, \sigma_1)$} (e-3-3);  
 
\path[-implies, font=\scriptsize]
 (e-3-2) edge[double equal sign distance] node[right]{$\ \beta$} (e-1-1)
 (e-4-2) edge[double equal sign distance] node[right]{$\ \alpha$} (e-3-1);
\end{tikzpicture}
\end{center}
with the commutativity of the diagram given by morphisms $\alpha$ and $\beta$ represented by common \'etale neighbourhoods $(V_\alpha, \phi_\alpha, \psi_\alpha)$ between $X_1/S_1$ and $(S_1 \times_{S_2} X_2)/S_1$, and $(W_\beta, \phi_\beta, \psi_\beta)$ between $X_2/S_2$ and $(S_2 \times_{S_3} X_3)/S_2$. Then the commutativity of the large triangle in the diagram is given by the morphism $f_1^* \beta \circ \alpha$ represented by the pullback of the common \'etale neighbourhoods:
\begin{align*}
\left(V_\alpha \times_{\left(S_1 \times_{S_2} X_2\right)} \left(S_1 \times_{S_2} W_\beta\right), \phi_\alpha \circ \pr_{V_\alpha}, f_1^*\psi_\beta \circ \pr_{S_1 \times_{S_2} W_\beta} \right) \\
 = \left(V_\alpha \times_{X_2} W_\beta, \phi_\alpha \circ \pr_{V_\alpha}, \rho_\alpha \times \psi_\beta \right).
\end{align*}

We wish to show that 
\begin{align}\label{goal}
\theta(\SF)(f_2 \circ f_1, f_1^*\beta \circ \alpha) = \theta(\SF)(f_1, \alpha) \circ f_1^*\theta(\SF)(f_2, \beta).
\end{align}

From the definition of $f_1^* \beta \circ \alpha$, we have that
\begin{multline}\label{step 1}
\theta(\SF)(f_2 \circ f_1, f_1^*\beta \circ \alpha) \\
= \overline{\tau_{f_1^* \beta \circ \alpha}}^* \left( \SF(\phi_\alpha \circ \pr_{V_\alpha}, \id_{S_1}) \circ \SF(\pr_{X_3} \circ \psi_\beta \circ \pr_{W_\beta}, f_2 \circ f_1)^{-1} \right).
\end{multline}

Applying the compatibility of $\SF(\bullet)$ with composition of fibrewise morphisms, we can rewrite (\ref{step 1}) as follows:

\begin{multline}\label{step 2}
\overline{\tau_{f_1^* \beta \circ \alpha}}^* \left( (\pr_{V_\alpha}, \id_{S_1})_{X/S}^* \SF(\phi_\alpha, \id_{S_1}) \circ \SF(\pr_{V_\alpha}, \id_{S_1}) \circ \SF(\pr_{W_\beta}, f_1)^{-1} \right. \\ \left. \circ (\pr_{W_\beta}, f_1)_{X/S}^* \SF( \psi_\beta, \id_{S_2} )^{-1} \circ (\pr_{W_\beta}, f_1)_{X/S}^* (\psi_\beta, \id_{S_2} )_{X/S}^* \SF(\pr_{X_3}, f_2)^{-1} \right).
\end{multline}

Next we note the following equalities:
\begin{enumerate}[label=(\alph*)]
\item $ (\pr_{V_\alpha}, \id_{S_1})_{X/S} \circ \overline{\tau_{f_1^* \beta \circ \alpha}} = \overline{\tau_\alpha} $;
\item $ (\pr_{W_\beta}, f_1)_{X/S} \circ \overline{\tau_{f_1^* \beta \circ \alpha}} = \overline{\tau_\beta} \circ f_1 $;
\item $ (\psi_{\beta}, \id_{S_2})_{X/S} \circ \overline{\tau_\beta} = \overline{f_2^*\sigma_3}$;
\item $ (\psi_\alpha, \id_{S_1})_{X/S} \circ \overline{\tau_\alpha} = \overline{f_1^*\sigma_2}$.
\end{enumerate}

We will use these repeatedly in the remainder of this section. For example, using (a),(b),(c), we can rewrite (\ref{step 2}) in the following way:
\begin{multline}\label{step 3}
\overline{\tau_\alpha}^* \SF(\phi_\alpha, \id_{S_1}) \circ \overline{\tau_{f_1^* \beta \circ \alpha}}^* \SF(\pr_{V_\alpha}, \id_{S_1}) \\ \circ \overline{\tau_{f_1^* \beta \circ \alpha}}^* \SF(\pr_{W_\beta}, f_1)^{-1} \circ f_1^* \overline{\tau_\beta}^* \SF(\psi_\beta, \id_{S_2})^{-1} \circ f_1^* \overline{f_2^*\sigma_3}^* \SF(\pr_{X_3}, f_2)^{-1}.
\end{multline}

On the other hand, the right hand side of (\ref{goal}) is given by
\begin{multline}\label{step 4}
\theta(\SF)(f_1, \alpha) \circ f_1^* \theta(\SF)(f_2, \beta) \\
= \left( \overline{\tau_\alpha}^* \left( \SF(\phi_\alpha, \id_{S_1}) \circ \SF(\pr_{X_2} \circ \psi_\alpha, f_1)^{-1} \right) \right) \\ \circ f_1^* \overline{\tau_\beta}^* \left(\SF(\phi_\beta, \id_{S_2}) \circ \SF(\pr_{X_3} \circ \psi_\beta, f_2)^{-1} \right).
\end{multline}

We expand this using the compatibility of $\SF(\bullet)$ with composition, and use the equalities (c) and (d) to obtain
\begin{multline}\label{step 5}
\overline{\tau_\alpha}^* \SF(\phi_\alpha, \id_{S_1}) \circ \overline{\tau_\alpha}^* \SF(\psi_\alpha, \id_{S_1})^{-1} \circ \overline{f_1^* \sigma_2}^* \SF(\pr_{X_2}, f_1)^{-1} \\ \circ f_1^* \overline{\tau_\beta}^* \SF(\phi_\beta, \id_{S_2}) \circ f_1^* \overline{\tau_\beta}^* \SF(\psi_\beta, \id_{S_2})^{-1} \circ f_1^* \overline{f_2^* \sigma_3}^* \SF(\pr_{X_3}, f_2)^{-1}.
\end{multline}

Comparing (\ref{step 3}) and (\ref{step 5}) we see that to prove the desired equality (\ref{goal}), it suffices to show that
\begin{multline*}
\overline{\tau_\alpha}^* \SF(\psi_\alpha, \id_{S_1})^{-1} \circ \overline{f_1^* \sigma_2}^* \SF(\pr_{X_2}, f_1)^{-1} \circ f_1^* \overline{\tau_\beta}^* \SF(\phi_\beta, \id_{S_2}) \\ = \overline{\tau_{f_1^* \beta \circ \alpha}}^* \SF(\pr_{V_\alpha}, \id_{S_1}) \circ \overline{\tau_{f_1^* \beta \circ \alpha}}^* \SF(\pr_{W_\beta}, f_1)^{-1},
\end{multline*}
or equivalently that
\begin{multline}\label{goal 2}
\overline{f_1^* \sigma_2}^* \SF(\pr_{X_2}, f_1) \circ \overline{\tau_\alpha}^* \SF(\psi_\alpha, \id_{S_1}) \circ \overline{\tau_{f_1^* \beta \circ \alpha}}^* \SF(\pr_{V_\alpha}, \id_{S_1}) \\ = f_1^* \overline{\tau_\beta}^* \SF(\phi_\beta, \id_{S_2}) \circ \overline{\tau_{f_1^* \beta \circ \alpha}}^* \SF(\pr_{W_\beta}, f_1).
\end{multline}

Using (b), the right hand side of (\ref{goal 2}) becomes 
\begin{multline}\label{step 6}
\overline{\tau_{f_1^* \beta \circ \alpha}}^* \left( (\pr_{W_\beta}, f_1)_{X/S}^* \SF(\phi_\beta, \id_{S_1}) \circ \SF(\pr_{W_\beta}, f_1) \right) \\ = \overline{\tau_{f_1^* \beta \circ \alpha}}^* \SF(\phi_\beta \circ \pr_{W_\beta}, f_1).
\end{multline}
Meanwhile, using (d) and (a), the left hand side of (\ref{goal 2}) can be written as
\begin{multline}\label{step 7}
\overline{\tau_{f_1^* \beta \circ \alpha}}^* \left( (\psi_\alpha, id_{S_1})_{X/S}^* (\pr_{V_\alpha}, \id_{S_1})^* \SF(\pr_{X_2},f_1) \right. \\ \left. \circ (\pr_{V_\alpha}, \id_{S_1})^* \SF(\psi_\alpha, \id_{S_1}) \circ \SF(\pr_{V_\alpha}, \id_{S_1}) \right).
\end{multline}

Using once again the compatibility of $\SF(\bullet)$ with composition, we see that this is equal to 
\begin{align}\label{step 8}
\overline{\tau_{f_1^* \beta \circ \alpha}}^* \SF( \pr_{X_2} \circ \psi_\alpha \circ \pr_{V_\alpha}, f_1).
\end{align}

Since $\pr_{X_2} \circ \psi_\alpha \circ \pr_{V_\alpha} = \phi_\beta \circ \pr_{W_\beta}$, the expression in (\ref{step 6}) is equal to the expression in (\ref{step 8}) and so the proof is complete. 

\section{Appendix: The main diagram}
\label{appendix: main diagram}
We include on the next page an extra copy of the main diagram. The reader may wish to cut it out for ease of reference.
\begin{figure*}
\centering
\includegraphics[angle=90, height=7.5in]{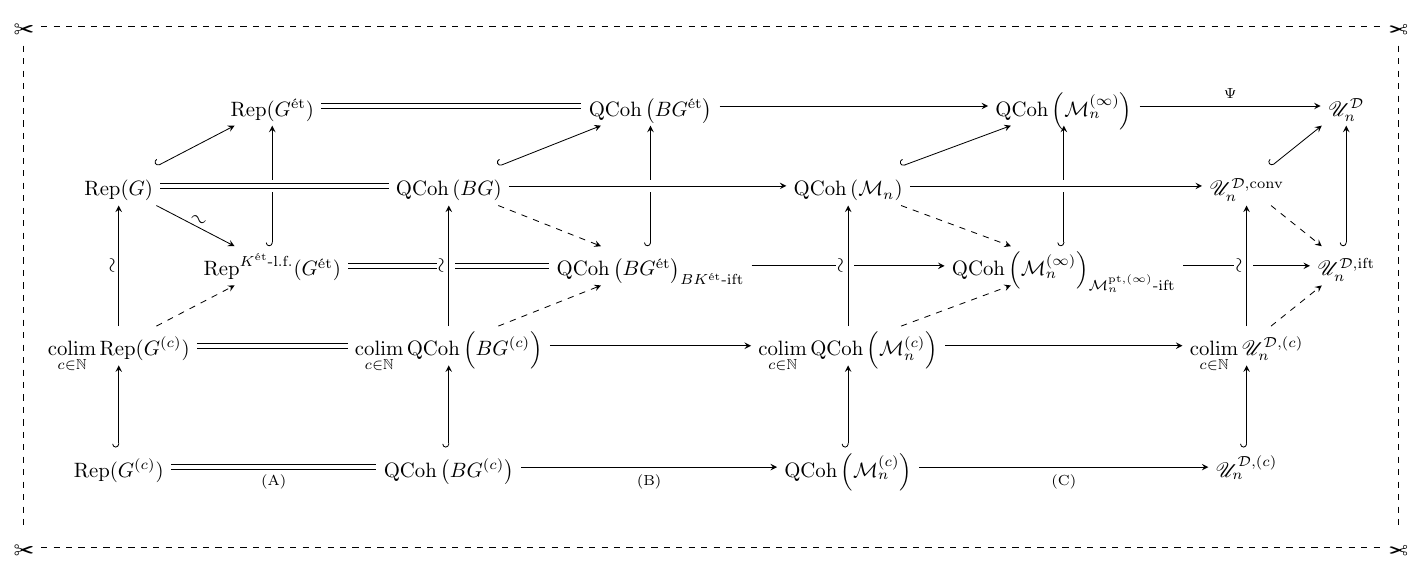}
\end{figure*}

\bibliographystyle{spmpsci}
\bibliography{bibliography-EJM}

\end{document}